\documentclass[abstract=no]{amsart} 
\usepackage[margin=3.8cm]{geometry}

\usepackage{amsmath}
\usepackage{bbm}
\usepackage{latexsym}
\usepackage[english]{babel}
\usepackage{amsfonts}
\usepackage[T1]{fontenc}
\usepackage[utf8]{inputenc}
\usepackage{amsthm}
\usepackage{amssymb}
\usepackage{dsfont}
\usepackage{version}
\usepackage{caption}
\usepackage[backend=bibtex,style=alphabetic,hyperref=true,maxnames=5]{biblatex}
\usepackage[psamsfonts]{eucal}
\usepackage{stmaryrd}
\usepackage{url}
\usepackage{hyperref}
\usepackage{color}
\usepackage{blindtext}
\usepackage{tikz}
\usetikzlibrary{arrows,matrix,positioning}
\tikzstyle{dmatrix}=[matrix of math nodes,row sep=2.5em, column sep=2.5em,
text height=1.5ex, text depth=0.25ex] 
\usepackage{mathtools}
\usepackage{microtype}
\usepackage{graphicx}
\usepackage{mathabx}
\usepackage{lipsum}
\usepackage{float}
\usepackage{csquotes}
\usepackage{empheq}
\usepackage{enumitem}
\usepackage{mathrsfs}
\usepackage{mathtools}
\DeclareSymbolFont{epsilon}{OML}{ntxmi}{m}{it}
\DeclareMathSymbol{\epsilon}{\mathord}{epsilon}{"0F}

\theoremstyle{plain}
\newtheorem{theorem}{Theorem}[section]
\newtheorem{lemma}[theorem]{Lemma}
\newtheorem{prop}[theorem]{Proposition}
\newtheorem{cor}[theorem]{Corollary}

\newtheorem{prop/Def}[theorem]{Propsition/Definition}
\newtheorem{theorem/Def}[theorem]{Theorem/Definition}
\theoremstyle{definition}

\newtheorem{Def}[theorem]{Definition}
\newtheorem{rem}[theorem]{Remark}
\newtheorem{exa}[theorem]{Example}

\def \Q {{\mathbb Q}}
\def \R {{\mathbb R}}

\def \Z {{\mathbb Z}}

\def \D {{\operatorname{D}}}

\def \S {{\mathbb S}}

\def \X {{ \mathcal{X}}}

\def \T {{\mathbb T}}
\def \div {{\operatorname{div}}}

\def \vol {{ \operatorname{vol}}}

\def \codim {{\operatorname{codim}}}
\newcommand{\on}[1]{\operatorname{#1}}
\newcommand{\ca}[1]{{\mathcal{#1}}}
\def \div {{ \operatorname{div}}}

\def \Spec {{ \operatorname{Spec}}}

\def \rec {{\operatorname{rec}}}

\def \CH {{\operatorname{CH}}}

\def \closed {{\operatorname{closed}}}

\def \closed {{\operatorname{closed}}}

\renewbibmacro{in:}{}

\usepackage{tgpagella}
\usepackage{tgheros}
\usepackage{eulervm}
\usepackage{tikz,tikz-cd}
\usepackage{etoolbox}
\usepackage{subfigure}
\usetikzlibrary{arrows,matrix,positioning}
\tikzstyle{dmatrix}=[matrix of math nodes,row sep=2.5em, column sep=2.5em,
text height=1.5ex, text depth=0.25ex] 
\usepackage{color}

\numberwithin{equation}{section}

\bibliography{bibliography_update}

\title{Equivariant non-archimedean Arakelov theory of toric varieties}
\author{Ana Mar\'ia Botero}
\date{}
\thanks{The author was supported by the by the Deutsche Forschungsgemeinschaft (DFG, German Research Foundation) – Project-ID 491392403 – TRR 358 and by collaborative research 
center SFB 1085. }
\subjclass{14M25; 14G40; 14G20; 14L30; 14C17; 14F43; 52A41}
\keywords{Arakelov theory, toric varieties, toric schemes over a DVR, equivariant Chow groups, polyhedral complexes, b-divisors}
\begin{document}

\maketitle

\begin{abstract}
We develop an equivariant version of the non-archimedean Arakelov theory of \cite{BGS} in the case of toric varieties.  We define the equivariant analogues of the non-archimedean differential forms and currents appearing in \emph{loc.~cit.} and relate them to piecewise polynomial functions on the polyhedral complexes defining the toric models. In particular, we give combinatorial characterizations of the Green currents associated to equivariant cycles and combinatorial descriptions of the arithmetic Chow groups.
\end{abstract}

\tableofcontents

\section{Introduction}

Let $X$ be a smooth projective variety over a number field $K$. Classical Arakelov theory considers algebraic cycles and vector bundles on a fixed regular model $\X$ of $X$ over the ring of integers of $K$, together with analytic data on the complex points of $X$ (\cite{arakelov1}, \cite{arakelov2}, \cite{Fa}, \cite{GS1}, \cite{GS2}). Adding this data can be thought of as a \emph{compactification} of $\X$ and leads to an arithmetic intersection theory on $X$. Note that this is in no way a compactification in the topological sense, only in the sense that is nice enough in order to have a well-defined intersection theory. 

Now, it is an old dream to handle archimedean and non-archimedean places in
a similar way. This means that we are looking for a description in terms of forms and
currents for the contributions of the non-archimedean places. Such a non-archimedean Arakelov theory at finite places was developed
by Bloch–Gillet–Soul\'e in \cite{BGS}, based on the Chow groups of \emph{all} regular proper models over the valuation ring. A drawback of this theory is that it relies strongly on the conjectured existence of resolution of singularities for models in mixed characteristics. Nevertheless, under this hypothesis, it establishes a complete dictionary between objects in differential geometry and intersection theory. For example, for any integer $k \geq 0$, the authors define non-archimedean analogues of the four groups: closed $(k,k)$-forms, $(k,k)$-forms modulo the image of $\partial$ and $\bar{\partial}$, closed $(k,k)$-currents and $(k,k)$-currents modulo the image of $\partial$ and $\bar{\partial}$. The non-archimedean counterparts are appropriate direct and inverse limits of Chow cycle groups and operational Chow cohomology groups of special fibers of the regular models.  They also define a non-archimedean analogue of the $dd^c$ operator in complex differential geometry  and show that these groups fit into natural exact sequences analogous to the ones occurring in the archimedean case.

The present article studies an equivariant version of this theory in the case of toric varieties over $K$, for which resolution of singularities is always satisfied.  There are two main objectives. The first one is to give a combinatorial characterization of the the objects involved. This provides convex geometrical and combinatorial tools to study non-archimedean Arakelov theory on toric varieties over $K$. These tools are general and explicit enough and this leads us to our second objective: to provide a large testing class where dealing with the non-archimedean analogue of singular metrics seems possible. It thus serves as a testing ground if one wishes to extend results in Arakelov geometry to consider singular metrics. 

Let us be a bit more precise. 
As we will see, this theory suggests that the non-archimedean analogue of a (smooth) toric metric on a toric vector bundle is the choice of a toric model and a piecewise polynomial function on the polyhedral complex associated to the toric model (see discussion after Theorem \ref{th:intro5}). This generalizes \cite{BPS}, where toric \emph{model} metrics on toric line bundles correspond to piecewise affine functions on the polyhedral complex.  

This gives a combinatorial interpretation of the first equivariant Chern form of such a metrized toric vector bundle. 

Potentially, the above analogy can be generalized to include singular metrics, hence allowing the study of equivariant Chern \emph{currents} on toric varieties over $K$ by combinatorial methods. This is explained in more detail below. 
 
The fact that one can mimic forms and currents on archimedean places by taking direct and inverse limits of Chow groups of regular models goes in line with recent attempts of several authors to view the first Chern class of a line bundle on a smooth projective variety $X$, endowed with a semi-positive singular metric, as a so called \emph{Weil b-divisor} encoding the singularities of the metric. Recall that a Weil b-divisor on $X$ is a tower of divisors indexed by modifications of $X$, compatible under pushforward. This is done by taking into account the so-called \emph{Lelong} numbers of the local potentials defining the singular semi-positive metric (see \cite{chern-weil} for details). If the psh metric is good enough, for example, if it has algebraic singularities, then this tower stabilizes and the resulting b-divisor lives in the direct limit. This is then an example of a \emph{Cartier} b-divisor.  Such an attempt was implemented in the archimedean case in \cite{chern-weil} and \cite{siegel-jacobi} leading to Chern-Weil type formulae. In the non-archimedean case, Boucksom, Favre and Jonsson gave an approach to non-archimedean plurisubharmonic functions via skeletons in the case of residue characteristic $0$ (see \cite{BFJ} and \cite{BFJ:valuations}).  

Another approach towards handling archimedean and non-archimedean places in the same way is an analytic description of non-archimedean Arakelov theory, based on a theory of differential forms and currents on Berkovich analytic spaces. It was initiated by Thuillier in his thesis, where he developed pluripotential theory on the Berkovich analytic space associated to a curve over a non-archimedean field. This was generalized to higher dimensions by Chambert--Loir and Ducros in \cite{CLD}, where they developed a theory of real differential forms and currents on Berkovich analytic spaces. The theory gives rise to classical formula such as the Poincaré--Lelong formula and it has the advantage that it is unconditional and does not relay on any conjecture on resolution of singularities. The theory of Chambert--Loir and Ducros has been generalized to allow a larger class of differential forms (see e.\,g.\,\cite{GK} and \cite{Mih}). 

In fact, the approach using Berkovich spaces can be thought of as a refined approach of Bloch--Gillet--Soulé, since Berkovich spaces are topological inverse limits of its skeleta, and the skeleta can be thought of as the (dual complexes) of the special fibres of SNC models.

The main reason for working with Bloch--Gillet--Soulé's approach in the present article instead of Chambert--Loir and Ducros, is that our main purpose is to exploit the torus action and to develop an \emph{equivariant} version of non-archimedean Arakelov theory. For this, we need either a theory of equivariant Chow groups, or a theory of equivariant differential forms on Berkovich spaces. Since the former has been developed by Edidin and Graham in \cite{EG}, and, at least from the author's knowledge, the latter has not, we chose to work with the existing equivariant theory. 
It would be interesting to develop an equivariant theory of differential forms on Berkovich spaces, and to compare it with the current theory in the case of toric varieties. 

Before we state our main results, we recall the main definitions and statements from \cite{BGS} in more detail. 

The setting is as follows. Let $R$ be an excellent discrete valuation ring with quotient field $K$ and residue field $\kappa$. Let $X$ be a smooth projective variety of dimension $n$ over $K$. Set $S = \on{Spec}(R) = \{\eta, s\}$, where $\eta$ and $s$ denote the generic and the special point, respectively. A \emph{model} of $X$ over $S$ is a proper and flat $S$-scheme $\X$ together with an isomorphism of the generic fiber $\X_{\eta} \simeq X$. Assume resolution of singularities of models of $X$ over $S$ (see \cite[pg.~432]{BGS} for the precise definition).  In particular this means that any model is dominated by an snc model (i.e. a model whose reduced special fiber is a simple normal crossings divisor). Hence the set $R(X)$ of snc models of $X$ is endowed with a directed set structure. 

For any $S$-scheme $\X$ one considers Chow homology groups $\CH_k(\X)$ and Fulton's operational Chow cohomology groups $\CH^k(\X)$ (\cite[Chapters 17, 20]{fulint}). Note that in the present article, as well as in \cite{BGS} and in \cite{chow-toric}, an \emph{absolute} notion of dimension is considered (see Definition \ref{def:abs-dim} and the remark after), which is independent of $S$. This differs from the \emph{relative} notion in \cite{fulint} by $+1$. The operational Chow cohomology group $\CH^*(\X)$ carries a graded ring structure, and taking cap products, $\CH_*(\X)$ becomes a graded $\CH^*(\X)$-module. 

For models in $R(X)$, the authors assume that Poincaré duality holds, i.e.~that capping with the fundamental class 
\[
\CH^{k}(\X) \longrightarrow \CH_{n+1-k}(\X), \; c \longmapsto c \cap [\X]
\]
is an isomorphism.  Apparently, this holds if one considers Chow groups tensored with $\Q$ by the arguments in Kleiman in \cite{KT} (see comment in \cite[pg.\,433]{BGS}).

We will show explicitly in Theorem \ref{th:dual2} that Poincaré duality holds in our equivariant setting for toric models over $S$.

Now, if $\X$ is an snc model of $X$ then it turns out that the Chow groups of the special fiber, seen as a (not necessarily irreducible nor reduced) scheme over $\kappa$ can be computed in terms of Chow groups of smooth varieties. Explicitly, write $\X_s^{[i]}$ for the disjoint union of $i$-fold intersections of irreducible components of $\X_s$. The Chow homology groups $\CH_k\left(\X_s\right)$ fit into an exact sequence 
\[
\CH_k\left(\X_s^{[2]}\right) \longrightarrow \CH_k\left(\X_s^{[1]}\right) \longrightarrow \CH_k\left(\X_s\right) \longrightarrow 0.
\]
Similarly, the operational Chow cohomology groups $\CH^k\left(\X_s\right)$ fit into an exact sequence 
\[
0 \longrightarrow \CH^k\left(\X_s\right) \longrightarrow \CH^k\left(\X_s^{[1]}\right) \longrightarrow \CH^k\left(\X_s^{[2]}\right).
\]
Both groups are co-and contravariant for maps of special fibers induced by maps of models. Taking direct and inverse limits, one obtains the four groups $\widetilde{A}^{k,k}(X), A^{k,k}_{\closed}(X), \widetilde{D}^{k,k}(X)$ and $D^{k,k}_{\closed}(X)$ of closed $(k,k)$-forms, $(k,k)$-forms modulo the image of $\partial$ and $\bar{\partial}$, closed $(k,k)$-currents and $(k,k)$-currents modulo the image of $\partial$ and $\bar{\partial}$, respectively.

For an snc model $\X$ in $R(X)$ let $\iota\colon \X_s \hookrightarrow \X$ denote the inclusion of the special fiber. Consider the map 
\[
\iota^*\iota_* \colon \CH_{n-k}(\X_s) \longrightarrow \CH^{k+1}(\X_s)
\]
consisting of pushing cycles forward from the special fiber to the model, then using Poincaré duality, and then pulling back. 

The first main result in \cite{BGS} states that the kernels of $\iota^*\iota_*$ for all choices of models $\X$ are isomorphic. And similar for the cokernels. We will see that in our toric equivariant setting, the existence of a toric canonical model, which is smooth whenever the toric variety we start with is smooth, implies that we can explicitly describe $\on{ker}(\iota^*\iota_*)$.  Moreover, in this case, it is isomorphic to $\on{coker}(\iota^*\iota_*)$.

Now, the map $\iota^*\iota_*$ is compatible with maps of special fibers induced by maps of models, hence one can define limit versions of this map:
\[
     dd^c \colon \widetilde{A}^{k,k}(X) \longrightarrow A_{\closed}^{k+1,k+1}(X)
   \quad \text{and} \quad
      dd^c \colon \widetilde{D}^{k,k}(X) \longrightarrow D_{\closed}^{k+1,k+1}(X). 
\]
which are shown to satisfy some regularity properties. 
These $dd^c$ maps are non-archimedean analogues of the $dd^c$ maps appearing in complex differential geometry. 

Let $Z \subseteq X$ be an algebraic cycle of codimension $k$. A \emph{Green current} for $Z$ is an element $g_Z \in \widetilde{D}^{k,k}(X)$ such that 
\[
dd^c(g_Z) + \delta_Z \in A_{\closed}^{k,k}(X).
\]
Here, $\delta_Z \in D^{k,k}_{\closed}(X)$ is the projective system consisting of the Zariski closures of $Z$ in the snc models of $X$. 

Finally, the $k$'th arithmetic Chow group $\widehat{\CH}^k(X)$ is defined mimicking the classical arithmetic Chow groups: generators are pairs of the form $(Z, g_Z)$, where $Z$ is an algebraic cycle of codimension $k$ on $X$ and $g_Z$ is a Green current for $Z$, and relations are given by rational equivalence. 

The second main result in \cite{BGS} states that there is a canonical isomorphism 
\begin{equation*}
\widehat{\CH}^*(X) = \varinjlim_{\X}\CH^*(\X).
\end{equation*}
The graded algebra structure on $\CH^*(\X)$ for any model $\X$ induces a graded algebra structure on the direct limit. Hence, via the above isomorphism, $\widehat{\CH}^*(X)$ inherits a graded algebra structure. 

One can interpret this result as saying that the non-archimedean analogue of choosing a (smooth) hermitian metric on a vector bundle $E$ on $X$ is the choice of a snc model $\X$ of $X$ together with an extension of $E$ to a vector bundle $\mathfrak{E}$ on $\X$ (see Remark \ref{rem:direct-toric}).

Finally, in \cite{GS-direct} the authors consider an extended version of the arithmetic Chow group $\widecheck{\CH}^k(X)\supseteq \widehat{\CH}^k(X)$. Generators are pairs of the form $(Z,g)$, where $Z$ is an algebraic cycle of codimension $k$ in $X$, and $g$ is any current in $\widetilde{D}^{k,k}(X)$ (not necessarily a Green current). Relations are given again by rational equivalence. It is shown in \emph{loc.~cit.~} that there is a canonical isomorphism 
\begin{equation*}
\widecheck{\CH}^*(X) \simeq \varprojlim_{\X}\CH^*(\X).
\end{equation*}

Note that since the pushforward is a map of groups, not of rings, there is a priori no ring structure on $\widecheck{\CH}^*(X)$.  However, one can think of elements in $\widecheck{\CH}^*(X)$ as some kind of arithmetic $b$-cycles and, it seems possible to define an intersection pairing if one is willing to restrict to some kind of ``positive'' elements. We will pursue this idea in the future. 

We remark that one can interpret the isomorphism above as saying that the non-archimedean analogue of choosing a (singular) hermitian metric on a vector bundle $E$ on $X$ is the choice of a tower of extensions $\left\{\mathfrak{E}_{\X}\right\}_{\X}$, where $\mathfrak{E}_{\X}$ is a vector bundle on $\X$, extending $E$, whose Chern classes, defined in the Chow groups of the special fibers, are compatible under pushforward (see Remark \ref{rem:proj}). 

\subsection{Statement of the main results.}
Let $R$,\,$K$,\,$\kappa$ and $S$ as before. Assume that $\kappa = \overline{\kappa}$ is algebraically closed. Let $\T_S$ be a split torus over $S$. We denote by $\T_K$ and by $\T{\kappa}$ the base
change to $\Spec(K)$ and to $\Spec(\kappa)$, respectively. Let $N$ be the lattice of cocharacters of $\T_K \simeq (K^*)^n$ and set $N_{\R} = N \otimes_{\Z} \R$. Let $X=X_{\Sigma}$ be a smooth projective toric variety of dimension $n$ over $K$ with corresponding (regular, complete) fan $\Sigma$ in $N_{\R}$. 

In this case, one can consider so called \emph{toric} models of $X$ over $S$ (see Definition~\ref{def:toric-model}). As it turns out, regular proper toric models of $X$ over $S$ are in bijection with regular, complete, strongly convex, rational polyhedral complexes $\Pi$ in $N_{\R}$ having recession fan equal to $\Sigma$ (we refer to Section \ref{sec:toric-schemes} for details). From now on, we will abbreviate \emph{strongly convex, rational} with "SCR". Given such a toric model $\X_{\Pi}$ corresponding to a SCR polyhedral complex $\Pi$, the special fiber $\X_{\Pi,s}$ is then a union of toric varieties glued in a way which is dictated by the combinatorics of the complex $\Pi$.  Moreover, consider $c(\Pi)$ in $N_{\R} \oplus \R_{\geq 0}$, the \emph{cone of} $\Pi$. This defines a rational polyhedral fan in $N_{\R} \oplus \R_{\geq 0}$, whose cones correspond (in order reversing correspondence) to torus orbits. We denote by $V(\sigma)$ the Zariski closure of the torus orbit corresponding to a cone $\sigma \in c(\Pi)$. If $\sigma \subseteq N_{\R} \times \{0\}$, then $V(\sigma)$ is a horizontal cycle, i.e. the structure morphism $V(\sigma) \to \S$ is dominant. Otherwise, it is vertical, i.e. it is contained in the special fibre.  

As we will see, an equivariant map of toric models corresponds to a map between SCR polyhedral complexes. The fact that given such a complex, one can always subdivide it into a regular one, implies that we always have resolution of singularities of toric models. In particular, the set $R(\Sigma)$ consisting of regular proper toric models of $X$ is directed. 

We shall consider the category of $\T_S$-schemes over $S$,  i.e. integral separated $S$-schemes of finite type $\X$ with a torus $\T_K \hookrightarrow \X_{\eta}$ and an $S$-action of $\T_S$ over $\X$ that extends the action of $\T_K$ on itself by translations. For a $\T_S$-scheme $\X$, we consider the equivariant Chow homology groups $\CH_*^{\T_S}(\X)$ and the equivariant operational Chow cohomology groups $\CH^*_{\T_S}(\X)$. These groups are introduced and studied in \cite[Section 6.2]{EG} for more general group actions. For the case of torus actions over a field, they have been studied in \cite{Brion-equi} (see also \cite{Gonz}). We recall the definitions in Section \ref{sec:eq-chow}.  As in the non-equivariant setting, the group $\CH_{\T_S}^*(\X)$ carries a graded algebra structure. 

Let $\widetilde{S}$ be the equivariant Chow homology group of our base scheme $S$. If $\X = \X_{\Pi}$ is a toric scheme corresponding to a SCR complex $\Pi$, then we show in Proposition~\ref{prop:chow-hom} that $\CH_*^{\T_S}(\X)$ is an $\widetilde{S}$-module, defined by generators $\left[V(\sigma)\right]$ for $\sigma \in c(\Pi)$. 
We also have a combinatorial description of the cohomology groups. Let $PP^*\left(c(\Pi)\right)$ denote the ring of piecewise polynomial functions on $c(\Pi)$ (see Section \ref{sec:toric} for the definition of this ring). It has the structure of an $\widetilde{S}$-module. The following is Theorem~\ref{th:equi-pol} in the text.
\begin{theorem}\label{th:intro1}
There is an $\widetilde{S}$-module isomorphism 
\begin{eqnarray}\label{eq:pp}
\CH^*_{\T_S}(\X) \simeq PP^*\left(c(\Pi)\right).
\end{eqnarray}
\end{theorem}
The analogous  result for toric varieties over a field is given in \cite[Theorem 1]{Payne-equi} (see also \cite[Proposition 2.2]{BR}). 

As in the non-equivariant case, there is a natural Poincaré duality map 
\[
\CH_{\T_S}^k(\X) \longrightarrow \CH^{\T_S}_{n+1-k}(\X)
\]
given by capping with the fundamental class. As we mentioned before, we will show in Theorem \ref{th:dual2} that this map is an isomorphism. Moreover, if $\sigma \in c(\Pi)$ is a cone of dimension $k$, then under \eqref{eq:pp}, the equivariant class $[V(\sigma)] \in \CH_{n+1-k}(\X)$ is identified with the piecewise polynomial function $\rho_{\sigma}$ of degree $k$ from  Definition~\ref{def:generators}.

Assume that the special fiber $\X_s$ is reduced. As in the non-equivariant case, we show that the equivariant Chow groups of $\X_s$ can be computed in terms of the equivariant Chow groups of smooth irreducible toric varieties. Explicitly, consider $\Pi$ the SCR polyhedral complex associated to $\X$. In Proposition \ref{prop:exact-seq} we show that we have the following exact sequences
\[
\bigoplus_{\stackrel{\gamma \in \Pi(1)}{\gamma \text{ bdd}}}\CH_{\T_{\kappa}}^k\left(V(\gamma)\right) \xrightarrow{\gamma} \bigoplus_{v \in \Pi(0)}\CH_{\T_{\kappa}}^{k+1}\left(V(v)\right) \to \CH_k^{\T_{\kappa}}\left(\X_s\right) \to 0
\]
and 
\[
0 \to \CH^k_{\T_{\kappa}}\left(\X_s\right) \to \bigoplus_{v \in \Pi(0)}\CH_{\T_{\kappa}}^k\left(V(v)\right) \xrightarrow{\rho}\bigoplus_{\stackrel{\gamma \in \Pi(1)}{\gamma \text{ bdd}}}\CH_{\T_{\kappa}}^k\left(V(\gamma)\right),
\]
where for any natural number $\ell$, $\Pi(\ell)$ denotes the set of $\ell$-dimensional polyhedra in $\Pi$.
In Section \ref{sec:special} we give explicit descriptions of the maps $\gamma$ and $\rho$ in terms of piecewise polynomial functions. This leads to a characterization of $\CH^k_{\T_{\kappa}}\left(\X_s\right)$ in terms of so-called \emph{affine piecewise polynomial functions on $\Pi$} (see Definition \ref{def:affine-pp}). The following is Theorem~\ref{th:cohomology-special} in the text. 
\begin{theorem}\label{th:intro2}
The equivariant operational Chow ring of the special fiber $\CH_{\T_{\kappa}}^*(\X_s)$ can be identified with the graded ring of affine piecewise polynomial functions $PP^*(\Pi)$.
\end{theorem}
This theorem has been applied to the study of equivariant Chern classes of toric vector bundles over a DVR and connections to Bruhat--Tits buildings (see Remark \ref{rem:bruhat}).

Consider the smooth, complete toric variety $X = X_{\Sigma}$ over $K$. Both $\CH_k^{\T_{\kappa}}\left(\X_s\right)$ and $\CH^k_{\T_{\kappa}}\left(\X_s\right) $ are co-and contravariant for maps of special fibers induced by maps of toric models. Taking direct and inverse limits, one obtains the four groups $\widetilde{A}_{\T_K}^{k,k}(X), {A}^{k,k}_{\closed, \T_K}(X), \widetilde{D}_{\T_K}^{k,k}(X)$ and ${D^{k,k}}_{\closed, \T_K}(X)$ of closed equivariant $(k,k)$-forms, equivariant $(k,k)$-forms modulo the image of $\partial$ and $\bar{\partial}$, closed equivariant $(k,k)$-currents and equivariant $(k,k)$-currents modulo the image of $\partial$ and $\bar{\partial}$, respectively. In Proposition \ref{prop:com-forms-currents} we give combinatorial descriptions of these groups. In particular, we obtain the following description of the closed forms and currents.
\begin{prop}\label{prop:intro3}
Let notations be as above. 
\begin{itemize}
\item We have
\[
A^{*,*}_{\closed,\T_K}(X)\simeq \left\{ f \colon N_{\R} \to \R \; | \; \exists \Pi \in R(\Sigma) \text{ with } f \in PP(\Pi) \right\}.
\]
\item An element in $D_{\closed,\T_K}^{k,k}(X)$ is a tuple $\left(f_{\Pi'}\right)_{\Pi' \in R(\Sigma)}$, where $f_{\Pi'} \in PP(\Pi')$, compatible under the pushforward map.

\end{itemize}
\end{prop}
In Example \ref{ex:b-div} we make a connection with so called toric b-divisors in the case $k=1$.

Now, also in this equivariant setting, given a model $\X = \X_{\Pi}$ of $X$ we have a map 

\[
\iota^*\iota_* \colon \CH^{\T_S}_{n-k}(\X_s) \longrightarrow \CH_{\T_S}^{k+1}(\X_s)
\]
consisting of pushing cycles forward from the special fiber to the model, then using Poincaré duality, and then pulling back.  We show in Section \ref{sec:iota} that one can compute this map via the composition $-\gamma \rho$, where $\gamma$ and $\rho$ are the maps given above. We provide an explicit combinatorial description of this composition in terms of piecewise polynomial functions. 

As we already mentioned, we show that $\on{ker}(\iota^*\iota_*)$ ( resp.~$\on{coker}(\iota^*\iota_*)$) is independent on the model. Hence, we may consider the canonical model $\X_{\Sigma}$, which is smooth, and we obtain the following corollary (which follows from Proposition~\ref{prop:ker=koker} and Corollary~\ref{cor:ker} in the text).
\begin{cor}\label{cor:intro4}
We have
\[
\on{ker}\left(\iota^*\iota_*\right) \simeq \on{coker}\left(\iota^*\iota_*\right) \simeq CH^{k}_{\T_K}(X_{\Sigma}).
\]
\end{cor}

The map $\iota^*\iota_*$ is compatible with maps of special fibers induced by maps of toric models, hence one has the limit versions
\[
     dd^c \colon \widetilde{A}_{\T_K}^{k,k}(X) \longrightarrow A_{\closed, \T_K}^{k+1,k+1}(X)
   \quad \text{ and }\quad
      dd^c \colon \widetilde{D}_{\T_K}^{k,k}(X) \longrightarrow D_{\closed, \T_K}^{k+1,k+1}(X).
\]
We show that these satisfy some regularity properties (see Proposition~\ref{prop:regularity}). 

Now, given an equivariant cycle $Z \in \CH_*{\T_K}(X)$ we define an \emph{ equivariant Green current for} $Z$ to be an element $g_Z \in \widetilde{D}_{\T_K}^{k,k}(X)$ such that 
\[
dd^c(g_Z) + \delta_Z \in A_{\closed,\T_K}^{k,k}(X).
\]
Here, $\delta_Z \in D^{k,k}_{\closed}(X)$ is equivariant $\delta$-current defined in Definition \ref{def:delta}.

We defined the \emph{$k$'th equivariant arithmetic Chow group} $\widehat{\CH}_{\T_K}^k(X)$ as the group generated by pairs of the form $(Z, g_Z)$, where $Z$ is an equivariant cycle of degree $n-k$ on $X$ and $g_Z$ is a Green current for $Z$, modulo ``invariant'' rational equivalence (see Section \ref{sec:arith-chow1} for details). 

We have the following combinatorial description of $\widehat{\CH}_{\T_K}^k(X)$ (see~Theorem~\ref{th:direct} in the text).
\begin{theorem}\label{th:intro5}
There are canonical isomorphisms
\[
\widehat{\CH}_{\T_K}^k(X) \simeq \varinjlim_{\X\in R(\Sigma)}\CH_{\T_S}^k(\X) \simeq PP_{\Sigma}^k\left(N_{\R} \oplus \R_{\geq 0}\right),
\]
where
\[
PP_{\Sigma}^k\left(N_{\R} \oplus \R_{\geq 0}\right) \coloneqq  \left\{ f \colon N_{\R} \oplus \R_{\geq 0} \to \R \; | \; \exists \Pi \in R(\Sigma) \text{ with } f \in PP(c(\Pi)) \right\}. 
\]
\end{theorem}
Hence, $\widehat{\CH}_{\T_K}^*(X)$ inherits a graded algebra structure given by multiplication of piecewise polynomial functions (see Corollary~\ref{cor:direct}). 

As was mentioned at the beginning of the introduction, the above theorem suggests that the non-archimedean analogue of a (smooth) toric metric on a toric vector bundle is the choice of a toric model and a piecewise polynomial function on the (cone over the) polyhedral complex associated to the toric model. In particular, for $k=1$ Theorem \ref{th:intro5} is saying that the non-arquimedean analogue of choosing a (smooth) hermitian toric metric on a toric line bundle associated to a virtual support function $\psi$, is the choice of a piecewise affine function on $N_{\R}$ whose recession function agrees $\psi$( see Remark~\ref{rem:direct-toric}).  These correspond to the so-called \emph{toric model metrics} from \cite{BPS}.
Theorem \ref{th:intro6} below extends this analogy to include singular metrics. 

Similar to the non-equivariant case, we  define an extended equivariant arithmetic Chow group $\widecheck{\CH}_{\T_K}^k(X)$. In Theorem~\ref{th:inverse} we show:
\begin{theorem}\label{th:intro6}
There is a canonical isomorphism 
\begin{equation}\label{eq:sing}
\widecheck{\CH}_{\T_K}^*(X) \simeq \varprojlim_{\X}\CH_{\T_K}^*(\X).
\end{equation}
\end{theorem}
The group $\widecheck{\CH}_{\T_K}^*(X)$ has the structure of an $\widehat{\CH}_{\T_K}^*(X)$-module. In Remark~\ref{rem:proj} comment on the interpretation of Equation \eqref{eq:sing}.

Finally, we expect that assuming some positivity properties we will be able to define an intersection product of (positive) elements in this extended equivariant arithmetic Chow group. This is then the analogue of developing an intersection theory of (positive) Chern currents on $X$. We will continue this line of work in the future.

\subsection{Outline of the paper}
In Section~\ref{sec:toric-schemes} we briefly recall the combinatorial characterization of toric schemes over a discrete valuation ring (DVR) in terms of strongly convex convex rational (SCR) polyhedral complexes. In Section~\ref{sec:eq-chow} we consider the equivariant Chow homology group and the equivariant operational Chow cohomology ring of a $\T_S$-scheme $\X$. We show that Poincaré duality holds for a toric scheme and that its equivariant operational Chow cohomology ring is isomorphic to the ring of piecewise polynomail functions on the cone over the polyhedral complex defining the toric scheme (Theorem~\ref{th:intro1}). In Section~\ref{sec:special} we give combinatorial descriptions of the equivariant Chow homology group and of the equivariant operational Chow cohomology ring of the special fiber of a toric scheme.  For this we show that these groups can be computed in terms of equivariant Chow groups of smooth  toric varieties over a field via an exact sequence. In particular, we show that the equivariant Chow cohomology ring of the special fiber can be identified with a ring of affine piecewise polynomial functions (Theorem~\ref{th:intro2}). 

Section~\ref{sec:currents} contains the main definitions of this paper, namely the equivariant forms and currents on a toric variety defined over a discretely valued field. We start by describing combinatorially the map $\iota^*\iota_*$ mentioned before. Then we define the equivariant forms and currents as limits of equivariant Chow groups of special fibers and characterize these combinatorially (Proposition~\ref{prop:intro3}). We also define an equivariant Green currents associated to an equivariant cycle. 

In Section~\ref{sec:equi-arith-chow} we define the equivariant arithmetic Chow group of a toric variety over a discretely valued field and an extended version thereof. We show that these can be identified with direct and inverse limits of equivariant Chow groups of toric models in Theorems~\ref{th:intro5} and \ref{th:intro6}), respectively.

\subsection{Conventions} For the whole article $K$ denotes a field equipped with a non-trivial discrete valuation $\operatorname{val}_K \colon K^{\ast} \to \Z$. We denote by $R$ the corresponding valuation ring, $\mathfrak{m}$ its maximal ideal with generator $\varpi$, and by $\kappa$ the residue field, which we assume to be algebraically closed. 

We set $S = \operatorname{Spec}\left(R\right)$ and denote by $\eta$ and by $s$ the generic and special points of $S$, respectively.  All schemes over $S$ are assumed to be integral and of finite type. Given a scheme $\mathcal{X}$ over $S$, we set $\mathcal{X}_{\eta} \coloneqq \mathcal{X} \otimes_S \operatorname{Spec}(K)$ for the generic, and $\mathcal{X}_{s} \coloneqq \mathcal{X} \otimes_S \operatorname{Spec}(k)$ for the special fiber. 

All of the Chow groups considered in this article are assumed to be tensored with $\Q$. 

\vspace{0.5cm}
\noindent
\emph{Acknowledgments:} I am thankful to José Burgos and Roberto Gualdi for useful and enlightening discussions.  I am especially thankful to Kiumars Kaveh for pointing out some inaccuracies regarding the definition of the space of affine piecewise affine polynomial functions on a polyhedral complex appearing in the previous version of this paper. I am especially thankful to anonymous referees for providing very insightful comments. This article would have not been possible without their support.

\section{Toric schemes over a DVR}\label{sec:toric-schemes}
We briefly recall the combinatorial characterization of toric schemes over a discrete valuation ring (DVR).  We refer to \cite[Chapter~3]{BPS} for definitions and results regarding strongly convex rational (SCR) polyhedral complexes and assume the reader is familiarized with such notions (see also \cite{chow-toric} for a recollection). \\

\textbf{Notation.} We use standard notation in toric geometry. We let $N$ be an $n$-dimensional lattice. We denote by $M = N^{\vee}$ its dual lattice. For any ring $R$ we write $N_R$ and $M_R$ for the tensor products $N \otimes_{\Z} R$ and $M_{R} = M \otimes_{\Z} R$, respectively.  Given a SCR polyhedral complex $\Pi$ in $N_{\R}$ we denote by $c(\Pi)$ in $N_{\R} \times \R_{\geq 0}$ the \emph{cone} over $\Pi$ and by $\rec(\Pi)$ its recession complex in $N_{\R}$. Recall that this is not always a fan. Indeed, in \cite{BS} the authors exhibit an example of a recession complex where the intersection of two cones is neither empty nor a common face. However, we will mostly deal with $\Pi$ being complete, in which case $\rec(\Pi)$ is indeed always fan.  $\Pi$ is said to be \emph{regular} if $c(\Pi)$ is regular, i.e.~if each cone in the fan $c(\Pi)$ is generated by a subset of a basis of $N \times \Z$.

 \subsection{Toric schemes over a DVR}\label{subsec:toric-schemes}

We recall the combinatorial characterization of proper toric schemes over $S$ in terms of complete SCR polyhedral complexes. 
We follow \cite[Section~3.5]{BPS}.

Assume that $N$ is the cocharacter lattice of a split torus $\mathbb{T}_K$ over $K$ of dimension $n$ which is the base change of a split torus $\mathbb{T}_S \simeq \mathbb{G}_{m,S}^n$ over $S$.  As before, $M = N^{\vee}$ denotes its dual lattice. Let $\widetilde{N} \coloneqq N \oplus \Z$, $\widetilde{M} \coloneqq M \oplus \Z$. 

The following definition is taken from \cite[Definition 3.5.1]{BPS}.  
\begin{Def}\label{def:toric-scheme}
A \emph{toric scheme over $S$} of relative dimension $n$ is a normal integral separated scheme of finite type $\mathcal{X}$, equipped with a dense open embedding $\mathbb{T}_K \hookrightarrow \mathcal{X}_{\eta}$ and an $S$-action of $\mathbb{T}_S$ over $\mathcal{X}$ that extends the action of $\mathbb{T}_K$ on itself by translations. 
\end{Def}
Note that if $\mathcal{X}$ is a toric scheme over $S$, then $\mathcal{X}_{\eta}$ is a toric variety over $K$ with torus $\mathbb{T}_K$. 

The following is \cite[Definition 3.5.2]{BPS}.
\begin{Def}\label{def:toric-model}
Let $X$ be a toric variety over $K$ with torus $\mathbb{T}_K$ and let $\mathcal{X}$ be a toric scheme over $S$ with torus $\mathbb{T}_S$. We say that $\mathcal{X}$ is a \emph{toric model of $X$ over $S$} if the identity morphism on $\mathbb{T}_K$ can be extended to an isomorphism from $X$ to $\mathcal{X}_{\eta}$.
If $\mathcal{X}'$ and $\mathcal{X}$ are two toric models of $X$ over $S$ and $\alpha \colon \mathcal{X}' \to \mathcal{X}$ is an equivariant $S$-morphsim, we say that $\alpha$ is a \emph{morphism of toric models} if its restriction to $\mathbb{T}_K$ is the identity. Two toric models are said to be \emph{isomorphic} if there exists a morphism of toric models between them which is an isomorphism.
\end{Def}
\begin{Def}
A toric model $\mathcal{X}$ is said to be \emph{regular} or \emph{proper},  if it is a regular, respectively, a proper scheme. 
\end{Def} 
\begin{rem}\label{rem:directed}
It follows from \cite[Chapter IV]{KKMD} that any toric model $\mathcal{X}$ of $X$ over $S$ is dominated by a regular one, i.e. there exists a regular toric model $\mathcal{X}'$ of $X$ over $S$ together with a morphism of toric models $\mu \colon \mathcal{X}'\to \mathcal{X}$.
\end{rem}

\begin{Def}\label{def:models} Let $X$ be a complete toric variety over $K$. The set $R(X)$ consists of all isomorphism classes of regular proper toric models of $X$ over $S$. 
It is endowed with the partial order given by $\mathcal{X}' \geq \mathcal{X}$ if $\mathcal{X}'$ dominates $\mathcal{X}$, i.e. if there exists a morphism of toric models $\mu \colon \mathcal{X}' \to \mathcal{X}$. By Remark \ref{rem:directed}, the partial order "$\geq$" defines a directed set structure on $R(X)$. 
\end{Def}

\subsubsection*{Combinatorial characterization}\label{subsec:correspondence}

Let $\widetilde{\Sigma}$ be a rational fan in $N_{\R} \oplus \R_{\geq 0}$. To $\widetilde{\Sigma}$ we can associate a toric scheme over $S$, which we denote by $\mathcal{X}_{\widetilde{\Sigma}}$. This is done in the usual way by associating an affine toric scheme $\mathcal{X}_{\sigma}$ to each cone $\sigma \in \widetilde{\Sigma}$ and then applying an appropriate gluing construction 
\begin{align}\label{eq:gluing}
\mathcal{X}_{\widetilde{\Sigma}} \coloneqq \bigcup_{\sigma \in \widetilde{\Sigma}}\mathcal{X}_{\sigma}.
\end{align}
We gather here some properties regarding this construction and refer for details to \cite[Section~3.5]{BPS}.

\begin{enumerate}
\item There are two different types of cones in $\widetilde{\Sigma}$. The ones that are contained in the hyperplane $N_{\R} \times \{0\}$ and the ones that are not.  In general, if $\sigma$ is contained in $N_{\R} \times \{0\}$, then $\mathcal{X}_{\sigma}$ is contained in the generic fiber and it agrees with the classical affine toric varity $X_{\sigma}$ over $K$.  Note that if $\widetilde{\Sigma}$ is complete in $N_{\R} \oplus \R_{\geq 0}$, then the cones contained in $N_{\R} \times \{0\}$ are the ones coming from recession cones of the polyhedra in the polyhedral complex $\Pi = \widetilde{\Sigma} \cap \left(N_{\R} \times \{1\}\right)$. 

If $\sigma$ is not contained in $N_{\R} \times \{0\}$ then $\mathcal{X}_{\sigma}$ is not contained in the generic fiber.

\item Set 
\[
\Pi \coloneqq \widetilde{\Sigma} \cap \left(N_{\R} \times \{1\}\right) \quad \text{and} \quad \Sigma \coloneqq \widetilde{\Sigma} \cap \left(N_{\R} \times \{0\}\right).
\]
Then, given polyhedra $\Lambda$, $\Lambda'$ in $\Pi$ with $\Lambda' \subseteq \Lambda$, we have a natural open immersion of affine toric schemes $\mathcal{X}_{\Lambda} \hookrightarrow \mathcal{X}_{\Lambda'}$. Moreover, if a cone $\sigma \in \Sigma$ is a face of a cone $c(\Lambda)$ for some $\Lambda \in \Pi$, then the affine toric variety $X_{\sigma}$ is also an open subscheme of $\mathcal{X}_{\Lambda}$. The gluing construction \eqref{eq:gluing} can be written as 
\[
\mathcal{X}_{\widetilde{\Sigma}} = \bigcup_{\Lambda \in \Pi}\mathcal{X}_{\Lambda} \cup \bigcup_{\sigma \in \Sigma}X_{\sigma}.
\]

\item There are open immersions
\[
\mathbb{T}_K \hookrightarrow \mathcal{X}_{\eta} \hookrightarrow \mathcal{X}_{\widetilde{\Sigma}}
\]
of schemes over $S$ and there is an action of $\mathbb{T}_S$ over $\mathcal{X}_{\widetilde{\Sigma}}$, constructed as in the case of toric varieties over a field, which extends the action of $\T_K$ on itself. It follows that $\mathcal{X}_{\widetilde{\Sigma}}$ is a toric scheme over $S$. 
\item $\Sigma$ is a fan and defines a toric variety $X_{\Sigma}$ over $K$ which coincides with the generic fiber $\mathcal{X}_{\widetilde{\Sigma},\eta}$. It follows that $\mathcal{X}_{\widetilde{\Sigma}}$ is a toric model of $X_{\Sigma}$.
\item The special fiber $\mathcal{X}_{\widetilde{\Sigma},s}$ has an induced action by $\mathbb{T}_k$ but, in general, it is not a toric variety over $k$ (it is not necessarily irreducible nor reduced). It is reduced if and only if the vertices of all $\Lambda \in \Pi$ are contained in the lattice $N$. The reduced schemes associated to its irreducible components are toric varieties over $k$ with this action (see Subsection \ref{subsec:orbits}). 

\item If the fan $\widetilde{\Sigma}$ is complete in $N_{\R} \times \R_{\geq 0}$,, then the scheme $\mathcal{X}_{\widetilde{\Sigma}}$ is proper over $S$. In this case, the set $\left\{\mathcal{X}_{\Lambda}\right\}_{\Lambda \in \Pi}$ is an open cover of $\mathcal{X}_{\widetilde{\Sigma}}$.

\end{enumerate}

In fact, any toric scheme over $S$ of relative dimension $n$ arises from such a rational fan in $N_{\R} \oplus \R_{\geq 0}$. The following follows from \cite[Section~IV. 3.]{KKMD}.
\begin{theorem}\label{th:kkmd}
The correspondence 
\[
\widetilde{\Sigma} \longmapsto \X_{\widetilde{\Sigma}}
\]
is a bijection between the set of rational fans in $N_{\R} \oplus \R_{\geq 0}$ and the set of isomorphism classes of toric schemes over $S$ of relative dimension $n$. The scheme $\X_{\widetilde{\Sigma}}$ is proper if and only if $\widetilde{\Sigma}$ is complete in $N_{\R} \times \R_{\geq 0}$. It is regular if and only $\widetilde{\Sigma}$ is regular. 
\end{theorem}

Moreover,  in the complete case, using \cite[Corollary 2.1.13]{BPS}, we obtain the following (see also \cite[Theorem~3.5.3]{BPS}). 
\begin{theorem}\label{th:corr-comp}
The correspondence 
\[
\Pi \longmapsto \mathcal{X}_{c(\Pi)}
\]
gives a bijection between complete SCR polyhedral complexes in $N_{\R}$ and isomorphism classes of proper toric schemes over $S$ of relative dimension $n$. Moreover, $\Pi$ is regular if and only if $\mathcal{X}_{c(\Pi)}$ is regular.

\end{theorem}

Now, fix a complete fan $\Sigma$ in $N_{\R}$. 
\begin{cor}\label{th:models}
The correspondence 
\[
\Pi \longmapsto \mathcal{X}_{c(\Pi)}
\]
gives a bijection between complete SCR polyhedral complexes in $N_{\R}$ such that $\rec(\Pi) = \Sigma$ and isomorphism classes of proper toric models of $X_{\Sigma}$ over $S$. Moreover, $\Pi$ is regular if and only if $\mathcal{X}_{c(\Pi)}$ is regular.

\end{cor}

Given a complete polyhedral complex $\Pi$, in order to simplify notation, we write $\mathcal{X}_{\Pi}$ for $\mathcal{X}_{c(\Pi)}$. 
\begin{Def}
Note that $\Sigma$ is also a complete SCR polyhedral complex in $N_{\R}$ with $\rec(\Sigma) = \Sigma$. The toric scheme $\mathcal{X}_{\Sigma}$ is hence a model over $S$ of $X_{\Sigma}$ which is called the \emph{canonical model}. Its special fiber $\mathcal{X}_{\Sigma,s} = X_{\Sigma,k}$ is the toric variety over $k$ defined by the fan $\Sigma$.

\end{Def}
\begin{Def}\label{def:toricmodels}
Consider the set $R(\Sigma)$ of all complete regular SCR polyhedral complexes $\Pi$ in $N_{\R}$ such that $\operatorname{rec}(\Pi) = \Sigma$.  It is endowed with the partial order given by
\[
\Pi' \geq \Pi \quad \text{iff} \quad \text{ $c(\Pi')$ is a subdivision of $c(\Pi)$}.
\]

Given two elements in $R(\Sigma)$ there is always a third one dominating both. This gives a directed set structure. 
\end{Def}
Now, given $\Pi' \geq \Pi$ in $R(\Sigma)$ we have an induced proper morphism of toric models $\pi \colon \X_{\Pi'} \to \X_{\Pi}$ and conversely, any such morphism is induced by a subdivision(see \cite{KKMD}).  Hence $R(\Sigma)$ agrees with the directed set $R(X_{\Sigma})$ of regular proper toric models of $X_{\Sigma}$ from Definition \ref{def:models}.
\begin{rem}
Note that the canonical model defines a minimal element in $R(\Sigma)$. The existance of this canonical model, which is smooth, will play an important role afterwards. 
\end{rem}

\subsection{Torus orbits}\label{subsec:orbits}

As before, $N$ is the cocharacter lattice of a split torus $\mathbb{T}_K$ over $K$ of dimension $n$ and $M = N^{\vee}$ is its dual lattice.
Let $\Pi$ be a complete regular SCR polyhedral complex in $N_{\R}$ and let $\X_{\Pi}$ be the corresponding proper regular toric scheme over $S$. Set $\Sigma = \rec(\Pi)$. 

We recall the combinatorial characterization of torus orbits in $\X_{\Pi}$. We mainly follow \cite[Section~3.5]{BPS}.

There are two kinds of toric orbits. 
\begin{itemize}
\item First, there is a bijection between cones in $\Sigma$ and the set of orbits under the action of $\T_K$ on $\X_{\Pi, \eta}$ taking $\sigma$ to the Zariski closure $V(\sigma)$ in $\X_{\Pi}$ of the orbit $O(\sigma)\subseteq \X_{\Pi,\eta} = X_{\Sigma}$. Then $V(\sigma)$ is a horizontal scheme over $S$ in the sense that the structure morphism $V(\sigma) \to S$ is dominant, of relative dimension $n-\dim(\sigma)$. 

Moreover, $V(\sigma)$ has itself the structure of a regular proper toric scheme over $S$. Indeed, let $N(\sigma) = N/(N \cap \R\sigma)$, $M(\sigma) = N(\sigma)^{\vee}$ and $\pi_{\sigma} \colon N_{\R} \to N(\sigma)_{\R}$ the linear projection. Consider the set
\[
\Pi(\sigma) = \left\{ \pi_{\sigma}(\Lambda) \; \big{|} \; \Lambda \in \Pi \; , \; \sigma \subseteq \rec(\Lambda) \right\} \subseteq N(\sigma)_{\R}.
\]
Then $\Pi(\sigma)$ is a complete regular SCR polyhedral complex in $N(\sigma)_{\R}$ and we have an isomorphism of toric schemes over $S$
\[
V(\sigma) \simeq \X_{\Pi(\sigma)}
\]
(see \cite[Proposition 3.5.7]{BPS}).
\item On the other hand, there is a bijection between polyhedra in $\Pi$ and the set of orbits under the action of $\T_{\kappa}$ on the special fiber $\X_{\Pi,s}$. Let $\widetilde{N} = N \oplus \Z$. Given $\Lambda \in \Pi$, set $\widetilde{N}(\Lambda) = \widetilde{N}/(\widetilde{N} \cap \R c(\Lambda))$, $\widetilde{M}(\Lambda) = \widetilde{N}(\Lambda)^{\vee}$ and $\pi_{\Lambda} \colon \widetilde{N}_{\R} \to \widetilde{N}(\Lambda)_{\R}$ the linear projection. Then to $\Lambda$ one associates a vertical cycle $V(\Lambda)$ (vertical in the sense that it is contained in the special fiber). It has the structure of a toric variety over $\kappa$ with torus $\T(\Lambda) = \operatorname{Spec}\left(k\left[\widetilde{M}(\Lambda)\right]\right)$ and corresponding fan given by 
\[
\Pi(\Lambda) = \left\{\pi_{\Lambda}\left(c\left(\Lambda'\right)\right) \; \big{|} \; \Lambda' \in \Pi \; , \; \Lambda \prec \Lambda' \right\} \subseteq \widetilde{N}(\Lambda)_{\R}
\]
(see \cite[Proposition 3.5.8]{BPS}). Here, recall that \enquote{$\prec$} stands for \enquote{is a face of}. 

In particular, there is a one-to-one correspondence between the vertices of $\Pi$ and the irreducible components of the special fiber: 
\[
v \leftrightarrow V(v).
\]

The toric orbits contained in $V(v)$ correspond to the polyhedra $\Lambda \in \Pi$ containing $v$. In particular, the components given by two vertices $v$ and $v'$ share an orbit of dimension $\ell$ if and only if there exists a polyhedron of dimension
$n - \ell$ containing both $v$ and $v'$.
\end{itemize}
\begin{rem}
The special fiber is not necessarily reduced. It is reduced if and only if all of the vertices of $\Pi$ are in $N$ (see \cite[Example 3.6.11]{BPS}).  
\end{rem}

\section{Equivariant Chow groups of toric schemes}\label{sec:eq-chow}
Assume that $N$ is the cocharacter lattice of a split torus $\mathbb{T}_K$ over $K$ of dimension $n$ which is the base change of a split torus $\mathbb{T}_S \simeq \mathbb{G}_{m,S}^n$ over $S$. 
We consider the category of $\T_S$-schemes over $S$, i.e. integral separated $S$-schemes of finite type with a torus $\T_K \hookrightarrow \X_{\eta}$ and an $S$-action of $\T_S$ over $\X$ that extends the action of $\T_K$ on itself by translations. Morphisms of $\T_S$-schemes are $\T_S$-equivariant. 

In this section we study properties of the equivariant Chow homology and the equivariant operational Chow cohomology group of a $\T_S$-scheme $\X$. The latter carries a ring structure. We show that Poincaré duality holds for a toric schemes and that its equivariant operational Chow cohomology ring is isomorphic to the ring of piecewise polynomial functions on the cone over the polyhedral complex defining the toric scheme. This generalizes the case of toric varieties over a field (see \cite[Theorem~1]{Payne-equi}). We will make use of \cite{EG},\cite{Brion-equi}, \cite{BR} and \cite{KT}.

Recall that we assume that the Chow groups are tensored with $\Q$.
\subsection{Equivariant Chow homology groups of toric schemes}\label{subsec:hom-schemes}

\begin{Def}\label{def:abs-dim}
For any scheme $\X$ over $S$ we write $\CH_k(\X)$ for the Chow homology group of absolute dimension $k$-algebraic cycles modulo rational equivalence.  Note that this is independent of $S$.
\end{Def}
\begin{rem}\label{rem:dim}
Our notion of absolute dimension coincides with the notion of \emph{relative} dimension given in \cite[Section 20.1]{fulint} plus 1. We choose to work with this absolute notion because it coincides with the one used in~\cite{chow-toric} and in \cite{BGS}. 
\end{rem}

Now, for any $\T_S$-scheme $\X$ over $S$ (that is, an $S$-scheme with an action of $\T_S$) of (absolute) dimension $n+1$, we consider the equivariant Chow groups $CH_k^{\T_S}(\X)$ in the sense of \cite[Section 6.2]{EG}.  Recall that these are given by choosing a finitely generated projective $S$-module $E$, such that $\T_S$ acts freely on an open set $U \subseteq E$ whose complement has arbitrarily high codimension ($> n+1-k$). Then the $k$'th equivariant Chow group of $\X$ is defined as $\CH_{k+\ell-n-1}\left(\X\times^{\T_S}U \right)$, where $\ell = \dim(U)$ and where $\X\times^{\T_S}U$ is the quotient space $(\X\times U)/{\T_S}$.
\begin{rem}
Existence of the sets $E, U$ above satisfying the desired properties is given in \cite[Lemma 7]{EG}. Moreover, by Proposition~23 in \emph{loc.~cit.} the quotient space $\X\times^{\T_S}U$ is again a scheme (which by the argument below can be even taken to be toric). The well-definiteness of $CH_k^{\T_S}(\X)$ (i.e. the fact that it is independent of the choice of $E,U$) follows as in \cite[Definition/Proposition~1]{EG} using \cite[Lemma 9]{EG}.
Set
\[
CH_*^{\T_S}(\X) \coloneqq \bigoplus_{k \in \Z}CH_k^{\T_S}(\X).
\]
\end{rem}

\begin{rem}
Note that unlike for ordinary Chow groups, the equivariant Chow groups $\CH^{\T_S}_k(\X)$ can be non-zero for any $k \leq n+1$.
\end{rem}
Similar to the complex case described in \cite[Section 2.2]{BR}, the approximation $U$ can be chosen such that $U$ has a faithful action of a large $S$-torus $\T_{U,S}$ with dense generic orbit, and such that the action of $\T_S$ on $U$ is given by an inclusion $\T_S \hookrightarrow \T_{U,S}$. Then the quotient space $\X\times^{\T_S}U$ has an action of $\widetilde{\T}_S = \T_{U,S}/\T_{S}$. If $\X$ is a toric scheme, then this gives $\X\times^{\T_S}U$ the structure of a toric scheme. It is moreover smooth, regular, complete, iff the scheme $\X$ has these properties. 

\begin{exa}\label{exa:point}
\begin{enumerate}
\item Consider the base scheme $S$. Then $\CH^{\T_S}_{*}(S)$ can be identified with the symmetric algebra $\widetilde{S}=\on{Sym}(\widetilde{M})$ over the character group $\widetilde{M} = M \oplus \Z$ of the torus $\T_S$ (actually over $\widetilde{M}_{\Q} = M_{\Q} \oplus \Q$ since we are tensoring with $\Q$). Moreover, for any positive integer $k$ the degree $-k$ part $\CH^{\T_S}_{-k}(S)$ can be identified with the degree $k$ part $\on{Sym}^k(\widetilde{M})$. For the proof of this fact we refer to \cite[Lemma 2]{EG-char}, where the result is proven over a field, but generalizes to the DVR case.  Indeed, the identification $\CH_{-k}(S) = \on{Sym}^k(\widetilde{M})$ is given explicitly in the following way: let $\chi \in \widetilde{M}$ a character and let $k_{\chi} \colon \T_S \to \mathbb{G}_{m,S}$ the corresponding $1$-dimensional $\T_S$-representation. Denote by $L_{\chi}$ the line bundle $U \times^{\T_S} k_{\chi} \to U / \T_S$.
The map $\chi\mapsto c_1(L_{\chi})$ extends to a graded ring isomorphism $\widetilde{S} \to \CH^{\T_S}_{*}(S)$, where the latter is negatively graded.

\item More generally, any $\CH_*^{\T_S}(\X)$ is an $\widetilde{S}$-module, where $\widetilde{M}$ acts on $\CH_*^{\T_S}(\X)$ by homogeneous elements of degree $-1$. 
\end{enumerate}
\end{exa}

Now, since the approximations can be taken within the toric world, standard results on Chow groups of toric schemes can be carried out to the equivariant setting. In particular, we have the following description of the equivariant Chow groups.

\begin{prop}\label{prop:chow-hom}
Let $\X_{\Pi}$ be a regular, proper toric scheme over $S$ with corresponding SCR polyhedral complex $\Pi$ in $N_{\R}$. Consider $c(\Pi)$ in $N_{\R} \oplus R_{\geq 0}$ the cone over $\Pi$. Then $CH_*^{\T_S}(\X_{\Pi})$ is an $\widetilde{S}$-module defined by generators $\left[V(\sigma)\right]$, for $\sigma \in c(\Pi)$.  Moreover, these generators (as an $\widetilde{S}$-module) are divided into ``horizontal'' and ``vertical'' cycles, depending on whether the cone is contained in $N_{\R} \times\{0\}$ or not. 
\end{prop}
\begin{proof}
Consider the following commutative diagram 
\begin{center}
    \begin{tikzpicture}
      \matrix[dmatrix] (m)
      {
      0 & 0&0  \\
      \widetilde{M} \CH_k^{\T_S}\left(\X_{\Pi,s}\right)  &  \widetilde{M}\CH_k^{\T_S}\left(\X_{\Pi}\right)   &   \widetilde{M}  \CH_k^{\T_S}\left(\X_{\Pi,\eta}\right) & 0 \\
        \CH_k^{\T_S}\left(\X_{\Pi,s}\right)  & \CH_k^{\T_S}\left(\X_{\Pi}\right)  & \CH_k^{\T_S}\left(\X_{\Pi,\eta}\right) & 0 \\
        \CH_k\left(\X_{\Pi,s}\right)  & \CH_k\left(\X_{\Pi}\right)  & \CH_k\left(\X_{\Pi,\eta}\right) & 0 \\
       0 & 0 & 0  \\
      };
      \draw[->] (m-3-1) to node[above]{$\iota_*^{\T_S}$}  (m-3-2);
       \draw[->] (m-3-2) to node[above]{$j_{\T_S}^*$}  (m-3-3);
        \draw[ ->] (m-3-3)to (m-3-4);
      \draw[ ->] (m-3-1)to (m-4-1);
       \draw[ ->] (m-3-2)to (m-4-2);
        \draw[ ->] (m-3-3)to (m-4-3);
        
      \draw[->] (m-4-1) to node[above]{$\iota_*$} (m-4-2);
       \draw[->] (m-4-2) to node[above]{$j^*$} (m-4-3);
        \draw[ ->] (m-4-3)to (m-4-4);
          \draw[ ->] (m-4-1) to (m-5-1);
            \draw[ ->] (m-4-2) to (m-5-2);
              \draw[ ->] (m-4-3) to (m-5-3);
               \draw[ ->] (m-1-1) to (m-2-1);
            \draw[ ->] (m-1-2) to (m-2-2);
              \draw[ ->] (m-1-3) to (m-2-3);
               \draw[ ->] (m-2-1) to (m-3-1);
            \draw[ ->] (m-2-2) to (m-3-2);
              \draw[ ->] (m-2-3) to (m-3-3);
               \draw[ ->] (m-2-1) to (m-2-2);
                \draw[ ->] (m-2-2) to (m-2-3);
                 \draw[ ->] (m-2-3) to (m-2-4);

     \end{tikzpicture}
     \end{center}
where $\iota \colon \X_{\Pi,s}\hookrightarrow \X_{\Pi}$ and $j \colon \X_{\Pi,\eta}\hookrightarrow \X_{\Pi}$ denote the inclusions of the special and the generic fiber, respectively. 

By excision, the bottom row is exact (see \cite[Equation 3.1]{chow-toric}, also \cite[Section 1.9]{fulint}).  Using equivariant excision formula (see \cite[Lemma 4]{EG}), we have that the middle row is also exact. Hence, $\CH_k^{\T_S}\left(\X_{\Pi}\right)$ is generated by the images of the generators of $\CH_k^{\T_S}\left(\X_{\Pi,s}\right)  $ and by a choice of preimages of the generators of $\CH_k^{\T_S}\left(\X_{\Pi,\eta}\right)$. 

Now, from \cite[Section 2.3, Corollary 1]{Brion-equi}, we have that the left and right columns are exact. Therefore, from the proof of \cite[Theorem~1.1]{chow-toric} and \cite[Theorem~2.1]{Brion-equi}, it follows that, as an $\widetilde{S}$-module, $\CH_k^{\T_S}\left(\X_{\Pi,s}\right)$ is generated by classes of vertical cycles corresponding to cones $\sigma \in c(\Pi)$ not contained in $N_{\R} \times \{0\}$, and $\CH_k^{\T_S}\left(\X_{\Pi,\eta}\right)$ is generated by horizontal cycles corresponding to cones $\sigma \in c(\Pi)$ which are contained in $N_{\R} \times \{0\}$. This concludes the proof. 

\end{proof}

\begin{prop}\label{prop:rel-inv}
In $\CH_*^{\T_S}(\X_{\Pi})$ we have the relations
\begin{equation}\label{eq:rel}
\left[\on{div}_W(f)\right]-\chi\left[W\right]=0,
\end{equation}
 where $W$ is an invariant subvariety of $\X_{\Pi}$ and $f$ is a non-constant rational function on $W$ which is an eigenvector of $\T_S$ of weight $\chi \in \widetilde{M}$, i.e. such that 
\[
g \cdot f = \chi(g) f \quad \forall g \in \T_S.
\]

In \eqref{eq:rel}, $\chi\left[W\right]$ denotes the action of $\widetilde{M}$ on $\CH_*^{\T_S}(\X_{\Pi})$ by homogeneous maps of degree $-1$ as in Example \ref{exa:point}.
\end{prop}
 
 \begin{proof}
Let $W$ and $f$ as in the statement of the proposition. Let $L_{\chi}$ be the line bundle over $U/\T_S$ from example \ref{exa:point}. Then $f$ can be considered as a rational section of the pullback of $L_{\chi}$ to $(W \times U)/\T_S$, whith divisor $\left[\div_W(f)\right]$. It follows that in the $\widetilde{S}$-module $\CH_*^{\T_S}(\X)$ the relation $\left[\div_W(f)\right] = \chi[W]$ holds. 
 \end{proof}
 
 \begin{rem}
 In the field case, the above relations describe \emph{all} relations. We don't know if this is the case here.
 \end{rem}

Our next goal is to study the equivariant operational Chow cohomology rings of toric schemes over $S$.  These are the equivariant versions of the operational Chow rings considered by Fulton in \cite[Chapters 17, 20]{fulint}. We show that these can be identified with rings of piecewise polynomial functions on the cones over the SCR polyhedral complexes defining the toric schemes.  

First we recall the case of toric varieties over a field.  
\subsection{Equivariant operational Chow cohomology groups of toric varieties}\label{sec:toric}
We recall the identification of the equivariant operational Chow cohomology ring of a toric variety over a field with the ring of piecewise polynomial functions on the associated fan.  We follow mainly \cite{BR}.  

\subsubsection*{Preliminaries on generalized fans}
Let $N$ be a lattice of dimension $n$. 

In order to study equivariant cohomology groups of toric varieties and functoriality properties thereof we need to distinguish between toric varieties in the classical sense, i.e. normal toric varieties containing a torus $T$ as a dense opens subset and an action of the torus on the variety which extends the action of the torus on itself, and toric varieties in the sense of Brion \cite{BR}, that is, normal varieties with an action of $T$ such that there is a dense obit and the stabilizers of points are connected. In the latter, the dimension of the torus may be bigger than the dimension of the variety. For example, if we want to think of the inclusion of a torus orbit inside the toric boundary as a T-equivariant map of toric varieties (i.e. we want the same torus to be acting on both sides) then we have to consider Brions notion. One can of course take the quotient of the torus by the stabilizer of a point in the dense orbit to a toric variety in the classical case, but in doing so, we don't longer have the action of the original torus. We will need this formalism when studying equivariant cohomology of toric varieties, in particular, when defining the pushforward map, since here the torus which is acting is fixed. 

From now on, we will use the term \emph{toric variety} in the sense of Brion. These are classified in terms of \emph{generalized fans}. We start by recalling this notion. 

A \emph{generalized fan} is a collection of rational polyhedral cones $\{\sigma\}$ satisfying the usual axioms of a fan. Hence, the only difference with the classical definition is that the cones are no longer required to be strongly convex. The usual notions of \emph{support} $|\Sigma|$ (union of cones), \emph{complete}, smooth and simplicial extend to generalized fans in the obvious way.  We write $L(\sigma)$ for the linear space $\sigma - \sigma$. 

If $\Sigma$ is a generalized fan, then $\sigma_0 \coloneqq \bigcap_{\sigma \in \Sigma}\sigma$ is a linear subspace of $N_{\R}$. Let $\overline{N} = N/\left(\sigma_0 \cap N\right)$ with quotient map $N \to \overline{N}$. Clearly, $\Sigma$ is a (usual) fan if and only if $\sigma_0 = \{0\}$ and $\overline{\Sigma} = \left\{\overline{\sigma} \; | \; \sigma \in \Sigma \right\}$ is a fan in $\overline{N}_{\R}$.  A generalized fan which is a fan the usual sense is said to be \emph{non-degenerate}.

If $\Sigma$ is a generalized fan with linearity space $\sigma_0$, then for any natural number $k$ we denote by $\Sigma(k)$ the set of cones of dimension $\dim(\sigma_0) + k$.

Let $\sigma$ be a cone of a generalized fan $\Sigma$ in $N_{\R}$. The \emph{star $\Sigma(\sigma)$} is the generalized fan in $N_{\R}$ with linear subspace $\sigma_0=L(\sigma)$ and cones corresponding to the cones in $\Sigma$ containing $\sigma$. 

The construction of a toric variety from a generalized fan and vice-versa is done similarly as in the classical case. We refer to \cite[Section 2.1]{BR} for details. In particular, there is a bijection between the cones in $\Sigma$ and the torus orbits in $X_{\Sigma}$.

\begin{rem} In \cite[pg.\,278]{CLS} the  toric variety $X_{\Sigma}$ of a generalized fan $\Sigma$ is defined to be the (classical) toric variety of the non-degenerate fan $\overline{\Sigma}$. Since we are dealing with equivariant cohomology, we prefer to work with Brion's definition. 
\end{rem}

\begin{exa}\label{exa:star} Figure \ref{fig:star} shows an example of a fan $\Sigma$ (in this case a complete, non-degenerate fan) and the generalized star fan $\Sigma(\tau)$. The shaded regions correspond to the two dimensional cones. 
\begin{figure}[H]
\begin{center}
\begin{tikzpicture}[scale=1.5]
\draw (-4.5,0.7) node{$\sigma'_1$};
\draw (-4.5,-0.7) node{$\sigma'_2$};
\draw (0.5,0.7) node{$\sigma_1$};
\draw (0.3,-0.7) node{$\sigma_2$};
    \draw[very thick] (0,0) to (1,0)  node[right]{$\tau$} ;
    \draw[very thick] (0,0) to  (0,1);
      \draw[very thick] (0,0) -- (-1,-1);
     \draw (0.5,-1.5) node{$\Sigma$};     
 \draw[very thick] (-6, 0) to node[above]{$\tau$} (-3,0);
       \draw (-4,-1.5) node{$\Sigma(\tau)$};     
       \filldraw[blue, opacity = 0.1] (0,0) -- (1,0) -- (1,1)--(0,1)--(0,0); 
       \filldraw[blue, opacity = 0.1] (0,0) -- (1,0) -- (1,-1)--(-1,-1)--(0,0); 
       \filldraw[blue, opacity = 0.1] (0,0) -- (-1,-1) -- (-1,1)--(0,1)--(0,0); 
        \filldraw[blue, opacity = 0.1] (-6,0) -- (-3,0) -- (-3,1)--(-6,1)--(-6,0); 
         \filldraw[blue, opacity = 0.1] (-6,0) -- (-3,0) -- (-3,-1)--(-6,-1)--(-6,0); 
 \end{tikzpicture}
\end{center}\caption{The star fan as a generalized fan}\label{fig:star}

 \end{figure}
\end{exa}
\begin{Def}\label{def:morphism-gen-fan}
A \emph{morphism between two generalized fans $\Sigma'$ and $\Sigma$} in $N_{\R}$ is a map $\mu \colon \Sigma' \to \Sigma$ satisfying the following properties:
\begin{itemize}
\item$\mu$ is order preserving with respect to inclusions of cones in $\Sigma'$, resp. $\Sigma$, 
\item For any $\sigma' \in \Sigma'$ we have that 
\[
\sigma' \subseteq \mu(\sigma')+L(\mu(\sigma'_0)).
\]
\end{itemize}
\end{Def}
\begin{exa}\label{exa:map-gen}
The map $\Sigma(\tau)$ to $\Sigma$ sending $\tau $ to $\tau$ and $\sigma'_i$ to $\sigma_i$ in Example \ref{exa:star} is an example of a morphism of generalized fans. 
\end{exa}
\begin{rem} A map of generalized fans $\mu  \colon \Sigma' \to \Sigma$ does not induce a morphism of non-degenerate fans $\overline{\mu} \colon \overline{\Sigma'} \to \overline{\Sigma}$ in the sense of (see e.\,g.\,\cite[Definition 3.3.1]{CLS}). Indeed, this latter notion notion does not include an equivariant version of the case of the closure of a toric variety contained in the boundary of a toric variety i.e. where the same torus is acting on both spaces. The definition above, which is due to Brion \cite[Proposition~2.1 (i)]{BR}, allows to handle this situation, as in the example above. 
%
%
\end{rem}
The following proposition can be found in \cite[Proposition 2.1]{BR}.
\begin{prop}\label{prop:induced-map}
Let $\Sigma, \Sigma'$ be two generalized fans in $N_{\R}$. There is a bijection between equivariant morphisms $\pi \colon X_{\Sigma'} \to X_{\Sigma}$ of toric varieties and maps of generalized fans $\mu \colon \Sigma' \to \Sigma$. Moreover, write $\mu_{\pi}$ for the map of generalized fans corresponding to an equivariant map $\pi$. Then
\begin{enumerate}
\item[(i)] $\pi$ is dominant (resp.\,a closed immersion) if and only if $\mu_{\pi}(\sigma'_0) = \sigma_0$ (resp.\,$\sigma' = \mu_{\pi}(\sigma') + L(\mu_{\pi}(\sigma'_0))$ for all $\sigma ' \in \Sigma'$).
\item[(ii)] $\pi$ is proper if and only of $|\Sigma'| = |\Sigma| + \sigma_0'$.
\end{enumerate}
\end{prop}
\begin{exa}
The map from Example \ref{exa:map-gen} induces an equivariant closed immersion between the toric variety associated to $\Sigma(\tau)$, i.\,e.\,the closure of the torus orbit corresponding to $\tau$ in $\Sigma$ (considered as a toric variety in the sense of Brion) and the toric varity of $\Sigma$.
\end{exa}
\subsubsection*{Rings of piecewise polynomial functions}
\begin{Def}\label{def:piecewisepoly} Let $\Sigma$ be a generalized fan in $N_{\R}$. A map $f \colon |\Sigma| \to \Q$ is said to be \emph{piecewise polynomial} if for any $\sigma \in \Sigma$, the restriction $f|_{\sigma} \colon \sigma \to \Q$ extends to a polynomial function $f_{\sigma}$ on the linear space $L(\sigma)$. In other words, a piecewise polynomial function on $\Sigma$ is a collection of polynomial functions $f_{\sigma} \colon L(\sigma) \to \Q$ such that $f_{\sigma}|_{L(\tau)} = f_{\tau}$ whenever $\tau$ is a face of $\sigma$. 
\end{Def}

Note that a piecewise polynomial function is continuous. The set of piecewise polynomial function on $\Sigma$ is denoted by $PP(\Sigma)$. We endow it with a ring structure given by pointwise addition and multiplication.

A piecewise polynomial function $f$ decomposes uniquely as a sum of homogeneous piecewise polynomial functions $f = \sum_kf_k$, where $f_k = (f_{\sigma, k})$ are the homogeneous components of degree $k$ of the $f_{\sigma}$'s. We write $PP^k(\Sigma)$ for the set of piecewise polynomial functions on $\Sigma$ of degree $k$. Then $PP(\Sigma)$ has a structure of a graded ring 
\[
PP^*(\Sigma)  = \bigoplus_{k \geq 0}PP^k(\Sigma).
\]

Denote by $PP^*(N_{\Q})$ the algebra of polynomial functions on $N_{\Q}$ (seen as a generalized fan). The decomposition of piecewise polynomial functions as sums of their homogeneous components defines a grading on $PP^*(\Sigma)$ with $PP^*(N_{\Q})$ as a graded subalgebra.  Hence we may view $PP^*(\Sigma)$ as an $PP^*(N_{\Q})$-module. 

We now recall from \cite[Section 1.2]{BR} a set of generators of $PP^*(\Sigma)$ as an $PP^*(N_{\Q})$-module in the case that $\Sigma$ is simplicial.  

\begin{Def} Let  $\Sigma$ be a simplicial generalized fan. Recall that $\Sigma(k)$ denotes the set of cones of dimension $\dim \sigma_0 + k$. For each cone $\tau \in \Sigma(1)$ we write $v_{\tau}$ for a representative in $\tau \cap N$ of the unique generator of the quotient $(\tau \cap N)/(\sigma_0 \cap N)$. 
If $\sigma \in \Sigma$ contains $\tau$, then there exists a unique maximal face $\tau'$ of $\sigma$ which does not contain $\tau$. Moreover,  
there exists a unique linear function $\varphi_{\sigma, \tau}\colon L(\sigma) \to \Q$ such that $\varphi_{\tau, \sigma}$ is identically zero on $\tau'$ and such that $\varphi_{\tau, \sigma}(v_{\tau}) = 1$. Since $\sigma_0 \subset \tau'$ it follows that $\varphi_{\tau, \sigma}$ vanishes on $\sigma_0$ and thus does not depend on the choice of $v_{\tau}$. 
\end{Def}

The linear forms $\varphi_{\tau,\sigma}$ glue together to a piecewise linear function $\varphi_{\tau}$ on the star $\Sigma(\tau)$ of $\tau$ (as defined at the beginning of the section). Moreover, we may view $\varphi_{\tau}$ as a function on the whole generalized fan $\Sigma$ by setting it to be zero outside the cones containing $\tau$. Thus, we may view $\varphi_{\tau}$ as an element in $PP^1(\Sigma)$. It follows from Proposition \ref{prop:generators} below that the piecewise linear functions $\varphi_{\tau}$, for $\tau \in \Sigma(1)$, generate the whole algebra $PP^*(\Sigma)$. 

\begin{Def}\label{def:generators}
Let  $\Sigma$ be a simplicial generalized fan. For any cone $\sigma \in \Sigma$ we set $\varphi_{\sigma} = \prod_{\tau \in \sigma(1)}\varphi_{\tau}$. Then $\varphi_{\sigma}$ is a homogeneous piecewise polynomial function on $\Sigma$ of degree $\dim(\sigma)-\dim(\sigma_0)$. 
\end{Def}

 The following is \cite[Corollary~1.2]{BR}.

\begin{prop}\label{prop:generators}
Let  $\Sigma$ be a simplicial generalized fan. Then $PP^*(\Sigma)$ is generated (as a $PP^*(N_{\Q})$-algebra) by the $\varphi_{\sigma}$ for $\sigma \in \Sigma$. In particular, it is finite as a $PP^*(N_{\Q})$-algebra, generated by finitely many elements of degree one. 
\end{prop}

\subsubsection*{Pullback and pushforward}\label{subsec:pull}

Consider two regular generalized fans $\Sigma$, $\Sigma'$ in $N_{\R}$ and a morphism of generalized fans $\mu \colon \Sigma' \to \Sigma$ as in Definition \ref{def:morphism-gen-fan}. By Proposition~\ref{prop:induced-map}, this induces an equivariant morphism of toric varieties $\pi\colon X_{\Sigma'} \to X_{\Sigma}$. Note that for $\sigma' \in \Sigma'$, the linear span $L(\sigma')$ of $\sigma'$ is contained in the linear span $L(\mu(\sigma'))$ of $\mu(\sigma')$. Hence, any polynomial function on $\mu(\sigma')$ defines a polynomial function on $\sigma'$.  Hence, we can make the following definition.

\begin{Def}\label{def:pullback} The \emph{pullback} map
\[
\pi^* \colon PP^*(\Sigma) \longrightarrow PP^*(\Sigma').
\]
 is the morphism of graded algebras over $PP^*(N_{\Q})$ (where $PP^*(N_{\Q})$ acts on $PP^*(\Sigma')$ via $\pi^*$) determined by 
 \[
 f = (f_{\sigma})_{\sigma \in \Sigma} \longmapsto \pi^*f = \left(\pi^*f\right)_{\sigma' \in \Sigma'},
 \]
 where $ \left(\pi^*f\right)_{\sigma'} = f_{\mu(\sigma')}$.
\end{Def}

If $\pi$ is proper then we also have a \emph{pushforward} map.

\begin{Def}\label{def:push}
Assume that $\pi$ is proper. The \emph{pushforward} map
\[
\pi_*\colon PP^*(\Sigma') \longrightarrow PP^*(\Sigma)
\]
is the morphism of $PP^*(N_{\Q})$-modules determined by the following properties
 \begin{enumerate}
\item[(i)] For any $\sigma' \in \Sigma'$ we have 
\[
\pi_*\varphi_{\sigma'}= \begin{cases} \varphi_{\sigma} & \; \text{ if }\mu(\sigma') = \sigma \; \text{ and }\codim (\sigma') =  \codim (\sigma)  \\ 0 & \; \text{ otherwise}.\end{cases}
\]
\item[(ii)] Assume that for all maximal cones $\sigma$ in $\Sigma$ and $\sigma'$ in $\Sigma'$ we have that $\codim(\sigma) = \codim(sigma')$, then for $f = (f_{\sigma'})_{\sigma' \in \Sigma'} \in PP^*(\Sigma')$ we have
\[
\left(\pi_*f\right)_{\sigma} = \varphi_{\sigma}\sum_{\mu(\sigma')= \sigma} \varphi_{\sigma'}^{-1}f_{\sigma'}
\]
for any maximal cone $\sigma \in \Sigma$.
\end{enumerate}
\end{Def}

\begin{exa}\label{exa:push-pull}
This is a generalization of Example \ref{exa:map-gen}. Let $\Sigma$ be a non-degenerate fan and let $\gamma \in \Sigma$ be a cone. Consider the torus orbit closure $V(\gamma)$ and the closed immersion 
\[
\pi\colon V(\gamma) \hookrightarrow X_{\Sigma},
\]
The corresponding morphism of generalized fans is 
\[
\mu \colon \Sigma(\gamma) \to \Sigma,
\]
given by $L(\gamma) \mapsto \gamma$ and $ \sigma + L(\gamma) \mapsto \sigma$.

\begin{itemize}
\item Let $f = (f_{\sigma})  \in PP^*(\Sigma)$. Then 
\[
(\pi^*f)_{\sigma'} = f_{\sigma}
\]
where $\sigma' = \sigma + L(\gamma)$.
\item On the other hand, for $\sigma' = \sigma + L(\gamma) \in \Sigma(\gamma)$ we have 
\[
\pi_*(\varphi_{\sigma'}) = \varphi_{\sigma}.
\]

\end{itemize}
\end{exa}
\subsubsection*{Equivariant operational Chow cohomology groups of toric varieties}
For any smooth and projective variety $Y$ endowed with a torus action $T \curvearrowright Y$ we write $\CH_T^*(Y)$ for the equivariant Chow cohomology group, graded by codimension. It comes with a graded ring structure given by intersection product (see \cite[Section 6.6]{Brion-equi}).

Now, let $N$ be the cocharacter lattice of a torus $\T_F$ over some field $F$ and let $\Sigma$ be a smooth complete generalized fan in $N_{\R}$.

For $\sigma \in \Sigma$ we write $O(\sigma)$ for the corresponding torus orbit.

The following theorem is a reformulation of \cite[Proposition 2.2] {BR} (see also \cite[Theorem~1]{Payne-equi}).

\begin{theorem}\label{th:chow-payne}
\begin{enumerate}
\item The equivariant Chow ring $\CH^*_{\T_F}(O(\sigma))$ is an algebra isomorphic to $PP^*(L(\sigma))$. 
\item Write $\iota_{\sigma}$ for the inclusion of $O(\sigma)$ in $X_{\Sigma}$.  Then the map 
\[
\bigoplus_{\sigma \in \Sigma}\iota_{\sigma}^* \colon \CH^*_{\T_F}(X_{\Sigma}) \to \bigoplus_{\sigma \in \Sigma}\CH^*_{\T_F}(O(\sigma))
\]
is injective and its image is isomorphic to $PP^*(\Sigma)$. 
\item The above isomorphism maps the class of the torus orbit closure $V(\sigma)$ to $\varphi_{\sigma}$. It follows that the $PP(N_{\Q})$-modules $\CH^*_{\T_F}(X_{\Sigma})$ and $PP^*(\Sigma)$ are isomorphic.
\end{enumerate}

\end{theorem}
\begin{rem}
Given an equivariant morphism $\pi \colon X_{\Sigma'} \to X_{\Sigma}$ of smooth, complete toric varieties, one has a pullback map of equivariant Chow cohomology rings
\[
\pi^* \colon \CH^*_{\T_F}(X_{\Sigma}) \longrightarrow \CH^*_{\T_F}(X_{\Sigma'}),
\]
and if $\pi$ proper, one has a pushforward map 
\[
\pi_* \colon \CH^*_{\T_F}(X_{\Sigma'}) \longrightarrow \CH^*_{\T_F}(X_{\Sigma}).
\]
As we mentioned before, the pushforward map does not preserve the product structure, it defines only a group morphism. 

Under the identification in Theorem \ref{th:chow-payne}, the maps $\pi^*$ and $\pi_*$ above agree with the combinatorial maps from Definitions \ref{def:pullback} and \ref{def:push}. (see \cite[Proposition~3.2]{BR}).  
\end{rem}

\subsection{Equivariant operational Chow groups of $\T_S$-schemes}

We consider the category of $\T_S$-schemes over $S$ (that is, $S$-schemes with an action of $\T_S$ together with equivariant $S$-morphisms). In Subsection \ref{subsec:hom-schemes} we considered the homological equivariant Chow groups $\CH_*^{\T_S}(\X)$ of a $\T_S$-scheme $\X$. We now consider the cohomological version thereof, namely, the equivariant operational Chow groups in the sense of \cite[Section 6.2]{EG}. Before we recall the definition, we remark that most of the results for equivariant homological Chow groups over a field hold also for schemes over a DVR, for example the functorial properties with respect to (equivariant) proper, flat and l.\,c.\,i. maps hold (see \cite[pg.\,36 (1)]{EG}). In particular, if $f \colon \X \to \mathcal{Y}$ is an equivariant regular embedding (or equivariant l.\,c.i.\,morphism) of codimension $d$ of $\T_S$-schemes over $S$, then we have a collection of homomorphisms 
\[
f^! \colon \CH^{\T_S}_k\mathcal{Y}' \longrightarrow \CH^{\T_S}_{k -d}\X'
\]
for all $\mathcal{Y}' \to \mathcal{Y}$, $\X' = \X \times_{\mathcal{Y}} \mathcal{Y}'$ and all $k$. The equivariant operational Chow groups studies the class of such operations.

Fix $f\colon \ca X\to \ca Y$ an equivariant morphism of $\T_S$-schemes over $S$.  Then for any equivariant morphism of $\T_S$ schemes $g \colon \ca Y' \to \ca Y$ we let $\ca X'$ denote the fiber square
\begin{center}
    \begin{tikzpicture}
      \matrix[dmatrix] (m)
      {
        \X' & \mathcal{Y}'\\
       \X & \mathcal{Y} \\
      };
      \draw[->] (m-1-1) to node[above]{$f'$} (m-1-2);
      \draw[ ->] (m-1-1) to node[left]{$g'$}(m-2-1);
      \draw[->] (m-1-2)to node[right]{$g$} (m-2-2);
      \draw[->] (m-2-1) to node[above]{$f$} (m-2-2);
      
     \end{tikzpicture}
     \end{center}
     Let $k$ be an integer. An element $c \in \CH_{\T_S}^k(f \colon \ca X \to \ca Y)$ is a collection of homomorphisms 
\[
c_g^m \colon \CH_m^{\T_S}(\ca Y') \longrightarrow \CH_{m-k}^{\T_S}(\ca X'/S)
\]
 for all  $m \in \Z$ and for all equivariant morphisms of $\T_S$ schemes $g \colon \ca Y' \to \ca Y$. These operations have to be  compatible with proper pushforward, flat pull-back and l.\,c.\,i.\,maps in the following way:
 \begin{enumerate}
 \item[(i)] If $h \colon \ca Y'' \to \ca Y'$ is an equivariant and proper map of $\T_S$-schemes, and $g \colon \ca Y' \to \ca Y$ any equivariant map of $\T_S$-schemes, and one forms the fiber diagram 
\begin{center}
    \begin{tikzpicture}
      \matrix[dmatrix] (m)
      {
      \X'' & \ca Y'' \\
        \X' & \mathcal{Y}'\\
       \X & \mathcal{Y} \\
      };
      \draw[->] (m-1-1) to node[above]{$f''$} (m-1-2);
      \draw[ ->] (m-1-1) to node[left]{$h'$}(m-2-1);
      \draw[->] (m-1-2)to node[right]{$h$} (m-2-2);
      \draw[->] (m-2-1) to node[below]{$f'$} (m-2-2);
       \draw[ ->] (m-2-1) to node[left]{$g'$}(m-3-1);
      \draw[->] (m-2-2)to node[right]{$g$} (m-3-2);
      \draw[->] (m-3-1) to node[below]{$f$} (m-3-2);
      
     \end{tikzpicture}
     \end{center}
     then for all $\alpha \in \CH_m^{\T_S}(\ca Y'')$, we have that
     \[
     c_g^m\left(h_*(\alpha)\right) = h_*'c_{gh}^m(\alpha)
     \]
     in $\CH^{\T_S}_{m -k}(\X')$. 
     \item[(ii)] If $h \colon \ca Y'' \to \ca Y$ is an equivariant and flat map of $\T_S$-schemes of relative dimension $\ell$, and $g \colon \ca Y' \to \ca Y$ is any equivariant map, and one forms the diagram as above, then for all $\alpha \in \CH_m\ca Y'$ we have that 
     \[
     c_{gh}^{m + \ell}\left(h^*\alpha\right) = {h'}^*c_g^m(\alpha)
     \]
     in $\CH_{m + \ell - k}^{\T_S}(\X'')$. 
     \item[(iii)] If $g \colon \ca Y' \to \ca Y$ and $h \colon \ca Y \to \ca Z'$ are equivariant morphisms of $\T_S$-schemes, and $\iota \colon \ca Z'' \to \ca Z'$ is an equivariant regular embedding of codimension $e$ lf $\T_S$-schemes, and one forms the fiber diagram 
    \begin{center}
    \begin{tikzpicture}
      \matrix[dmatrix] (m)
      {
      \X'' & \ca Y''& \ca Z'' \\
        \X' & \mathcal{Y}' & \ca Z'\\
        \X & \ca Y & \\
      };
      \draw[->] (m-1-1) to node[above]{$f''$} (m-1-2);
       \draw[->](m-1-2) to node[above]{$h'$} (m-1-3);
      \draw[ ->] (m-1-1) to node[left]{$\iota''$}(m-2-1);
      \draw[->] (m-1-2)to node[right]{$\iota'$} (m-2-2);
      \draw[->] (m-2-1) to node[below]{$f'$} (m-2-2);
       \draw[ ->] (m-2-1) to node[left]{$g'$}(m-3-1);
      \draw[->] (m-2-2)to node[right]{$g$} (m-3-2);
      \draw[->] (m-3-1) to node[below]{$f$} (m-3-2);
      \draw[->] (m-2-2) to node[below]{$h$}(m-2-3);
      \draw[->] (m-1-3) to node[right]{$\iota$}(m-2-3);
      
     \end{tikzpicture}
     \end{center}
     then for all $\alpha \in \CH_m^{\T_S}(\ca Y')$, we have that 
     \[
     \iota^!c_g^m(\alpha) = c_{g\iota'}^{m -e}\left(\iota'\alpha\right)
     \]
     in $\CH^{\T_S}_{m-k-e}(\X'')$.

 \end{enumerate}
Conditions (i)--(iii) are the equivariant analogues of \cite[Properties $C_1$--$C_3$]{fulint}. for $\T_S$-schemes over $S$.
 \begin{exa} 
 \begin{enumerate}
 \item If $f \colon \mathfrak{Y} \to \mathfrak{Z}$ is an equivariant, flat morphism of $\T_S$ schemes with fiber dimension $d$, there is a canonical \emph{orientation class} 
 \[
 [f] \in \CH^{-d}_{\T_S}(f \colon \mathfrak{Y} \to \mathfrak{Z}).
 \]
 Indeed, for a morphism of $\T_S$-schemes $\X \to \mathfrak{Z}$ and $\alpha \in \CH_k^{\T_S}(\X)$, the pullback $\X \times_{\mathfrak{Z}}\mathfrak{Y} \to \X$ is equivariant and flat and one sets 
 \[
 [f](\alpha) = f^{-1}(\alpha) \in \CH^{\T_S}_{k+d}(\X\times_{\mathfrak{Z}}\mathfrak{Y}).
 \]
 \item If $f \colon \mathfrak{Y} \to \mathfrak{Z}$ is an equivariant l.\,c.\,i.\,of codimension $d$, then there is also a canonical orientation class 
 \[
 [f] \in \CH_{\T_S}^d(f \colon \mathfrak{Y} \to \mathfrak{Z}).
 \]
One can construct such an equivariant class mimicking the construction in the non-equivariant setting from \cite[Chapter~17]{fulint}.
 \end{enumerate}
 \end{exa}
 
 As in \cite[Chapter 17]{fulint}, the most salient functorial properties of equivariant operational Chow groups products,  pushforward and pullback, which we now describe.

 \subsubsection*{Products}
For all equivariant morphisms of $\T_S$-schemes $f \colon \X \to \ca Y$, $g \colon \ca Y \to \ca Z$ and for all integers $k,\ell$, there are homomorphism 
\[
\CH_{\T_S}^k(\X \xrightarrow{f} \ca Y) \times \CH_{\T_S}^{\ell}(\ca Y \xrightarrow{g} \ca Z) \longrightarrow \CH_{\T_S}^{k+\ell}(\X \xrightarrow{gf} \ca Z).
\]
The image of $(c, d)$ is denoted by $c \cdot d$ and is given in the following way. Given an equivariant morphism of $\T_S$-schemes $\ca Z' \to \ca Z$, form the fiber diagram 
\begin{figure}[H]
\begin{center}
    \begin{tikzpicture}
      \matrix[dmatrix] (m)
      {
        \X' & \mathcal{Y}' & \ca Z'\\
       \X & \mathcal {Y} &\ca Z& \\
      };
      \draw[->] (m-1-1) to node[above]{$f'$}  (m-1-2);
       \draw[->] (m-1-2) to node[above]{$g'$}  (m-1-3);
      \draw[ ->] (m-1-1)to (m-2-1);
       \draw[ ->] (m-1-2)to (m-2-2);
        \draw[ ->] (m-1-3)to (m-2-3);
      \draw[->] (m-2-1) to node[below]{$f$} (m-2-2);
       \draw[->] (m-2-2) to node[below]{$g$} (m-2-3);

     \end{tikzpicture}
     \end{center}\caption{Products}\label{fig:prod}
     \end{figure}
     
     If $\alpha \in \CH^{\T_S}_m(\ca Z')$ then $d(\alpha) \in \CH^{\T_S}_{m-\ell}(\ca Y')$ and $c(d(\alpha)) \in \CH^{\T_S}_{m-\ell-k}(\X')$. Then $c\cdot d$ is given by $c\cdot d(\alpha) = c(d(\alpha))$. 
     
     \subsubsection*{Proper pushforward}
If $f \colon \X \to \ca Y$ is an equivariant proper morphism of $\T_S$-schemes, $g \colon \ca Y \to \ca Z$ any equivariant morphism of $\T_S$-schemes and $k$ an integer, then there is a homomorphism 
\[
f_* \colon \CH_{\T_S}^k(X \xrightarrow{gf} \ca Z) \longrightarrow \CH_{\T_S}^k(\ca Y \xrightarrow{g} \ca Z).
\]
Given an equivariant morphism of $\T_S$-schemes $\ca Z' \to \ca Z$ form the fiber diagram in Figure \ref{fig:prod}. Let $c \in \CH_{\T_S}^k(X \xrightarrow{gf} \ca Z)$ and $\alpha \in \CH^{\T_S}_{m}(\ca Z')$. Then $c(\alpha) \in \CH^{\T_S}_{m-k}(\X')$ is given by $f_*(c)(\alpha) = f_*'(c(\alpha))$. 

\subsubsection*{Pullback}
Given equivariant morphisms of $\T_S$-schemes $f \colon \X \to \ca Y$ and $g \colon \ca Y_1 \to \ca Y$ form the fiber square 

\begin{center}
    \begin{tikzpicture}
      \matrix[dmatrix] (m)
      {
        \X_1 & \mathcal{Y}_1\\
       \X & \mathcal{Y} \\
      };
      \draw[->] (m-1-1) to node[above]{$f_1$}  (m-1-2);
      \draw[ ->] (m-1-1)to  (m-2-1);
      \draw[->] (m-1-2)to node[right]{$g$} (m-2-2);
      \draw[->] (m-2-1) to node[below]{$f$} (m-2-2);
      
     \end{tikzpicture}
     \end{center}
     Then for each $k$ there is a homomorphism 
     \[
     g^* \colon \CH_{\T_S}^k(\X \xrightarrow{f}\ca Y) \longrightarrow \CH_{\T_S}^k(\X_1 \xrightarrow{f_1}\ca Y_1).
     \] 
     Given $c \in \CH_{\T_S}^k(\X \xrightarrow{f} \ca Y)$, $\ca Y' \to \ca Y_1$ and $\alpha \in \CH^{\T_S}_m(\ca Y')$, then composing with $g$ gives an equivariant morphism of $S$-schemes $\ca Y' \to \ca Y$. Hence $c(\alpha) \in \CH^{\T_S}_{m-k}(\X')$, $\X' = \X \times_{\ca Y}\ca Y' = \X_1 \times_{\ca Y_1}\ca Y'$. Let $g^*(c)(\alpha) = c(\alpha)$. 

\begin{rem}\label{rem:axioms-ful}
Given the functorial properties of the equivariant operational Chow groups, one verifies that the operations above satisfy the equivariant analogue of the seven axioms $(A_1)$--$(A_{123})$ from \cite[Section 17.2]{fulint}.  In particular, we have 
\begin{enumerate}
\item[$(A_{13}$)] \emph{Product and pullback commute}. Let $c \in \CH^*_{\T_S}\left(\X \xrightarrow{f} \ca Y\right)$,  $d \in \CH^*_{\T_S}\left(\ca Y \xrightarrow{h} \ca Z\right)$ and $g \colon \ca Z_1 \to \ca Z$ an equivariant morphism of $\T_S$-schemes, from the fiber diagram 
\begin{center}
    \begin{tikzpicture}
      \matrix[dmatrix] (m)
      {
        \X_1 & \mathcal{Y}_1 & \ca Z_1\\
       \X & \mathcal{Y} & \ca Z \\
      };
      \draw[->] (m-1-1) to node[above]{$f'$}  (m-1-2) to node[above]{$h'$} (m-1-3);
      \draw[ ->] (m-1-1)to  (m-2-1);
      \draw[->] (m-1-2)to node[right]{$g'$} (m-2-2);
      \draw[->] (m-2-1) to node[above]{$f$} (m-2-2) to node[above]{$h$} (m-2-3);
      \draw[->] (m-1-3) to node[right]{$g$} (m-2-3);
      
     \end{tikzpicture}
     \end{center}
Then 
\[
g^*(c \cdot d) = g'^*(c) \cdot g^*(d) \in \CH^*_{\T_S}(\X \xrightarrow{h' f'} \ca Z_1).
\] 
\item[$(A_{123}$)] \emph{Projection formula}.  Given a diagram 
\begin{center}
    \begin{tikzpicture}
      \matrix[dmatrix] (m)
      {
        \X' & \mathcal{Y}' & \\
       \X & \mathcal{Y} & \ca Z \\
      };
      \draw[->] (m-1-1) to node[above]{$f'$}  (m-1-2) ;
      \draw[ ->] (m-1-1)to node[left]{$g'$} (m-2-1);
      \draw[->] (m-1-2)to node[right]{$g$} (m-2-2);
      \draw[->] (m-2-1) to node[above]{$f$} (m-2-2) to node[above]{$h$} (m-2-3);

     \end{tikzpicture}
     \end{center}
     with $g$ equivariant and proper,  the square a fiber square and $c \in \CH^*_{\T_S}(\X \rightarrow \ca Y)$, $d \in \CH^*_{\T_S}(\ca Y' \to \ca Z)$, then 
     \[
     c \cdot g_*(d) = g_*'(g^*(c)\cdot d) \in \CH^*_{\T_S}(\X \to \ca Z).
     \]
\end{enumerate}
\end{rem}

\begin{Def}
For any $\T_S$-scheme $\X$ and any integer $k$, the \emph{equivariant operational $k$'th cohomology group} $\CH_{\T_S}^k(\X)$ is defined by 
\[
\CH_{\T_S}^k(\X) \coloneqq \CH_{\T_S}^k(\X \xrightarrow{\on{id}} \X).
\]
\end{Def}
In other words, an element in $\CH_{\T_S}^k(\X)$ can be seen as a rule which associates to each equivariant morphism of $\T_S$-schemes $g\colon \mathcal{X}'\to\mathcal{X}$ and to each $\alpha\in\CH^{\T_S}_m(\mathcal{X}')$ a class~$c_g(\alpha)\in\CH_{m-k}(\mathcal{X}')$, satisfying the above mentioned compatibility properties (i)--(iii).

Set 
\[
\CH_{\T_S}^*(\X) \coloneqq \bigoplus_{k \in \Z}\CH_{\T_S}^k(\X).
\] 

\begin{itemize}
\item there are \emph{cup products} 
\[
\CH_{\T_S}^k(\mathcal{X}) \times \CH_{\T_S}^{\ell}(\mathcal{X}) \longrightarrow \CH_{\T_S}^{k+\ell}(\mathcal{X})
\]
given by 
\[
(c,d) \longmapsto c \cup d \coloneqq c\cdot d.
\]
\item There is an element $1 \in \CH_{\T_S}^0(\X)$ acting as the identity on $\CH_{\T_S}^k(\X)$ for any $k$.

This makes $\CH_{\T_S}^*(\mathcal{X})$ into an associative graded ring with one. Moreover, given that we have tensored with $\Q$, this ring is moreover commutative. This follows from the commutativity in the non-equivariant setting, which can be deduced from \cite[p.~335 and Proposition 3.8]{KT} and \cite[Example 17.3.2]{fulint}.
\item For any morphism of $\T_S$-schemes $g \colon \mathcal{X}' \to \mathcal{X}$, the pullback $g^* \colon \CH_{\T_S}^*(\mathcal{X}) \to \CH_{\T_S}^*(\mathcal{X}')$ is a ring homomorphism, functorial in $g$.  This follows from Axiom $(A_{13})$ (see Remark \ref{rem:axioms-ful}).
\item There are \emph{cap products}
\[
\CH_{\T_S}^k(\mathcal{X}) \times \CH^{\T_S}_m(\mathcal{X}) \longrightarrow \CH^{\T_S}_{m-k}(\mathcal{X}) 
\]
given by
\[
c \times \alpha \longmapsto c \cap \alpha \coloneqq c(\alpha).
\]
This makes $\CH^{\T_S}_*(\mathcal{X})$ into a left $\CH_{\T_S}^*(\mathcal{X})$-module.
\item Let $f \colon \mathcal{X}' \to \mathcal{X}$ be an equivariant proper morphism of $\T_S$-schemes. The projection formula (see Remark \ref{rem:axioms-ful})
\[
f_*\left(f^* \beta \cap \alpha \right) = \beta \cap f_*(\alpha)
\]
for $\alpha \in \CH^{\T_S}_*(\mathcal{X}')$ and $\beta \in \CH_{\T_S}^*(\mathcal{X})$ is a homomorphism of $\CH^*_{\T_S}(\X)$-modules. 
\end{itemize}

We start with the following Poincaré duality statement for regular proper $\T_S$-schemes.
Recall that we have denoted $\widetilde{S} = \CH_*^{\T_S}$. 

\begin{theorem}\label{th:dual2}
Let $\X$ be a regular, proper $\T_S$-scheme of (absolute) dimension $n+1$. Assume that the equivariant Chow homology groups are generated (as $\widetilde{S}$-modules) by the classes of finitely many invariant cycles. Then for any integer $k$, the natural map 
\begin{eqnarray}\label{equ}
\alpha\colon \CH^k_{\T_S}(\X) \longrightarrow \CH^{\T_S}_{n+1-k}(\X)
\end{eqnarray}
given by 
\[
c \longmapsto c[\X]
\]
is an isomorphism. In particular, the result holds for regular proper toric schemes over $S$.
\end{theorem}
\begin{proof}
We expand and adapt the arguments given in \cite[pg.~335]{KT} to the equivariant setting.
Let $f\colon \mathfrak{Y} \to \X$ be an equivariant morphism of $\T_S$-schemes. One says that $f$ is \emph{equivariently oriented} if it is equipped with a class $[f] \in \CH_{\T_S}^*( \mathfrak{Y} \xrightarrow{f} \X)$ such that $[f]\left([\X]\right) = [\mathfrak{Y}]$.

By Proposition \ref{prop:orienting} we may assume that any equivariant projective map $f \colon V \to \X$ of $\T_S$-schemes is equivariantly oriented.  
Using the equivariant Chow's lemma in \cite[Corollary 5.3.7]{Brion-lin}, we may further assume that any equivariant map to $\X$ has an orientation.

Now, define a $\widetilde{S}$-linear map
\[
\beta \colon \CH^{\T_S}_{n+1-k}(\X) \longrightarrow \CH^k_{\T_S}(\X)
\]
by sending the class $[V]$ of an invariant cycle $V$ in $\X$ to 
\[
{j_V}_*[j_V] \in \CH^k_{\T_S}(\X),
\]
where $j_V \colon V \hookrightarrow \X$ denotes the equivariant closed immersion and $[j_V] \in \CH^k_{\T_S}(j_V \colon V \to \X)$ an orientation class, which exists by assumption. 

We show that $\beta$ is well-defined. Suppose that $[V] = [V']$ in $\CH^{\T_S}_{n+1-k}(\X)$. 

We have to show that 
\[
{j_V}_*[j_V] = {j_{V'}}_*[j_{V'}]
\]
in $\CH^k_{\T_S}(\X)$. Let thus $h \colon \ca Z \to \X$ been equivariant map of $\T_S$-schemes. Using the projection formula we compute
\[
{j_V}_*[j_V][\ca Z] = {j_V}_*[j_V]h^*[\X] = h^*{j_V}_*[j_V][\X] = h^*[j_V] = h^*[j_{V'}] = \dotsc = {j_{V'}}_*[j_{V'}][\ca Z],
\]
which show that $\beta$ is well-defined. By construction its inverse map is $\alpha$. Hence the map \eqref{equ} this proving the statement in the case that any projective equivariant morphism has an orientation class.
The result then follows from the following Proposition. 
\end{proof}
\begin{prop}\label{prop:orienting}
Let $\X$ be a regular, proper $\T_S$-scheme of (absolute) dimension $n+1$. Then any projective equivariant morphism $f \colon V \to \X$ of $\T_S$-schemes is equivariantly oriented with orientation class $[f] \in \CH_{\T_S}^*(f \colon \mathfrak{Y} \to \X)$.
\end{prop}
\begin{proof}

We have to show that for any such $f$ there exists an orientation class $[f] \in \CH^*_{\T_S}(f \colon V \to \X)$ such that $[f]\left([\X]\right) = [V]$. We factor the equivariant projective morphism $f$ through its scheme theoretic image $f(V)$, so we get a diagram of equivariant morphisms of $\T_S$-schemes 
\begin{center}

    \begin{tikzpicture}
 \matrix[dmatrix] (m)
      {
       & V\\
       \X & f(V)\\
      };
      \draw[->] (m-1-2) to node[above]{$f$}  (m-2-1);
      \draw[ ->] (m-2-2)to  (m-2-1);
      \draw[->] (m-1-2)to node[right]{$\tilde{f}$} (m-2-2);
        \end{tikzpicture}
     \end{center}
with $\tilde{f}\colon V \to f(V)$ equivariant, dominant and projective, and $f(V) \to \X$ an equivariant closed embedding. By \cite[IV Theorem 6.9.1]{EGA} we know that there is a non-empty subset $U \subseteq f(V)$ such that 
     \[
     \tilde{f}|_{\tilde{f}^{-1}(U)} \colon \tilde{f}^{-1}(U) \longrightarrow U
     \]
     is flat. With this, we use the ``equivariant flattening technique'' (see \cite[Theorem 8]{TV}) to extend the above diagram to
\begin{center}
    \begin{tikzpicture}
      \matrix[dmatrix] (m)
      {
       & V& Z\\
       \X & f(V) & W\\
      };
      \draw[->] (m-1-2) to node[above]{$f$}  (m-2-1);
      \draw[ ->] (m-2-2)to  (m-2-1);
      \draw[->] (m-1-2)to node[right]{$\tilde{f}$} (m-2-2);
        \draw[->] (m-1-3)to node[right]{$h$} (m-2-3);
          \draw[->] (m-2-3)to node[below]{$g$} (m-2-2);
            \draw[ ->] (m-1-3)to  (m-1-2);

     \end{tikzpicture}
     \end{center}
     where $h$ is an equivariant and flat morphism of $\T_S$-schemes, and $g$ an equivariant $U$-admissible blow up, i.e. $g$ is an equivariant morphism of $\T_S$-schemes given as a blow up along a $\T_S$-stable closed subscheme $Z$ with $Z \cap U = \emptyset$. In particular, $W$ is integral and $g$ is an equivariant, birational and projective morphism of $\T_S$-schemes. 
     
       We now use equivariant semi-stable reduction over $S$, for which we refer to \cite{lore-thesis} and we
 
       find an equivariant projective morphism of $\T_S$-schemes 
     \[
     \phi \colon W' \to W',
     \]
     (possibly after a base extension $S_1 \to S$), generically finite of degree $d$, with $W'$ regular. Now we expand further the diagram to
\begin{center}
    \begin{tikzpicture}
      \matrix[dmatrix] (m)
      {
       & V& Z & Z'\\
       \X & f(V) & W & W'\\
      };
      \draw[->] (m-1-2) to node[above]{$f$}  (m-2-1);
      \draw[ ->] (m-2-2)to  (m-2-1);
      \draw[->] (m-1-2)to node[right]{$\tilde{f}$} (m-2-2);
        \draw[->] (m-1-3)to node[right]{$h$} (m-2-3);
          \draw[->] (m-2-3)to node[below]{$g$} (m-2-2);
          \draw[->] (m-2-4)to node[below]{$\phi$} (m-2-3);
            \draw[ ->] (m-1-3)to  (m-1-2);
              \draw[ ->] (m-1-4)to  (m-1-3);
                \draw[ ->] (m-2-4)to  (m-2-3);
                \draw[->] (m-1-4)to node[right]{$f'$} (m-2-4);
                  \draw[->, bend right =25] (m-1-4) to node[above]{$\omega'$}  (m-1-2);
                    \draw[->, bend left =25] (m-2-4) to node[below]{$\omega$}  (m-2-1);

     \end{tikzpicture}
     \end{center}
     with maps as labeled. Then all morphism are equivariant, $\omega$ is projective, $f'$ is flat and $\omega'$ is proper and generically finite of degree $d$. Moreover, since $W'$ and $\X$ are regular, and since equivariant projective morphisms between regular schemes are equivariant local complete intersections (l.c.i.)we have that $\omega$ is an equivariant l.c.i. morphism. Hence it has an equivariant canonical orientation class $[\omega] \in \CH_{\T_S}^*(\omega \colon W' \to \X)$. The equivariant flat morphism $f'$ induces also a canonical orientation class $[f'] \in \CH_{\T_S}^*\left(f' \colon Z' \to W'\right)$. Consider now the class
     \[
     [f] \coloneqq \frac{1}{d}\omega'_*\left([f']\cdot [\omega]\right) \in \CH^*_{\T_S}(f \colon V \to \X).
     \]
     We check that $[f]$ is an orientation class:
\begin{align*}
[f]\left([\X]\right) =& \left(\frac{1}{d}\omega'_*\left([f']\cdot [\omega]\right)\right)\left([\X]\right)\\
=& \frac{1}{d}\omega'_*[f'][\omega]\left([\X]\right) \\
=& \frac{1}{d}\omega_*'[f']\left([W']\right) \\
=& \frac{1}{d}\omega'_*\left[Z'\right] \\
=& \frac{1}{d}\cdot d \cdot [V] = [V].
\end{align*}
 
\end{proof}

Let $f \colon \mathfrak{Y} \to \mathfrak{X}$ be an equivariantly oriented map of $\T_S$-schemes with orientation $[f] \in \CH_{\T_S}^{\ell}(f \colon \mathfrak{Y} \to \mathfrak{X})$. We denote by  
     \[
     f^! \colon \CH_k^{\T_S}(\mathfrak{X}) \longrightarrow \CH_{k-\ell}^{\T_S}(\mathfrak{Y})
     \]
the induced map on equivariant Chow homolog. If $f$ is moreover proper, we write
     \[
    f_! \colon \CH_{\T_S}^k(\mathfrak{Y}) \longrightarrow \CH_{\T_S}^{k+\ell}(\mathfrak{X})
    \]
  for the induced map on equivariant Chow cohomology. 
\begin{exa}\label{exa:orient}
For a toric scheme $\X$ over $S$, the special fiber $\iota\colon \X_s \hookrightarrow \X$ is a $\T_S$- invariant principal Cartier divisor so an orientation $[\iota] \in \CH_{\T_S}^1(\iota \colon \X_s \to \X)$ is defined. 

Moreover, given a map of toric models $\pi \colon \X' \to \X$ and the corresponding commutative diagram 

\begin{center}
    \begin{tikzpicture}
      \matrix[dmatrix] (m)
      {
        \X_s' & \X'\\
       \X_s & \X\\
      };
      \draw[->] (m-1-1) to node[above]{$\iota'$}  (m-1-2);
      \draw[ ->] (m-1-1)to node[left]{$\pi_s$}  (m-2-1);
      \draw[->] (m-1-2)to node[right]{$\pi$} (m-2-2);
      \draw[->] (m-2-1) to node[below]{$\iota$} (m-2-2);
      
     \end{tikzpicture}
     \end{center}
     all four maps are equivariantly orientable, i.e. there are orientation classes 
     \[
     [\pi] \in \CH^0_{\T_S}(\X' \to \X)\; , \; [\pi_s] \in \CH^0_{\T_S}(\X_s' \to \X_s).
     \]
These satisfy
     \begin{eqnarray*}
    & \iota^*[\pi]= [\pi_s] \in \CH^0_{\T_S}(\pi_s\colon \X_s'\to \X_s) \\
     & [\pi \iota'] = [\iota \pi_s] = [\pi_s]\cdot [\iota] = [\iota'] \cdot [\pi] \in \CH^1_{\T_s}(\X_s' \to \X).
     \end{eqnarray*}
    
     \end{exa}

\subsection{Equivariant operational Chow groups of toric schemes}\label{sec:op-toric}
We extend Theorem \ref{th:chow-payne} to the case of toric schemes over $S$ and we describe combinatorially the Poincaré duality map in Theorem \ref{th:dual2} in this case.

Fix a complete, regular SCR polyhedral complex $\Pi$ in $N_{\R}$ and let $\X = \X_{\Pi}$ the corresponding complete, regular toric scheme over $S$ (absolute) dimension $n+1$. 

Consider the conical polyhedral complex $c(\Pi)$ in $N_{\R} \oplus \R_{\geq 0}$  as in Section \ref{sec:toric-schemes}. Further, consider the ring of rational piecewise polynomial functions $PP^*(c(\Pi))$ on $c(\Pi)$ as in Definition \ref{def:piecewisepoly}.

\begin{theorem}\label{th:equi-pol}
There is an isomorphism (as $\widetilde{S}$-modules)
\[
\CH^*_{\T_S}(\X) \simeq PP^*(c(\Pi)).
\]
\end{theorem}
\begin{proof}
From the Poincaré duality in Theorem \ref{th:dual2} we know that capping with the fundamental class gives an isomorphism
\[
\CH^k_{\T_S}(\X) \longrightarrow \CH_{n+1-k}^{\T_S}(\X).
\]
Hence, $\CH_*^{\T_S}(\X)$ inherits a graded algebra structure. Note that the grading is such that $\CH^k_{\T_S}(\X)$ vanishes for $k < 0$. It thus suffices to give an $\widetilde{S}$-module isomorphism
\[
g \colon \CH_*^{\T_S}(\X) \longrightarrow PP^*(c(\Pi)).
\]

This  was shown in the case of simplicial toric varieties over a field by Brion in \cite[Theorem 5.4]{Brion-equi} and the proof extends to the case of toric schemes over $S$ using the regularity of the fan $c(\Pi)$. We sketch the ideas. Let $\sigma \in c(\Pi)$. Let $\widetilde{N}_{\sigma}$ be the subgroup of $\widetilde{N}$ generated by $\widetilde{N} \cap \sigma$ and denote by $\Omega_{\sigma}$ the field of rational functions on the rational vector space $(\widetilde{N}_{\sigma})_{\Q}$ with rational coefficients, i.e. the quotient field of $PP\left((\widetilde{N}_{\sigma})_{\Q}\right)$. One defines an $\widetilde{S}$-linear map 
\begin{eqnarray}\label{eq:equi-mul}
e_{\sigma}\colon \CH_*^{\T_S}(\X) \longrightarrow \Omega_{\sigma}.
\end{eqnarray}
It is defined by Brion in \cite[Section 4.2]{Brion-equi} in the case of toric varieties over a field using the notion of \emph{equivariant multiplicities}. One can extend this notion to toric schemes over $S$ and the map in \eqref{eq:equi-mul} is the one induced by this notion. It is the $\widetilde{S}$-linear map
determined by the following properties (recall that $\CH_*^{\T_S}(\X)$ is generated as an $\widetilde{S}$-module by classes of invariant cycles):  let $\tau \in c(\Pi)$. Then
\begin{enumerate}
\item if $\tau$ is not contained in $\sigma$, then $e_{\sigma}[V(\tau)] = 0$,
\item if $\tau$ is contained in $\sigma$, then for any $\lambda$ in the relative interior of $\sigma$, the value at $\lambda$ of $e_{\sigma}[V(\tau)]$ is $\left(\dim(\sigma)-\dim(\tau)\right)! \cdot \vol\left( \on{Im}\left(P_{\sigma}(\lambda) \cap \tau^{\perp}\right)\right)$,
where
\[
P_{\sigma}(\lambda) \coloneqq \left\{x \in \sigma^{\vee} \; | \; \langle \lambda, x \rangle \leq 1 \right\}
\]
 and we take the image in $\tau^{\perp}/\sigma^{\perp}$. 
 Here, $\langle \cdot, \cdot \rangle$ denotes the pairing between $\widetilde{N}_{\R}$ and $\widetilde{M}_{\R}$. In particular, $e_{\sigma}[V(\sigma)] = 1$. 
\end{enumerate}

Moreover, one can show that for any $\tau \in c(\Pi)$ the following is satisfied
\[
\left(e_{\sigma}[V(\tau)]/e_{\sigma}[\X]\right)_{\sigma \in c(\Pi)} = \varphi_{\tau},
\]
where $\varphi_{\tau}$ is the homogeneous piecewise polynomial function on $c(\Pi)$ of degree $\dim(\tau)$ from Definition \ref{def:generators}. 

It follows that the map
\[
g \colon \CH_*^{\T_S}(\X) \longrightarrow \prod_{\sigma \in c(\Pi)}\Omega_{\sigma}
\]
given by 
\[
u \longmapsto \left(e_{\sigma}u/e_{\sigma}[\X]\right)_{\sigma \in c(\Pi)}
\]
has image the $\widetilde{S}$-module generated by the $\varphi_{\sigma}$, $\sigma \in c(\Pi)$. But this is exactly $PP^*(c(\Pi))$ by Proposition \ref{prop:generators}.  It remains to show that $g$ is injective. For this, we follow the arguments in \cite[Proposition 5.3]{Brion-equi} and prove the injectivity by induction over the number of cones in $c(\Pi)$. Choose a cone $\sigma$ in $c(\Pi)$. We denote by $O(\sigma)$ the corresponding torus orbit. Consider the commutative diagram 
\begin{center}
    \begin{tikzpicture}
      \matrix[dmatrix] (m)
      {
       &  \CH_*^{\T_S}(O(\sigma))&  \CH_*^{\T_S}(\X) & \CH_*^{\T_S}(\X\setminus O(\sigma)) & 0\\
      0 & \Omega_{\sigma} & \prod_{\tau \in c(\Pi)}\Omega_{\tau} & \prod_{\tau \in c(\Pi)\setminus \{\sigma\}}\Omega_{\tau} & 0\\
      };
      \draw[->] (m-1-2) to (m-1-3);
      \draw[ ->] (m-1-3)to  (m-1-4);
      \draw[->] (m-1-4)to (m-1-5);
        \draw[->] (m-2-1) to (m-2-2);
      \draw[ ->] (m-2-2)to  (m-2-3);
      \draw[->] (m-2-3)to (m-2-4);
       \draw[->] (m-2-4)to (m-2-5);
        \draw[->] (m-1-3)to node[right]{$g$} (m-2-3);
          \draw[->] (m-1-4)to (m-2-4);
      
     \end{tikzpicture}
     \end{center}
     which induces a map $\CH_*^{\T_S}(O(\sigma)) \to \Omega_{\sigma}$. This map is $\widetilde{S}$-linear and sends $\left[O(\sigma)\right] = \left[V(\sigma)\right]$ to $e_{\sigma}\left[V(\sigma)\right]=1$. Moreover, $\CH_*^{\T_S}(O(\sigma))$ is the symmetric algebra over $\widetilde{M}(\sigma)$, therefore the map $\CH_*^{\T_S}(O(\sigma)) \to \Omega_{\sigma}$ identifies with the inclusion of this symmetric algebra into its quotient field $\Omega_{\sigma}$. Now it follows by induction that $g$ is injective. 

Hence,  we obtain an isomorphism of $\widetilde{S}$-modules
\[
g \colon \CH_*^{\T_S}(\X) \longrightarrow PP^*(c(\Pi))
\]
as claimed.

\end{proof}
\begin{cor}
Let $\sigma \in c(\Pi)$ of dimension $k$. The Poincaré duality isomorphism
\[
\CH_{n+1-k}^{\T_S}(\X) \simeq \CH^k_{\T_S}(\X)
\]
identifies $\left[V(\sigma)\right]$ with the piecewise polynomial function $\varphi_{\sigma}$.
\end{cor}

\section{Equivariant Chow groups of the special fiber}\label{sec:special}
As before, we fix a complete, regular SCR polyhedral complex $\Pi$ in $N_{\R}$ and let $\X = \X_{\Pi}$ be the corresponding complete, regular toric scheme over $S$ of $S$-absolute dimension $n+1$. 
The goal of this section is to provide combinatorial descriptions of the equivariant Chow homology and of the equivariant operational Chow cohomology groups of the special fiber $\X_s$.  We view $\X_s$ as a $\T_{\kappa}$-scheme over $\kappa$.  It turns out that these can be computed in terms of equivariant Chow cohomology groups of smooth  toric varieties over $\kappa$ via an exact sequence.
\subsection{The exact sequences}
We give a combinatorial description of the exact sequences computing the equivariant Chow groups of the special fiber. These are the equivariant analogues of the exact sequences in \cite[Appendix A]{BGS}. 

For any natural number $\ell$ we write $\Pi(\ell)$ for the set of $\ell$-dimensional polyhedra in $\Pi$. We set $r = |\Pi(0)|$ and recall that this equals the number of irreducible components of $\X_s$. 

For any $i \in \{1, \dotsc, r\}$ let $\X_s^{[i]}$ be the disjoint union of $i$-fold intersections of irreducible components of $\X_s$. Set also $\X_s^{[0]}=\X$ and $\X_s^{[r]}= \emptyset$ if $r >n$. 

We will mostly only consider the cases $i = 0,1,2$. Using the correspondence between vertical cycles and polyhedra in $\Pi$ we can write
\[
\X_s^{[1]} = \bigsqcup_{v \in \Pi(0)}V(v) \quad \text{ and }\quad \X_s^{[2]} = \bigsqcup_{\stackrel{\gamma \in \Pi(1)}{\gamma \text{ bounded}}}V(\gamma).
\]
Recall that $V(\zeta)$ denotes the vertical cycle corresponding to $\zeta \in \Pi$, which is regular since $\Pi$ is regular.

Let $v, \gamma \in \Pi$ with $v \prec \gamma $ and $v$ of codimension one in $\gamma$. Consider the inclusion of the corresponding toric varieties
\[
u_{v, \gamma}\colon V(\gamma) \hookrightarrow V(v).
\]
We have group homomorphisms 
\[
u_{v,\gamma}^*\colon \CH_{\T_{\kappa}}^k\left(V(v)\right) \longrightarrow \CH_{\T_{\kappa}}^k\left(V(\gamma)\right)
\]
and (using Poincaré duality)
\[
{u_{v,\gamma}}_*\colon \CH_{\T_{\kappa}}^k(V\left(\gamma)\right) \longrightarrow \CH_{\T_{\kappa}}^{k+1}\left(V(v)\right).
\]

The morphisms $u_{v,\gamma}^*$ are ring homomorphisms and the following projection formula holds: 
\[
{u_{v,\gamma}}_*\left(u_{v,\gamma}^*(x) \cdot y \right) = x \cdot {u_{v,\gamma}}_*(y).
\]
In particular, 
\[
{u_{v,\gamma}}_*u_{v,\gamma}^*(x) = x \cdot {u_{v,\gamma}}_*(1), 
\]
and, using \cite[Proposition 2.6 c)]{fulint} we get 
\[
u_{v,\gamma}^*{u_{v,\gamma}}_*(x) 
=  x \cdot u_{v,\gamma}^*{u_{v,\gamma}}_*(1).
\]
For any $\zeta \in \Pi$, recall that $V(\zeta)$ is a toric variety over $\kappa$ with fan given by $\Pi(\zeta)$ (see Section \ref{subsec:orbits}).  We view $\Pi(\zeta)$ as a generalized fan in $N_{\R}$ with linear subspace $L(\zeta)$ and cones corresponding to the polyhedra in $\Pi$ containing $\zeta$.

Then it follows from Section \ref{sec:toric} that we can identify $\CH_{\T_{\kappa}}^k\left(V(\zeta)\right)$ with the ring of piecewise polynomial functions $PP^k(\Pi(\zeta))$.

With this identification we can describe the maps $u_{v,\gamma}^*$ and ${u_{v,\gamma}}_*$ combinatorially using Definitions \ref{def:pullback}, \ref{def:push} and Example \ref{exa:push-pull}.  
For example we have that
\[
{u_{v,\gamma}}_* \colon PP^k(\Pi(\gamma)) \longrightarrow PP^{k+1}(\Pi(v)) 
\]
sends 
\begin{eqnarray}\label{eq:proj}
1 \longmapsto \varphi_{v,\gamma},
\end{eqnarray}
where $\varphi_{v,\gamma}$ denotes the piecewise linear function on $\Pi(v)$ corresponding to $\gamma$ from Proposition \ref{prop:generators}. 

We now define morphisms 
\[
\rho \colon \CH_{\T_{\kappa}}^k\left(\X_s^{[1]}\right)\simeq \bigoplus_{v \in \Pi(0)}\CH_{\T_{\kappa}}^k\left(V(v)\right)  \longrightarrow \bigoplus_{\stackrel{\gamma \in \Pi(1)}{\gamma \text{ bdd}}}\CH_{\T_{\kappa}}^k\left(V(\gamma)\right) \simeq   \CH_{\T_{\kappa}}^k\left(\X_s^{[2]}\right)
\]

and 

\[
\gamma \colon \CH_{\T_{\kappa}}^k\left(\X_s^{[2]}\right)\simeq \bigoplus_{\stackrel{\gamma \in \Pi(1)}{\gamma \text{ bdd}}}\CH_{\T_{\kappa}}^k\left(V(\gamma)\right) \longrightarrow \bigoplus_{v \in \Pi(0)}\CH_{\T_{\kappa}}^{k+1}\left(V(v)\right)  \simeq   \CH_{\T_{\kappa}}^{k+1}\left(\X_s^{[1]}\right)
\]

in the following way. First we choose an ordering of the vertices $v_1 \geq v_2 \geq \dotsc \geq v_r$. For any bounded $\gamma \in \Pi(1)$ we denote by $v_1^{\gamma} \geq v_2^{\gamma}$ the endpoints of $\gamma$. Then for $f = \left(f_v\right)_{v \in \Pi(0)}$ we set
\[
\rho(f) = \left(\tilde{f}_{\gamma}\right)_{\stackrel{\gamma \in \Pi(1)}{\gamma \text{ bdd}}},
\]
where
\[
\tilde{f}_{\gamma} = u_{v_1^{\gamma},\gamma}^*f_{v_1^{\gamma}}- u_{v_2^{\gamma},\gamma}^*f_{v_2^{\gamma}}.
\]

On the other hand, for $\tilde{f} = \left(\tilde{f}_{\gamma}\right)_{\stackrel{\gamma \in \Pi(1)}{\gamma \text{ bdd}}}$, then $\gamma\left(\tilde{f}\right)$ is defined componentwise by 
\[
\gamma\left(\tilde{f}_{\gamma}\right) = {u_{v_1^{\gamma},\gamma}}_*\tilde{f}_{\gamma}-{u_{v_2^{\gamma},\gamma}}_*\tilde{f}_{\gamma}.
\]
\begin{exa} Denote by $\mu_i \colon \Pi(\gamma) \to \Pi\left(v_i^{\gamma}\right)$ the morphism of generalized fans corresponding to the inclusion $v_i^{\gamma} \prec \gamma$, for $i =1,2$. 
\begin{itemize}
\item Let $f = \left(f_v\right)_{v \in \Pi(0)} \in \CH_{\T_{\kappa}}^k\left(\X_s^{[1]}\right)$. Then for $\sigma \in \Pi(\gamma)$, the restriction of the piecewise polynomial function $(\rho(f))_{\gamma}$ to $\sigma$ is given by the polynomial function
\[
 \left(f_{v_1^{\gamma}}\right)_{\mu_1(\sigma)} -  \left(f_{v_2^{\gamma}}\right)_{\mu_2(\sigma)}.
\]
\item Let $\tilde{f} = \left(f_{\gamma}\right)_{\stackrel{\gamma \in \Pi(1)}{\gamma \text{ bdd}}}$ for $f_{\gamma}= \varphi_{\sigma_{\gamma}}$ with $\sigma_{\gamma} \in \Pi(\gamma)$ of dimension $\dim(L(\gamma))+k$. Then 
\[
\gamma\left(\tilde{f}_{\gamma}\right) = \gamma\left(\varphi_{\sigma_{\gamma}}\right) = \varphi_{\mu_1(\sigma_{\gamma})}-\varphi_{\mu_2(\sigma_{\gamma})}.
\]
\end{itemize}
\end{exa}

\begin{prop}\label{prop:exact-seq}
The equivariant Chow cohomology and homology groups of the special fiber $\X_s$ fit into the following exact sequences 
\[
\CH_k^{\T_{\kappa}}\left(\X_s^{[2]}\right) \xrightarrow{\gamma} \CH_k^{\T_{\kappa}}\left(\X_s^{[1]}\right) \to \CH_k^{\T_{\kappa}}\left(\X_s\right) \to 0
\]
and 
\[
0 \to \CH^k_{\T_{\kappa}}\left(\X_s\right) \to \CH^k_{\T_{\kappa}}\left(\X_s^{[1]}\right) \xrightarrow{\rho} \CH^k_{\T_{\kappa}}\left(\X_s^{[2]}\right).
\]
In particular, we get 
\begin{eqnarray}\label{eq:coho}
\CH_{\T_{\kappa}}^k(\X_s)\simeq \on{ker}(\rho) 
\end{eqnarray}
and 
\begin{eqnarray}\label{eq:hom}
\CH_k^{\T_{\kappa}}(\X_s) \simeq \on{coker}(\gamma).
\end{eqnarray}
\end{prop}
\begin{proof}
This is the equivariant version of \cite[Appendix A]{BGS} which follows using the equivariant localization sequences in \cite[Appendix A]{Gonz} (see in particular Remark~A.4).
\end{proof}
In the following section, we will see that we can identify $\on{ker}(\rho) $ with the ring of affine piecewise polynomial functions on $\Pi$.

\subsection{Rings of affine piecewise polynomial functions}
Let $\Pi$ be a  complete, regular SCR polyhedral complex in $N_{\R}$.  We now define the ring of \emph{affine} piecewise polynomial functions on $\Pi$.
\begin{Def}\label{def:affine-pp} 
An \emph{affine piecewise polynomial function of degree $k$ on $\Pi$} consists of a tuple of piecewise polynomial functions $\left(f_v\right)_{v \in \Pi(0)}$ indexed over the vertices of $\Pi$ subject to the following conditions:
\begin{enumerate}
\item[(i)] $f_v \in PP^k(\Pi(v))$ for any vertex $v$ of $\Pi$, i.e.~$f_v$ is a piecewise polynomial function of degree $k$ with respect to the fan $\Pi(v)$.

For any (full-dimensional) polyhedron $\Lambda \in \Pi$ containing a vertex $v$ in $\Pi$, denote by $f_{v,\Lambda}$ the polynomial in $N_{\R}$ which represents the restriction of $f_v$ to $\Lambda$.
\item[(ii)] For any (full-dimensional) polyhedron $\Lambda \in \Pi$ and any two vertices $v$ and $v'$on $\Lambda$, we have that $f_{v,\Lambda}=f_{v',\Lambda}$ as polynomials in $N_{\R}$.
\end{enumerate}
The set of degree $k$ affine piecewise polynomial functions on $\Pi$ is denoted by $PP^k(\Pi)$.

One can add and multiply piecewise polynomial functions and we get a graded ring structure on $PP^*(\Pi) \coloneqq \bigoplus_{\geq 0}PP^k(\Pi)$.

\end{Def}
\begin{rem} Note that in the above definition the cone $\Lambda$ at $v$ and the cone of $\Lambda$ at $v'$ are different cones. What we require is that the piecewise polynomial function $f_v$ restricted to $\Lambda$ and the piecewise polynomial function $f_{v'}$ restricted to $\Lambda$ are represented by the same polynomial. 
\end{rem}

\begin{rem}\label{rem:piecewise-affine} An affine piecewise polynomial function on $\Pi$ of degree one induces a continuous piecewise affine function on $\Pi$ up to a constant. Indeed, it does not specify a value at any vertex $v$, however, it specifies the slope along any edge containing $v$.  Fixing a value at any vertex determines the values at the other vertices. This is illustrated in the following example.

\end{rem}

\begin{exa}
Let $\Pi$ as in Figure \ref{fig:pp} consisting of two vertices $v_1$ and $v_2$, $5$ polyhedra of dimension $1$ and $4$ of maximal dimension $2$. In the figure we see $\gamma$ and $\gamma'$ of dimension $2$. Its recession cone is the fan of $\mathbb{P}^1 \times \mathbb{P}^1$. 

\begin{figure}[H]
\begin{center}
\begin{tikzpicture}[scale=1]
    \draw (0,0)node[above]{$v_1$} to (-1,0);
    \draw (0,0) -- (0,-1) ;
      \draw (0,0) -- (1,1);
       \draw (1,1) node[above right]{$v_2$} to (2,1);
       \draw (1,1) -- (1,2);
       \draw (0,1.3) node{$\gamma'$};
       \draw (1,0) node{$\gamma$};
       
       \draw(3,0.5) to (4,0.5) node[above]{$v_1$};
         \draw(5,1.5) -- (4,0.5);
          \draw(4,-0.5) -- (4,0.5);
          \draw (4.4,0) node{$\gamma$};
           \draw (3.8,1.3) node{$\gamma'$};
          
            \draw(6.5,0.5)node[above right]{$v_2$} to (7.5,0.5) ;
         \draw(6.5,0.5) -- (5.5,-0.5);
          \draw(6.5,0.5) -- (6.5,1.5);
          \draw(6.7,-0.2) node{$\gamma$};
          \draw (6,0.6) node{$\gamma'$};
 \end{tikzpicture}
\end{center}
\caption{The polyhedral complexes $\Pi$, $\Pi(v_1)$ and $\Pi(v_2)$, respectively}\label{fig:pp}
 \end{figure}

An affine piecewise polynomial function of degree one on $\Pi$ is given by two piecewise polynomial functions of (homogeneous) degree one:
\[
f_{v_1} = \left(f_{v_1,\sigma}\right)_{\sigma \succ v_1} \quad \text{ and }\quad f_{v_2} = \left(f_{v_2,\sigma}\right)_{\sigma \succ v_2} 
\]
 on the fans $\Pi(v_1)$ and $\Pi(v_2)$, respectively The compatibility condition says that 
\[
f_{v_1, \gamma} = f_{v_2, \gamma}
\]
as linear functions on $N_{\R} \simeq \R^2$, i.e. of the form $\ell_{\gamma}(x,y) = ax + by$. And similarly for $\gamma'$, say of the form $\ell_{\gamma'}(x,y) = cx + dy$. The functions $\ell_{\gamma}$ and $\ell_{\gamma'}$, seen as functions on $\Pi$, don't have any specified values at the vertices $v_1$, $v_2$. However, if we fix a value at one vertex, say at $v_1$, then the constant part of  $\ell_{\gamma}$ and $\ell_{\gamma'}$ are determined, hence also the value at $v_2$. 
\end{exa}

The following is the main result of this section.

\begin{theorem}\label{th:cohomology-special}
The equivariant operational Chow ring of the special fiber $\CH_{\T_{\kappa}}^*(\X_s)$ can be identified with the graded ring of affine piecewise polynomial functions $PP^*(\Pi)$.
\end{theorem}
\begin{proof}
By \eqref{eq:coho} it suffices to identify $\on{ker}(\rho)$ with $PP^*(\Pi)$.
Let $f = \left(f_v\right)_{v \in \Pi(0)} \in \on{ker}(\rho)$. For any bounded $\gamma \in \Pi(1)$, as before, we let $v_1^{\gamma} \geq v_2^{\gamma}$ be the endpoints of $\gamma$ and we denote by $\mu_i \colon \Pi(\gamma) \to \Pi(v_i^{\gamma})$ the corresponding morphisms of fans. Then $f$ being in the Kernel of $\rho$ exactly means that for all such $\gamma$ and for all $\sigma \in \Pi(\gamma)$ we have 
\begin{eqnarray}\label{eq:kernel}
\left(f_{v_1^{\gamma}}\right)_{\mu_1(\sigma)} = \left(f_{v_2^{\gamma}}\right)_{\mu_2(\sigma)}.
\end{eqnarray}
This is equivalent to condition (ii) of Definition \ref{def:affine-pp}.

\end{proof}
\begin{rem}
In \cite[Theorem A.9]{Gonz} the author also provides a description of the equivariant Chow ring of the special fiber, which is an example of a \emph{T-skeletal variety}. Indeed, the author identifies the equivariant Chow ring of the special fiber with the space of piecewise polynomial functions given by tuples of polynomial functions $\left(f_{\sigma}\right)_{\sigma\in \Pi(n)}$ (one polynomial for each full-dimensional polyhedron in the complex $\Pi$) subject to the following compatibility condition:


For full-dimensional polyhedra $\sigma_1$ and $\sigma_2$ (not necessarily distinct) containing a facet $\tau$, one has that $f_{\sigma_1} \equiv f_{\sigma_2} \; \operatorname{mod}(\chi_{12})$, where $\chi_{12}$ is the character associated to the torus orbit $O(\tau)$. 

For the proof that this space agree with the equivariant Chow ring,  the author refers to \cite[Theorem 5.4]{Gonz-pol} where the statement is shown in the context of equivariant operational $K$-theory.

By Theorem \ref{th:cohomology-special} we have that $PP(\Pi)$ can be identified with the space of piecewise polynomial functions described above.  
\end{rem}
\begin{rem}\label{rem:bruhat}
Theorem \ref{th:cohomology-special}  has been applied to the study of equivariant Chern classes of toric vector bundles over a DVR in connection to Bruhat--Tits buildings in \cite{BKM}. The idea is that toric vector bundles over a DVR of rank $r$ are classified by piecewise affine maps to the Bruhat--Tits building associated to the general linear group of rank $r$ (See \cite{KMT} for this classification and for terminology regarding Bruhat--Tits buildings.) Then the Chern classes of such a toric vector bundle (defined as the Chern classes of the restriction to the special fiber) are described by the composition of this piecewise affine map and the elementary symmetric functions on the building.
\end{rem}

\section{Equivariant forms and currents}\label{sec:currents}
Let $X = X_{\Sigma}$ be a smooth projective toric variety over the discretely valued field $K$. 
The goal of this section is to define equivariant analogues of the spaces of forms and currents on $X$ appearing in \cite{BGS} and to provide combinatorial descriptions of these spaces.

Recall that $R(\Sigma)$ denotes the directed set of complete, regular SCR polyhedral complexes $\Pi$ in $N_{\R}$ such that $\on{rec}(\Pi) = \Sigma$. This agrees with the directed set of of regular proper toric models of $\X_{\Sigma}$.  For $\Pi \in R(\Sigma)$ we denote by $\X_{\Pi}$ the corresponding toric model. 
\subsection{The map $\iota^*\iota_*$}\label{sec:iota}
Let $\iota\colon \X_s \hookrightarrow \X$ denote the inclusion of the special fiber. The map $\iota^*\iota_*$ between equivariant Chow groups then corresponds to taking the pushforward of $\iota$, then Poincaré duality and then pullback of $\iota$. We give a combinatorial description of this map in terms of piecewise polynomial functions. 

This map will induce maps between spaces of forms and currents analogous to the $dd^c$ map in differential geometry. 

First, consider the map 
\[
\bigoplus_i u_i^* \colon \CH^k_{\T_S}(\X) \longrightarrow \bigoplus_{v_i \in \Pi(0)}\CH^k_{\T_S}(V(v_i)),
\]
where $u_i \colon V(v_i) \hookrightarrow \X$ denotes the inclusion. It is easy to see that the composition
\[
\CH_{\T_S}^k(\X) \xrightarrow{\bigoplus_i u_i^*}\CH_{\T_S}^k(\X_s^{[1]}) \xrightarrow{\rho} \CH_{\T_S}^k(\X_s^{[2]})
\]
is zero.  Similarly, the composition 
\[
\CH_{\T_S}^{k-1}(\X_s^{[2]}) \xrightarrow{\gamma}\CH_{\T_S}^k(\X_s^{[1]}) \xrightarrow{\bigoplus_i {u_i}_*} \CH_{\T_S}^{k+1}(\X)
\]
is zero.
 
 Hence, the maps $\rho$ and $\gamma$ induce morphisms 
 \[
 \iota^* \colon \CH_{\T_S}^k(\X) \longrightarrow \CH_{\T_S}^k(\X_s)
 \]
 and 
 \[
 \iota_* \colon \CH^{\T_S}_{n-k+1}(\X_s)\longrightarrow \CH_{\T_S}^k(\X).
 \]
 \begin{lemma}
The composite map 
 \[
 \CH_{\T_S}^{k-1}(\X_s^{[1]}) \to \CH^{\T_s}_{n-k+1}(\X_s) \xrightarrow{ \iota^*\iota_*}\CH_{\T_S}^k(\X_s) \to \CH_{\T_S}^k(\X_s^{[1]})
 \]
 coincides with $-\gamma \rho$. 
 \end{lemma}
 \begin{proof}
 First, it follows form the definitions that the morphism $\gamma \rho + \rho\gamma$ vanishes on $\CH^k_{\T_S}(\X_s^{[1]})$. On the other hand, the map 
 \[
 \rho \gamma\colon \CH^{k-1}_{\T_S}(\X_s^{[1]}) \to \CH^{k}_{\T_S}(\X_s^{[1]})
 \]
 is, by definition, the composite of morphisms 
 \[
 \CH_{\T_S}^{k-1}(\X_s^{[1]}) \to \CH^{\T_s}_{n-k+1}(\X_s) \xrightarrow{ \iota^*\iota_*}\CH_{\T_S}^k(\X_s) \to \CH_{\T_S}^k(\X_s^{[1]}).
 \]
 \end{proof}
 We can now give an expression of the map $-\gamma\rho$ in terms of piecewise polynomial functions.  
 \begin{lemma}\label{lem:compo}
 Let $f = \left(f_v\right)_{v \in \Pi(0)} \in \CH_{\T_S}^{k-1}\left(\X_s^{[1]}\right)$. For a fixed $v \in \Pi(0)$ we denote for any bounded $\gamma \in \Pi(1)$ containing $v$ the other endpoint of $\gamma$ by $v_{\gamma}$.  
 Then 
 \[
 -\gamma \rho(f) = \left(-\hat{f}_v\right)_{v \in \Pi(0)} \in \CH_{\T_S}^k(\X_0^{[1]}),
 \]
where $-\hat{f}_v$ is given by 
 \[
 -\hat{f}_v= \sum_{\stackrel{\gamma \in \Pi(1)}{\gamma \text{ bdd }, \;v \prec \Gamma}}\left({u_{v, \gamma}}_*u_{v_{\gamma},\gamma}^*f_{v_{\gamma}}-\varphi_{v,\gamma}f_v\right).
 \]
Here, the function $\varphi_{v,\gamma}$ is defined as in Equation \eqref{eq:proj}.
 \end{lemma}
 \begin{proof}
 This follows from the explicit description of $\rho$ and $\gamma$ using the projection formula
 \[
 {u_{v,\gamma}}_*u_{v,\gamma}^*f_v = f_v{u_{v,\gamma}}_*(1) = f_v\varphi_{v,\gamma}.
 \]
 \end{proof}
 \begin{exa}
 \begin{enumerate}
 \item Consider $\Pi = \Sigma$ the canonical model. We have 
 \[
 \CH^k_{\T_S}(\X_{\Pi}) \simeq \CH^k_{\T_S}(\X_{\eta})
 \]
and $\iota^*\iota_* = 0$ in this case. 

\item Suppose that $f = f_{\tilde{v}}$ for some $\tilde{v} \in \Pi(0)$. Then 
\[
-\gamma \rho(f) = \left(-\hat{f}_v\right)_{v \in \Pi(0)}
\]
with 
\[
-\hat{f}_v = \begin{cases} {u_{v,\gamma}}_*u_{\tilde{v},\gamma}^*f_{\tilde{v}} \; &\text{ if }\; \exists \gamma \in \Pi(1) \; \text{ with }\; v, \tilde{v}\prec \gamma, v \neq \tilde{v}, \\
\sum_{\gamma} -\varphi_{\tilde{v},\gamma}f_{\tilde{v}} \; & \text{ if }\; v = \tilde{v},
\end{cases}
\]
where the sum in the second line is over all bounded $\gamma \in \Pi(1)$ having $v$ as a face.
 \end{enumerate}
 \end{exa}
 \subsection{On $\on{ker}(\iota^*\iota_*)$ and $\on{coker}(\iota^*\iota_*)$}\label{subsec:ker}

Denote by 
\[
a \colon \CH_{\T_S}^k(\X_s) \longrightarrow \CH_{n-k}^{\T_S}(\X_s)  
\]
the composition of the morphisms
\[
\CH_{\T_S}^k(\X_s) \longrightarrow \CH_{\T_S}^k(\X_s^{[1]}) \longrightarrow \CH^{\T_S}_{n-k}(\X_s^{[1]})  \longrightarrow \CH^{\T_S}_{n-k}(\X_s).
\]
Then, as in \cite[Theorem 5]{bgs-deg}, we have that the following long sequence is exact.
\begin{eqnarray*}
\begin{tikzcd}
{}\arrow{r} &\CH^{\T_S}_{n+1-k}(\X_{s})\arrow["{\iota^*\iota_*}"]{r}&\CH_{\T_S}^{k}(\X_s)
\arrow["{a}"]{r}&\CH^{\T_S}_{n-k}(\X_s)
\arrow["{\iota^*\iota_*}"]{r}&\CH_{\T_S}^{k+1}(\X_s) 
\arrow{r} &{}
\end{tikzcd}
\end{eqnarray*}
One can think of the above exactness as saying that the failure of Poincaré duality on the special fiber is encoded in $\iota^*\iota_*$. 

We get the following proposition.
\begin{prop}\label{prop:ker=koker}
There is an isomorphism $\on{ker}(\iota^*\iota_*) \simeq \on{coker}(\iota^*\iota_*)$. Moreover, these groups are all isomorphic for any given toric model. 
\end{prop}
\begin{proof}
The first statement follows from the long exact sequence above. For the second one, one can check that this follows from the explicit description of the map $-\gamma \rho$ in Lemma \ref{lem:compo}.
\end{proof}
\begin{rem} The above proposition is the equivariant version of \cite[Theorem~2.2.1]{BGS} and \cite[Proposition~2.3.3]{BGS}. Note that the latter assumes the existence of a smooth model, which for toric schemes always exists (the canonical model).
\end{rem}
It follows from the above proposition that we can compute $\on{ker}(\iota^*\iota_*)$ on a single toric model, e.g. on the canonical model $\X = \X_{\Sigma}$. We obtain the following.
\begin{cor}\label{cor:ker}
For any $\Pi \in R(\Sigma)$ we have that
\[
\on{ker}\left(\iota^*\iota_* \colon \CH^{\T_S}_{n-k+1}(\X_{\Pi,s}) \longrightarrow \CH_{\T_S}^k(\X_{\Pi,s}) \right) 
\]
 is isomorphic to 
$CH^{k}_{\T_K}(X_{\Sigma})$.
\end{cor}
\begin{proof}
Consider the canonical model $\X_{\Sigma}.$ In this case, as we have already seen, the $a$ in the above exact sequence is an isomorphism and the map $\iota^*\iota_*$ is the zero map. Hence
\[
\on{ker}\left(\iota^*\iota_*\right) = \CH_{n-k+1}^{\T_S}(\X_{\Sigma,s}) \simeq \CH_{n-k}^{\T_K}(X_{\Sigma}) \simeq \CH^{k}_{\T_K}(X_{\Sigma}).
\]
\end{proof}

 \subsection{Equivariant forms and currents as limits}\label{sec:forms-currents}
 We define the equivariant analogues of closed $(k,k)$-forms, $(k,k)$-forms modulo the image of $\partial$ and $\bar{\partial}$, closed $(k,k)$-currents and 
$(k,k)$-currents modulo the image of $\partial$ and $\bar{\partial}$ appering in \cite{BGS} in the case of toric varieties. These are defined using appropriate direct and inverse limits of equivariant Chow groups of the special fibers. We then relate these spaces to piecewise polynomial functions.

 Let $\Pi', \Pi \in R(\Sigma)$ be two complete, regular SCR polyhedral complexes corresponding to the toric models $\X' \coloneqq \X_{\Pi'}$ and $\X \coloneqq \X_{\Pi}$. Assume that $\Pi'\geq \Pi$ and let $\pi \colon \X' \to \X$ denote the corresponding (proper, birational) map of toric models. Recall from Example \ref{exa:orient} that we have an equivariant orientation class $[\pi_s] \in \CH_{\T_S}^0(\pi_s\colon \X_s' \to \X_s)$. 
 
 In general, if $ \mathfrak{Y}$ and $\mathfrak{Z}$ are $\T_S$-schemes and $f\colon \mathfrak{Y} \to \mathfrak{Z}$ is equivariantly orientable, with orientation $[f] \in \CH_{\T_S}^d(f \colon \mathfrak{Y} \to \mathfrak{Z})$, we write 
 \[
 f^!\colon \CH_k^{\T_S}(\mathfrak{Z}) \to \CH^{\T_S}_{k-d}(\mathfrak{Y})
 \]
 for the induced map on equivariant Chow homology, and if $f$ is proper, we write 
 
 \[
 f_!\colon \CH^k_{\T_S}(\mathfrak{Y}) \to \CH_{\T_S}^{k+d}(\mathfrak{Z})
 \]
 for the induced map on equivaraint cohomology.
 
%
 
Thus, we have two covariant maps 
 \begin{eqnarray*}
 {\pi_s}_* \colon \CH_k^{\T_S}(\X_s') &\longrightarrow& \CH_k^{\T_S}(\X_s), \\
  {\pi_s}_! \colon \CH_{\T_S}^{n-k}(\X_s') &\longrightarrow & \CH_{\T_S}^{n-k}(\X_s),
 \end{eqnarray*}
 and two contravariant maps 
  \begin{eqnarray*}
 {\pi_s}^* \colon \CH^{n-k}_{\T_S}(\X_s) &\longrightarrow& \CH^{n-k}_{\T_S}(\X'_s), \\
  {\pi_s}^! \colon \CH^{\T_S}_{k}(\X_s) &\longrightarrow & \CH^{\T_S}_{k}(\X_s').
 \end{eqnarray*}
 
 Our next goal is to describe these maps in terms of piecewise polynomial functions. 
 \begin{itemize}
 \item (The map ${\pi_s}_*$) 
 
 By Equation \eqref{eq:hom} the map 
 \[
  {\pi_s}_* \colon \CH_k^{\T_S}(\X_s') \longrightarrow \CH_k^{\T_S}(\X_s)
  \]
  is determined by a map 
 
  \[
  \alpha \colon \CH_k^{\T_S}(\X_s'^{[1]}) \longrightarrow \CH_k^{\T_S}(\X_s^{[1]})
  \]
  such that the following diagram commutes 
\begin{center}
    \begin{tikzpicture}
      \matrix[dmatrix] (m)
      {
        \CH_k^{\T_S}(\X_s'^{[1]})&  \CH_k^{\T_S}(\X_s^{[1]})\\
       \CH_k^{\T_S}(\X_s') & \CH_k^{\T_S}(\X_s)\\
      };
      \draw[->] (m-1-1) to node[above]{$\alpha$}  (m-1-2);
      \draw[ ->] (m-1-1)to  (m-2-1);
      \draw[->] (m-1-2)to (m-2-2);
      \draw[->] (m-2-1) to node[below]{$ {\pi_s}_*$} (m-2-2);
      
     \end{tikzpicture}
     \end{center}
  The map $\alpha$ is given in the following way: Write $\alpha = \left(\alpha_{v'}\right)_{v' \in \Pi'(0)}$.  Then each $\alpha_{v'}$ is given by 
  \[
  \alpha_{v'} = \begin{cases} h_{v'} &\; \text{ if } v' \notin \Pi(0) \\
  {\pi_{v'}}_*  &\; \text{ if } v' \in \Pi(0),
  \end{cases}
  \]
  where ${\pi_{v'}}_* \colon \CH_k^{\T_S}(X_{\Pi'(v')}) \to \CH_k^{\T_S}(X_{\Pi(v')})$ denotes the toric pushforward map induced by the refinement $\Pi'(v') \to \Pi(v')$ in case $v' \in \Pi(0)$. In the case that  $v' \notin \Pi(0)$, $h_{v'}$ is defined as follows. Let $\sigma \in \Pi$ a polyhedron containing $v'$ in its interior and let the $v_i$'s be the vertices of $\sigma$. To each such $v_i$ we have a pushforward map $\pi_{v_i,*} \colon \CH_k^{\T_S}(V(\sigma)) \to \CH_k^{\T_S}(V(v_i))$. Then $h_{v'} = \sum_i \pi_{v_i,*}$.
  
 It is easy to check that if $f \in \CH_k^{\T_S}(\X_s'^{[1]})$ is in the image of $\gamma'$ then $\alpha(f) \in \CH_k^{\T_S}(\X_s^{[1]})$ is in the image of $\gamma$. Hence the above diagram commutes and $\alpha$ induces a well-defined map between the equivariant Chow groups.
  
\item (The map ${\pi_s}_!$) 
By Theorem \ref{th:cohomology-special} the map 
\[
 {\pi_s}_! \colon \CH_{\T_S}^{n-k}(\X_s') \longrightarrow  \CH_{\T_S}^{n-k}(\X_s)
 \]
 can be thought of as a pushforward map of affine piecewise polynomial functions. It maps an affine piecewise polynomial function on $\Pi'$ to an affine piecewise polynomial function on $\Pi$. Explicitly, it is given by the dual map
 \[
 \beta = \alpha^* \colon \CH_{\T_S}^{n-k}(\X_s'^{[1]}) \longrightarrow  \CH_{\T_S}^{n-k}(\X_s^{[1]}), \quad \left(f_{v'}\right))_{v' \in \Pi'(0)} \longmapsto \left(f_v\right)_{v \in \Pi(0)},
 \]
 and it makes the following diagram commute
\begin{center}
    \begin{tikzpicture}
      \matrix[dmatrix] (m)
      {
        \CH^{n-k}_{\T_S}(\X_s')&  \CH^{n-k}_{\T_S}(\X_s^)\\
       \CH^{n-k}_{\T_S}(\X_s'^{[1]}) & \CH^{n-k}_{\T_S}(\X_s^{[1]})\\
      };
      \draw[->] (m-1-1) to node[above]{${\pi_s}_!$}  (m-1-2);
      \draw[ ->] (m-1-1)to  (m-2-1);
      \draw[->] (m-1-2)to (m-2-2);
      \draw[->] (m-2-1) to node[below]{$\beta = \alpha^*$} (m-2-2);
      
     \end{tikzpicture}
     \end{center}
Here one checks that if $f \in  \CH^{n-k}_{\T_S}(\X_s'^{[1]})$ is in the Kernel of $\rho'$, then $\beta(f) \in \CH^{n-k}_{\T_S}(\X_s^{[1]})$ is in the Kernel of $\rho$. 
 
 \item (The map ${\pi_s}^*$) The map 
 \[
{\pi_s}^* \colon \CH^{n-k}_{\T_S}(\X_s) \longrightarrow \CH^{n-k}_{\T_S}(\X'_s)
  \]
 can be thought of as a pullback map of affine piecewise polynomial functions. It has the following explicit description. If $f = (f_v)_{v \in \Pi(0)} \in \CH^{n-k}_{\T_S}(\X_s)$ is an affine piecewise polynomial function on $\Pi$,   then $ {\pi_s}^*(f) \in  \CH^{n-k}_{\T_S}(\X'_s)$ is the affine piecewise polynomial function $(f_{v'})$ on $\Pi'$ induced by $f$, i.e. 
 \[
 f_{v'} = \begin{cases} f_{v} &\text{ if } v = v' \in \Pi(0),\\
 f_{v,\Lambda} &\text{ otherwise }, \end{cases}
 \]
where $\Lambda$ is any (full-dimensional) polyhedron in $\Pi$ containing $v' $ in its interior and $v$ is any vertex of $\Lambda$.  Recall that $f_{v,\Lambda}$ denotes the polynomial function on $N_{\R}$ representing the restriction of $f_v$ to $\Lambda$.

This is well defined by the characterization of $\CH^{n-k}_{\T_S}(\X_s)$ in terms of $PP^{n-k}(\Pi)$.
  
  \item (The map $ {\pi_s}^!$) By Equation \eqref{eq:hom} the map 
  \[
   {\pi_s}^! \colon \CH^{\T_S}_{k}(\X_s) \longrightarrow  \CH^{\T_S}_{k}(\X_s')
   \]
   is determined by a map 
   \[
   \zeta \colon \CH^{\T_S}_{k}(\X_s^{[1]}) \longrightarrow  \CH^{\T_S}_{k}(\X_s'^{[1]})
  \]
 such that the following diagram commutes
\begin{center}
    \begin{tikzpicture}
      \matrix[dmatrix] (m)
      {
        \CH_{k}^{\T_S}(\X_s^{[1]})&  \CH_{k}^{\T_S}(\X_s'^{[1]})\\
       \CH_{k}^{\T_S}(\X_s) & \CH_{k}^{\T_S}(\X_s')\\
      };
      \draw[->] (m-1-1) to node[above]{$\zeta$}  (m-1-2);
      \draw[ ->] (m-1-1)to  (m-2-1);
      \draw[->] (m-1-2)to (m-2-2);
      \draw[->] (m-2-1) to node[below]{${\pi_s}^!$} (m-2-2);
      
     \end{tikzpicture}
     \end{center}
     
    The map $\zeta$ is given as follows. Write $\zeta = (\zeta_v)_{v \in \Pi(0)}$. Then each $\zeta_{v}$ is given by $\zeta_v = \pi_v^*$, where $\pi_v^* \colon \CH_k^{\T_S}(X_{\Pi(v)}) \to    \CH_k^{\T_S}(X_{\Pi'(v)})$ is the pullback corresponding to the fan refinement $\Pi'(v) \to \Pi(v)$. 
    
    
    Here one checks that if $f \in \on{Im}(\gamma)$ then $\zeta(f) \in \on{Im}(\gamma')$.  
 \end{itemize}
 The following definition is the equivariant version of \cite[Def. 1.4.1--1.4.4]{BGS}.
 \begin{Def}\label{def:forms&currents} 
 \begin{enumerate}
 \item The \emph{group of closed equivariant $(k,k)$-forms on $X$} is the direct limit
 \[
 A^{k,k}_{\closed,\T_K}(X) \coloneqq \varinjlim_{\Pi \in R(\Sigma)}\CH_{\T_S}^k\left(\X_{\Pi,s}\right),
 \]
 with maps given by $\pi_s^*$. 
 
 \item The \emph{group of equivariant $(k,k)$-forms modulo the image of $\partial$ and $\bar{\partial}$} is the direct limit
 \[
 \widetilde{A}_{\T_K}^{k,k}(X) \coloneqq \varinjlim_{\Pi \in R(\Sigma)}\CH^{\T_S}_{n-k}\left(\X_{\Pi,s}\right),
 \]
 with maps given by $\pi_s^{!}$.
 \item The \emph{group of closed equivariant $(p,p)$-currents is} the projective limit
 \[
 D_{\closed,\T_K}^{k,k}(X) \coloneqq \varprojlim_{\Pi \in R(\Sigma)}\CH_{\T_S}^k\left(\X_{\Pi,s}\right),
 \]
 with maps given by the ${\pi_s}_!$. 
 \item The \emph{group of equivariant $(k,k)$-currents modulo the image of $\partial$ and $\bar{\partial}$} is the projective limit
 \[
 \widetilde{D}_{\T_K}^{k,k}(X) \coloneqq \varprojlim_{\Pi \in R(\Sigma)}\CH^{\T_S}_{n-k}\left(\X_{\Pi,s}\right),
 \]
 with maps given by the ${\pi_s}_*$. 
 \end{enumerate}
 \end{Def}
The following proposition gives a combinatorial description of these spaces.  
 \begin{prop}\label{prop:com-forms-currents}
 \begin{enumerate}
\item An element in  $A^{k,k}_{\closed,\T_K}(X)$ is determined by an element $f \in \CH_{\T_S}^k\left(\X_{\Pi,s}\right)$ for some model $\Pi \in R(\Sigma)$, i.e. by an affine piecewise  polynomial function on $\Pi$. All the other elements are obtained from $f$ by the pullback ${\pi_s}^*$.  In other words, we get 
\begin{align*}
A^{k,k}_{\closed,\T_K}(X) \simeq \big{\{} f \colon N_{\R} \to \R \; | \; f &\text{ is affine piecewise polynomial of degree $k$}\\ &\text{ with respect to some }\Pi \in R(\Sigma)\big{\}}.
\end{align*}
Hence,
\begin{align*}
A^{*,*}_{\closed,\T_K}(X)\coloneqq \bigoplus_{k \in \Z} A^{k,k}_{\closed,\T_K}(X) \simeq \big{\{} f \colon N_{\R} \to \R \; | \; f &\text{ is affine piecewise polynomial}\\ &\text{with respect to some }\Pi \in R(\Sigma)\big{\}}.
\end{align*}
 
 \item An element in $\widetilde{A}_{\T_K}^{k,k}(X)$ is determined by an element in $\CH_{n-k}^{\T_S}\left(\X_{\Pi,s}^{[1]}\right)$ for some model $\Pi \in R(\Sigma)$ and all the other elements are obtained by pullback under $\zeta$. 
 \item An element in $D_{\closed,\T_K}^{k,k}(X)$ is a tuple of affine piecewise polynomial functions $\left(f_{\Pi'}\right)_{\Pi' \in R(\Sigma)}$, where $f_{\Pi'}$ is piecewise polynomial with respect to $\Pi'$ of degree $k$, compatible under the pushforward map ${\pi_s}_!$, i.e. such that 
 \[
 {\pi_s}_!(f_{\Pi'})= f_{\Pi}
 \]
 for any $\Pi' \geq \Pi$ in $R(\Sigma)$. 
For $k=1$ such a tuple determines a function 
 \[
 f \colon N_{\Q} \to \Q,
\]
up to a constant (see Remark \ref{rem:piecewise-affine}).

 \item An element in $\widetilde{D}_{\T_K}^{k,k}(X)$ is a tuple $\left(g_{\Pi}\right)_{\Pi \in R(\Sigma)}$ with $g_{\Pi} \in \CH_{n-k}^{\T_S}\left(\X_{\Pi,s}^{[1]}\right)$, compatible under the pushforward map $\alpha$ given above, i.e. such that 
 \[
 \alpha(g_{\Pi'}) = g_{\Pi}
 \]
for any $\Pi' \geq \Pi$ in $R(\Sigma)$.  Note that here on new vertices $v\in \Pi'(0) \setminus \Pi(0)$ one allows any affine piecewise polynomial function on $PP^*(\Pi'(v))$ modulo $\on{Im}(\gamma)$.  
 \end{enumerate}
 \end{prop}
 \begin{proof}
 This follows directly from the definitions and the combinatorial descriptions of the equivariant Chow groups given in Section \ref{sec:eq-chow}.
 \end{proof}

 \begin{exa}\label{ex:b-div} Consider $k =1$ and pick any $\Pi \in R(\Sigma)$. We have seen that a form in $A^{1,1}_{\closed,\T_K}(X)$ defines, up to a constant, a piecewise affine function on $\Pi$ and hence a piecewise linear function on the fan $c(\Pi)$. On the other hand, toric Cartier b-divisor on the fan $c(\Pi)$ can be identified with piecewise linear functions on the support $|c(\Pi)|$ and two of them represent linear equivalent divisors if and only if the difference is a global linear function. In particular, we see that a form  in $A^{1,1}_{\closed,\T_K}(X)$ determines a class of Cartier b-divisors on the fan $c(\Pi)$. On the other hand, a current in $D_{\closed,\T_K}^{1,1}(X)$ should define a ``class'' of a toric Weil b-divisor on $c(\Pi)$ (we refer to \cite{botero} for the notion of toric b-divisors). 
 \end{exa}
 \begin{rem}
 For arbitrary $k$, it would be interesting to relate forms in $A^{k,k}_{\closed,\T_K}(X)$ to elements in McMullen's Polytope Algebra \cite{FS} and currents in $D_{\closed,\T_K}^{k,k}(X)$ to elements in the Convex Set Algebra \cite{botero-conv}.  
 \end{rem}
 Clearly we have unique injective maps 
 \[
 \iota_{\closed} \colon A^{k,k}_{\closed,\T_K}(X)  \longrightarrow D^{k,k}_{\closed,\T_K}(X)
 \]
 such that for every map of toric models $\pi \colon \X_{\Pi''} \to \X_{\Pi'}$, the composition 
 \begin{eqnarray*}
\begin{tikzcd}
\CH_{\T_S}^k\left(\X_{\Pi',s}\right)  \arrow{r} & A^{k,k}_{\closed,\T_K}(X) \arrow["{\iota_{\closed}}"]{r}&D_{\closed,\T_K}^{k,k}(X)
\arrow{r}&\CH_{\T_S}^k\left(\X_{\Pi'',s}\right)
\end{tikzcd}
\end{eqnarray*}
is equal to ${\pi_s}^*$.  
 
 Similarly we have unique injective maps 
 \[
 \widetilde{\iota} \colon \widetilde{A}_{\T_S}^{k,k}(X) \longrightarrow \tilde{D}_{\T_S}^{k,k}(X)
 \]
 such that for $\pi \colon \X_{\Pi''} \to \X_{\Pi'}$ as above, the composition 
 \begin{eqnarray*}
\begin{tikzcd}
\CH^{\T_S}_{n-k}\left(\X_{\Pi',s}\right)  \arrow{r} & \widetilde{A}^{k,k}_{\T_K}(X) \arrow["{\widetilde{\iota}}"]{r}&\widetilde{D}_{\T_K}^{k,k}(X)
\arrow{r}&\CH^{\T_S}_{n-k}\left(\X_{\Pi'',s}\right)
\end{tikzcd}
\end{eqnarray*}
is equal to ${\pi_s}^!$.

Now, let $\X = \X_{\Pi}$ be a toric scheme. Recall the map 
 \[
 \iota_*\iota^* \colon \CH_{n-k}^{\T_S}(\X_s) \longrightarrow \CH_{\T_S}^{k+1}(\X_s)
 \]
 from Section \ref{sec:iota}, which is given combinatorially by $-\gamma \rho$. The next proposition says that this map is compatible with the covariant maps ${\pi_s}_*$, ${\pi_s}_!$ and with the contravariant maps ${\pi_s}^*$, ${\pi_s}^!$. 
 \begin{prop}
 Let $\Pi' \geq \Pi \in R(\Sigma)$ with corresponding map of toric models $\pi \colon \X' \to \X$. We have commutative diagrams 
 
\begin{center}
    \begin{tikzpicture}
      \matrix[dmatrix] (m)
      {
      \CH_{\T_S}^{n-k}\left(\X_s'^{[1]}\right) & \CH_k^{\T_S}\left(\X_s'\right) & \CH_{\T_S}^{n+1-k}\left(\X_s'\right) & \CH_{\T_S}^{n+1-k}\left(\X_s'^{[1]}\right) \\
       \CH_{\T_S}^{n-k}\left(\X_s^{[1]}\right) & \CH_k^{\T_S}\left(\X_s\right) & \CH_{\T_S}^{n+1-k}\left(\X_s\right) & \CH_{\T_S}^{n+1-k}\left(\X_s^{[1]}\right)\\
      };
      \draw[->] (m-1-1) to (m-1-2);
      \draw[->] (m-1-2) to node[above]{$\iota^*\iota_*$}  (m-1-3);
        \draw[->] (m-1-3) to (m-1-4);
          \draw[->] (m-2-1) to (m-2-2);
      \draw[->] (m-2-2) to node[above]{$\iota^*\iota_*$}  (m-2-3);
        \draw[->] (m-2-3) to (m-2-4);
        
      \draw[ ->] (m-1-1)to node[left]{$\alpha$} (m-2-1);
        \draw[ ->] (m-1-2)to node[left]{${\pi_s}_*$} (m-2-2);
        
          \draw[ ->] (m-1-3)to node[right]{${\pi_s}_!$} (m-2-3);
        \draw[ ->] (m-1-4)to node[right]{$\beta = \alpha^*$} (m-2-4);
        
        \draw[->, bend left =20] (m-1-1) to node[above]{$-\gamma\rho$} (m-1-4);
         \draw[->, bend right =20] (m-2-1) to node[above]{$-\gamma\rho$} (m-2-4);

     \end{tikzpicture}
     \end{center}
 
 and 
 
\begin{center}
    \begin{tikzpicture}
      \matrix[dmatrix] (m)
      {
      \CH_{\T_S}^{n-k}\left(\X_s^{[1]}\right) & \CH_k^{\T_S}\left(\X_s\right) & \CH_{\T_S}^{n+1-k}\left(\X_s\right) & \CH_{\T_S}^{n+1-k}\left(\X_s^{[1]}\right) \\
       \CH_{\T_S}^{n-k}\left(\X_s'^{[1]}\right) & \CH_k^{\T_S}\left(\X_s'\right) & \CH_{\T_S}^{n+1-k}\left(\X_s'\right) & \CH_{\T_S}^{n+1-k}\left(\X_s'^{[1]}\right)\\
      };
      \draw[->] (m-1-1) to (m-1-2);
      \draw[->] (m-1-2) to node[above]{$\iota^*\iota_*$}  (m-1-3);
        \draw[->] (m-1-3) to (m-1-4);
          \draw[->] (m-2-1) to (m-2-2);
      \draw[->] (m-2-2) to node[above]{$\iota^*\iota_*$}  (m-2-3);
        \draw[->] (m-2-3) to (m-2-4);
        
      \draw[ ->] (m-1-1)to node[left]{$\zeta$} (m-2-1);
        \draw[ ->] (m-1-2)to node[left]{${\pi_s}^!$} (m-2-2);
        
          \draw[ ->] (m-1-3)to node[right]{${\pi_s}^*$} (m-2-3);
        \draw[ ->] (m-1-4)to node[right]{${\pi_s}^*$} (m-2-4);
        
        \draw[->, bend left =20] (m-1-1) to node[above]{$-\gamma\rho$} (m-1-4);
         \draw[->, bend right =20] (m-2-1) to node[above]{$-\gamma\rho$} (m-2-4);

     \end{tikzpicture}
     \end{center}
     
     Hence, we obtain induced maps on the direct and inverse systems
     
     \begin{eqnarray}\label{eq:ddc-forms}
     dd^c \colon \widetilde{A}_{\T_K}^{k,k}(X) \longrightarrow A_{\closed, \T_K}^{k+1,k+1}(X)
     \end{eqnarray}
and 
     \begin{eqnarray}\label{eq:ddc-currents}
      dd^c \colon \widetilde{D}_{\T_K}^{k,k}(X) \longrightarrow D_{\closed, \T_K}^{k+1,k+1}(X).
     \end{eqnarray}
 
 \end{prop}
 \begin{proof}
One checks the commutativity of the diagrams using the combinatorial descriptions of the maps involved. 
 \end{proof}
 We now give a combinatorial description of the maps \eqref{eq:ddc-forms} and \eqref{eq:ddc-currents}.  First note that by the exact sequence in Section \ref{subsec:ker} we have that 
 \begin{eqnarray}\label{eq:pa}
 \on{Im}(-\gamma\rho) \subseteq \on{ker}(\rho).
 \end{eqnarray}
 \begin{itemize}
 \item \eqref{eq:ddc-forms} Let $c \in \widetilde{A}_{\T_S}^{k,k}(X)$.  The element $c$ is determined by an element $\widehat{c} \in \CH_{n-k}^{\T_S}\left(\X_s^{[1]}\right)$ for some model $\X$.  Then $dd^c(c)$ is determined by the element $-\gamma \rho \left(\widehat{c}\right) \in \CH_{\T_S}^{k+1}\left(\X_s^{[1]}\right)$.
 
 Note that by \eqref{eq:pa} we have that $-\gamma \rho \left(\widehat{c}\right)$ indeed determines an affine piecewise polynomial function on some model $\X' \geq \X$ (not necessarily $\X$). 
 \item \eqref{eq:ddc-currents} Let $c \in \widetilde{D}_{\T_S}^{k,k}(X)$.  The element $c$ corresponds to a tuple $\widehat{c} = \left(f_{\Pi'}\right)_{\Pi'}$ with $f_{\Pi'} \in \CH^{\T_S}_{n-k}\left(\X_{\Pi',s}^{[1]}\right)$ compatible under the map $\alpha$. Then $dd^c(c)$ is given by the tuple $\left(-\gamma\rho(f_{\Pi'})\right)_{\Pi'}$ with $-\gamma\rho(f_{\Pi'}) \in \CH_{\T_S}^{k+1}\left(\X_{\Pi,s}^{[1]}\right)$. 
 
 Also here, by \eqref{eq:pa}, each $-\gamma \rho\left(f_{\Pi'}\right)$ gives an affine piecewise polynomial function (defined on a possibly higher model). Hence, $\left(-\gamma \rho\left(f_{\Pi'}\right)\right)_{\Pi'}$ defines a tuple of affine piecewise polynomial functions, and these are compatible under ${\pi_s}_!$. 
 \end{itemize}
 \begin{exa}\label{exa:green-current}
 For $k=0$, an element $c \in \widetilde{D}_{\T_S}^{0,0}(X)$ is an element in $\varprojlim_{\Pi \in \R(\Sigma)}\bigoplus_{v \in \Pi(0)}\Q$, i.e. it associates to each rational point in $N_{\Q}$ a rational number, hence it defines a function $N_{\Q} \to \Q$. This function is associated to $dd^c(c) \in D^{1,1}_{\closed, \T_K}$ in Proposition \ref{prop:com-forms-currents}. 
 \end{exa}
 The $dd^c$ maps satisfy the following regularity property. It is the equivariant analogue of \cite[Theorem 2.3.1 iv)]{BGS}.
 \begin{prop}\label{prop:regularity} Assume that $g \in \widetilde{D}^{k,k}_{\T_K}(X)$ is such that $dd^c(g)$ lies in the subgroup $A^{k+1,k+1}_{\closed,\T_K}(X)$ of $D^{k+1,k+1}_{\closed,\T_K}(X)$. Then $g$ lies in $\widetilde{A}^{k,k}_{\T_K}(X) \subseteq \widetilde{D}^{k,k}_{\T_K}(X)$. 
 \end{prop}
 \begin{proof}
 Suppose that $dd^c(g)$ is an affine piecewise polynomial function, determined on some toric model $\Pi \in R(\Sigma)$. Then it follows from the combinatorial description of the map $-\gamma \rho$ given above, that $g$ has to be already determined on some (possibly higher) toric model $\Pi'$.  The statement then follows from Proposition \ref{prop:com-forms-currents}.
 \end{proof}
 Now we consider the question of mapping closed forms or currents into forms or currents modulo $\partial$, $\bar{\partial}$.  In combinatorial terms, this asks whether we can write affine piecewise polynomial functions as direct sums of piecewise polynomials functions, indexed over all the vertices of the polyhedral complex, in such a way that functorial compatibilities are satisfied. 

Let $\X = \X_{\Pi}$ be a model of $X$ and consider the fundamental class of the special fiber $[\X_s] \in \CH_n^{\T_S}(\X_s)\simeq \CH_n^{\T_S}(\X_s^{[1]})\simeq \bigoplus_{v \in \Pi(0)} \Q$.  
Then $[\X_s]$ corresponds then to an element $(m_v)_{v\in \Pi(0)}$, where $m_v \in \Q$ is the multiplicity of $V(v)$ in $\X_s$. 

Multiplication by $[\X_s]$ gives the cap product 
\[
\cap [\X_s]\colon \CH_{\T_S}^k(\X_s) \longrightarrow \CH_{n-k}^{\T_S}(\X_s)
\]
which in this case is given by 
\[
f \in PP^k(\Pi) \longmapsto f \cap [\X_s] = \left(m_vf_v\right)_{v \in \Pi(0)} \in \CH_{\T_S}^k(\X_s^{[1]}) \simeq \CH_{n-k}^{\T_S}(\X_s^{[1]})
\]
and taking the image in $\CH_{n-k}^{\T_S}(\X_s)$.

Given a map of toric models $\X' \to \X$, $f \in \CH^k_{\T_S}(\X_s)$ and $f' \in \CH^k_{\T_S}(\X'_s)$, by the functoriality properties of the equivariant operational Chow groups, we have that 
\[
{\pi_s}^!\left( f \cap [\X_s]\right) = \left(\pi_s^* f\right) \cap [\X_s']
\]
and 
\[
{\pi_s}_*\left(f' \cap [\X_s']\right) = \left({\pi_s}_!f'\right)\cap [\X_s].
\]
Hence, the cap products $\cap [\X_s]$ induce maps of direct and inverse limits
\[
g\colon  A^{k,k}_{\closed, \T_S}(X) \longrightarrow  \widetilde{A}^{k,k}_{\T_S} 
\]
\[
g' \colon D_{\closed, \T_S}^{k,k}(X) \longrightarrow  \widetilde{D}_{\T_S}^{k,k}(X).
\]
We obtain a commutative diagram  
\begin{center}
    \begin{tikzpicture}
      \matrix[dmatrix] (m)
      { A^{k,k}_{\closed, \T_S}(X) & D_{\closed, \T_S}^{k,k}(X)\\
      \widetilde{A}^{k,k}_{\T_S} & \widetilde{D}_{\T_S}^{k,k}(X) \\
      };
      \draw[->] (m-1-1) to (m-1-2);
    \draw[->] (m-1-1) to node[left]{$g$}(m-2-1);
     \draw[->] (m-2-1) to (m-2-2);
      \draw[->] (m-1-2) to node[right]{$g'$} (m-2-2);

    \end{tikzpicture}
     \end{center}
     where the horizontal maps are the canonical injections mentioned above.
     The following proposition follows from the combinatorial description of the maps taking into account Equation \eqref{eq:pa}. It is the equivariant analogue of \cite[Proposition 1.4.10]{BGS}.
\begin{prop}
The compositions
\begin{eqnarray*}
\begin{tikzcd}
A_{\closed,\T_K}^{k,k}(X)  \arrow["g"]{r} & \widetilde{A}^{k,k}_{\T_K}(X) \arrow["dd^c"]{r}&A_{\closed,\T_K}^{k+1,k+1}(X)
\end{tikzcd}
\\
\begin{tikzcd}
D_{\closed,\T_K}^{k,k}(X)  \arrow["g'"]{r} & \widetilde{D}^{k,k}_{\T_K}(X) \arrow["dd^c"]{r}&D_{\closed,\T_K}^{k+1,k+1}(X)
\end{tikzcd}
\\
\begin{tikzcd}
\widetilde{A}_{\T_K}^{k,k}(X)  \arrow["dd^c"]{r} & A^{k+1,k+1}_{\closed,\T_K}(X) \arrow["g"]{r}& \widetilde{A}_{\T_K}^{k+1,k+1}(X)
\end{tikzcd}
\\
\begin{tikzcd}
\widetilde{D}_{\T_K}^{k,k}(X)  \arrow["dd^c"]{r} & D^{k+1,k+1}_{\closed,\T_K}(X) \arrow["g'"]{r}& \widetilde{D}_{\T_K}^{k+1,k+1}(X)
\end{tikzcd}
\end{eqnarray*}
are all $0$.
\end{prop}
\subsection{Some properties}
We list some properties of the forms and currents defined in the previous section. As before,  $X= X_{\Sigma}$ denotes a smooth projective toric variety over the discretely valued field $K$.\\

\textbf{Products.} 
The following proposition follows from the combinatorial description of the forms and currents in Proposition \ref{prop:com-forms-currents}.
\begin{prop} 
\begin{enumerate}
\item[i)] The group $A_{\closed,\T_K}^{*,*}(X)$ is a graded commutative ring. The product is given by the product of affine piecewise polynomial functions. 
The groups $\widetilde{A}_{\T_K}^{*,*}(X)$, $D_{\closed,\T_K}^{*,*}(X)$ and $\widetilde{D}_{\T_K}^{*,*}(X)$ are all graded $A_{\closed,\T_K}^{*,*}(X)$-modules.
\item[ii)] The $dd^c$ map on both forms and currents is a map of graded $A_{\closed,\T_K}^{*,*}(X)$-modules.
\end{enumerate}
\end{prop}

Note that the ring and module structures are uniquely determined by the requirement that they be compatible with the product on $\CH_{\T_S}^*(\X_s)$ and the $\CH_{\T_S}^*(\X_s)$-module structure on $\CH_*^{\T_{S}}(\X_s)$ for any toric model $\X \in R(X)$.

\begin{rem}
There is also an associative product on $\widetilde{A}_{\T_K}^{*,*}(X)$ 
\[
\widetilde{A}_{\T_K}^{*,*}(X) \otimes \widetilde{A}_{\T_K}^{*,*}(X) \longrightarrow  \widetilde{A}_{\T_K}^{*,*}(X)
\]
given by 
\begin{eqnarray}\label{eq:prod}
 c \otimes d  \longmapsto  c\cdot dd^c d.
\end{eqnarray}
We expect by \cite[Example 17.4.4]{fulint} that this product is also commutative. 
\end{rem}
\vspace{0.1cm}
\textbf{Functoriality.} Let $f \colon X \to Y$ be an equivariant map of smooth projective toric varieties over $K$. By resolution of singularities, for each toric model $\mathfrak{Y}$ of $Y$ corresponding to some SCR polyhedral complex $\Psi$, there exists a regular toric model $\X \in R(X)$ corresponding to $\Pi$ and a commutative diagram 
\begin{center}
    \begin{tikzpicture}
      \matrix[dmatrix] (m)
    { X & \X \\
     Y & \mathfrak{Y}\\
      };
      \draw[->] (m-1-1) to (m-1-2);
    \draw[->] (m-1-1) to node[left]{$f$}(m-2-1);
     \draw[->] (m-2-1) to (m-2-2);
      \draw[->] (m-1-2) to node[right]{$h$} (m-2-2);
    
     \end{tikzpicture}
     \end{center}
     One calls $\X$ a \emph{model of $X$ over $\mathfrak{Y}$.} Denote by $h_s\colon \X_s \to \mathfrak{Y}_s$ the induced map between the special fibers. This is compatible with maps of toric models $\mathfrak{Y}' \to \mathfrak{Y}$ in $R(Y)$. Hence one may take the direct limit of the maps 
     \[
     h_s^*\colon \CH_{\T_S}^*\left(\mathfrak{Y}_s\right) \longrightarrow \CH_{\T_S}^*\left(\X_s\right)
     \]
     and obtain a ring homomorphism 
     \[
     f^* \colon A^{*,*}_{\closed, \T_K}(Y) \longrightarrow A^{*,*}_{\closed, \T_K}(X).
     \]
By Proposition \ref{prop:com-forms-currents} this is given as the pullback map of affine piecewise polynomial functions.

In this way, $A^{*,*}_{\closed, \T_K}(\cdot)$ defines a contravariant functor from the category of smooth projective toric varieties over $K$ to rings of affine piecewise polynomial functions.
     
     Dually, let $d = \dim(X) - \dim(Y)$. Taking the inverse limit of the maps 
     \[
     {h_s}_*\colon \CH_*^{\T_S}\left(\X_s\right) \longrightarrow \CH_*^{\T_S}\left(\mathcal{Y}_s\right)
     \]
     we obtain 
     \[
     f_*\colon \widetilde{D}^{*,*}_{\T_K}(X) \longrightarrow \widetilde{D}_{\T_S}^{*-d,*-d}(Y).
     \]
     Similarly we get maps 
       \[
     f^* \colon \widetilde{A}^{*,*}_{\T_K}(Y) \longrightarrow \widetilde{A}_{\T_S}^{*-d,*-d}(X)
     \]
     and
     \[
     f_* \colon D_{\closed,\T_K}^{*,*}(X) \longrightarrow D_{\closed,\T_S}^{*-d,*-d}(Y).
     \]
     
     \begin{exa}\label{degree}
     Consider the structure map $f \colon X\to \on{Spec}(K)$. 
We have \[
\widetilde{D}_{\T_K}^{k,k}\left(\on{Spec}K\right) = \CH^{\T_S}_{-k}(S) \simeq \on{Sym}^k\left(\widetilde{M}_{\Q}\right)
\]
for $k \in [0, \infty)$ (see Example \ref{exa:point}). 
     
     We obtain a map 
     \[
     f_* \colon \widetilde{D}^{n,n}_{\T_K}(X) \longrightarrow  \on{Sym}^0\left(\widetilde{M}_{\Q}\right).
     \]
     
     By composition we obtain 
     \[
     \begin{cases} & A^{n,n}_{\closed, \T_K}(X) \\ & D^{n,n}_{\closed, \T_K}(X) \\ & \widetilde{A}^{n,n}_{\T_K}(X) \\ & \widetilde{D}^{n,n}_{\T_K}(X) \end{cases}  \; \longrightarrow \widetilde{D}^{n,n}_{\T_K}(X)  \longrightarrow  \on{Sym}^0\left(\widetilde{M}_{\Q}\right).
     \]
     We denote these maps by 
     \[
     \alpha \longmapsto \int_X\alpha
     \]
     and call them the \emph{equivariant degree} maps. These are the equivariant analogues of the degree maps in \cite[Example 1.6.3]{BGS}.
     
    \end{exa}
     
\subsection{Equivariant $\delta$- and Green currents}
Let $X = X_{\Sigma}$ be a smooth complete toric variety over $K$ of dimension $n$. As before, for any complete, regular SCR polyhedral complex $\Pi$ in $N_{\R}$ with $\rec(\Pi) = \Sigma$ we denote by $ \X_{\Pi}$ the corresponding complete, regular toric model of $X$ over $S$ of $S$-absolute dimension $n+1$.  

We define the equivariant analogues of $\delta$- and Green currents from \cite[Sections 1.7, 1.8]{BGS} and describe them in combinatorial terms. \\

\textbf{Equivariant $\delta$-currents.}
Let $Y \subseteq X$ be a closed invariant subvariety of codimension $k$. 
 For $\Pi \in R(\Sigma)$ we denote by $\overline{Y}^{\Pi} \subseteq \X_{\Pi}$ the Zariski closure of $Y$ in $\X_{\Pi}$ and by $\left[\overline{Y}^{\Pi}\right] \in \CH^k_{\T_S}\left(\X_{\Pi}\right)$ its associated equivariant cohomology class.

We also write $\iota\colon \X_{\Pi',s} \hookrightarrow \X_{\Pi'}$ for the inclusion of the special fiber for any given model $\Pi' \in R(\Sigma)$.  
The following proposition follows exactly as in \cite[Proposition-Definition 1.7.1]{BGS}.
\begin{prop} 
The association $\Pi \mapsto \iota^*\left[\overline{Y}^{\Pi}\right]$ gives a well defined class $\delta_Y \in D^{k,k}_{\closed,\T_S}(X)$.  
\end{prop}

Now, let $Z$ be an equivariant cycle in $\CH^{\T_K}_*(X)$. Recall that $Z$ is an $\on{Sym}(M)_{\Q}$-linear combination of invariant cycles $Z_i$ (see \cite[Therem 2.1]{BR}). 
\begin{Def}\label{def:delta}
Let notations be as above. The $\delta$-current $\delta_Z$ is the $\on{Sym}(M)_{\Q}$-linear extension of the delta currents of the invariant cycles $Z_i$.
\end{Def}

Let us describe $\delta_Y$ for $Y$ an invariant cycle in terms of piecewise polynomial functions.  Let $\sigma \in \Sigma(k)$ and assume that $Y = V(\sigma)$ the $k$-codimensional toric subvariety corresponding to the cone $\sigma$. Then for any $\Pi \in R(\Sigma)$ we may consider $\varphi_{\sigma, \Pi}$, the piecewise polynomial function on $c(\Pi)$ of degree $k$ associated to $\sigma$ from Definition \ref{def:generators}. Here we view $\sigma$ as a cone in $c(\Pi)$, contained in $N_{\Q}\times \{0\}$.  

It follows that, under the isomorphism in Theorem \ref{th:equi-pol}, $\varphi_{\sigma, \Pi}$ exactly corresponds to $\left[\overline{Y}^{\Pi}\right]$. Moreover, the pullback $\iota^*\left[\overline{Y}^{\Pi}\right]$ corresponds then to taking the pullback of $\varphi_{\sigma, \Pi}$ to $\Pi(v)$ for each vertex $v \in \Pi(0)$. Hence we get that
\[
\delta_Y = \left(\iota^*\left[\overline{Y}^{\Pi}\right]\right)_{\Pi \in R(\Sigma)} \in \D^{k,k}_{\on{closed}, \T_S}(X)
\]
corresponds to the tower of piecewise polynomial functions
\[
\left(\left(\varphi_{\sigma,\Pi}|_{\Pi(v)}\right)_{v \in \Pi(0)}\right)_{\Pi \in R(\Sigma)}
\]
compatible under the pushforward map from Section \ref{subsec:pull}.

As we have seen, this tuple defines a function 
\begin{eqnarray}\label{eq:delta-current}
\delta_Z \colon N_{\Q} \to \Q
\end{eqnarray}
which is also denoted by $\delta_Z$.\\

\textbf{Equivariant Green currents.} Let $\eta = \sum_in_iY_i$ be a codimension $k$ invariant cycle on $X$. Given $\Pi \in R(X)$ we choose a lifting of $\eta$ to a codimension $k$ algebraic cycle $\widehat{\eta} \in \CH_{\T_S}^k\left(\X_{\Pi}\right)$. The image of $\widehat{\eta}$ under the composition 
\begin{eqnarray*}
\begin{tikzcd}
\CH_{\T_S}^k\left(\X_{\Pi}\right)  \arrow["\iota^*"]{r} &\CH_{\T_S}^k\left(\X_{\Pi,s}\right) \arrow{r}& A^{k,k}_{\closed,\T_K}(X)
\end{tikzcd}
\end{eqnarray*}
gives a class $\omega \in A^{k,k}_{\closed,\T_K}(X)$.

Given a map of toric models $\pi\colon \X_{\Pi'}\to \X_{\Pi}$ corresponding to $\Pi'\geq \Pi \in R(\Sigma)$, we let 
\[
g_{\Pi'} \coloneqq \pi^*\left[\widehat{\eta}\right]-\overline{\eta}^{\Pi'},
\]
where, as before, $\left[\overline{\eta}^{\Pi'}\right]$ denotes the class of the Zariski closure of (the irreducible components of) $\eta$ in $\X_{\Pi'}$.  As in \cite[Prop. 1.84]{BGS}, this cycle is supported on the special fiber $\X_{\Pi',s}$, hence it defines a class in $\CH_{n-k+1}^{\T_S}\left(\X_{\Pi',s}\right)$. It follows that the association $\Pi \mapsto g_{\Pi'}$ induces a class 
\[
g = \left(g_{\Pi'}\right)_{\Pi' \in R(X)} \in \widetilde{D}_{\T_K}^{k-1,k-1}(X).
\]
In terms of piecewise polynomial functions, the choice of a lifting $\widehat{\eta}$ to $\X_{\Pi}$ is equivalent to the choice of a piecewise polynomial function $f_{\widehat{\eta}}$ on $c(\Pi)$ whose restriction to $\Sigma$ is the piecewise polynomial function $f_{\eta}$ corresponding to the cycle $\eta$ on $X$. Then $\pi^*\left[\widehat{\eta}\right]$ corresponds to the pullback of $f_{\widehat{\eta}}$ to $c(\Pi')$. On the other hand, $\overline{\eta}^{\Pi'}$ also corresponds to a piecewise polynomial function on $c(\Pi')$ whose restriction to $\Sigma$ is $f_{\eta}$. Hence, the difference 
\[
g_{\X'} = \pi^*\left[\widehat{\eta}\right]-\overline{\eta}^{\Pi'},
\]
is a piecewise polynomial function on $c(\Pi')$ which vanishes when restricted to $\Sigma$. Thus it defines an element in $\CH_{n-k+1}^{\T_S}\left(\X_{\Pi',s}\right)$ and for different toric models $\Pi'' \geq \Pi' \in R(\Sigma)$ these are compatible under the pushforward map ${\pi_s}_*$.
The following proposition follows from the above discussion.

\begin{prop}\label{prop:green-fun}
Let $\eta = \sum_in_iY_i$ be an invariant cycle on $X$ of codimension $k$. Then the choice of a toric model $\Pi \in R(\Sigma)$ and a lifting $\widehat{\eta}$ in $\CH_{\T_S}^k(\X_{\Pi})$ determines a class $g \in \widetilde{D}_{\T_K}^{k-1,k-1}(X)$ such that 
\[
dd^cg = \omega - \delta_{\eta}.
\]
Moreover, since the $dd^c$ map is a map of $\on{Sym}(M)_{\Q}$, such a $g$ exists for any equivariant cycle $\eta$. 
\end{prop}
This $g$ is then is an example of a so called \emph{equivariant Green current} for $\eta$.

\begin{Def}
Let $Z \in \CH_{n-k}^{\T_K}(X)$ be an equivariant cycle. Any element $g \in \widetilde{D}_{\T_K}^{k-1,k-1}(X)$ such that 
\[
dd^cg+\delta_{Z} \in A^{k,k}_{\closed, \T_K}(X)
\]
 is called an \emph{equivariant Green current for} $Z$.
\end{Def}
In terms of piecewise polynomial functions, an element $g \in \widetilde{D}_{\T_K}^{k-1,k-1}(X)$ is a Green current for $Z$ if and only if the sum 
\[
dd^cg+\delta_{Z}
\]
is an affine piecewise polynomial function on $N_{\R}$ with respect to some $\Pi \in R(\Sigma)$.  

\begin{exa}
Let $L_{\psi}$ be a toric Cartier divisor on $X$ associated to a virtual support function $\psi$ on $\Sigma$ and let $s$ be a toric section of $L_{\psi}$. By Proposition \ref{prop:green-fun} we have that the analogue of a toric metric on $L$ is the choice of an element $\mathcal{L}$ in $ \varinjlim_{\Pi\in R(\Sigma)}\CH^1_{\T_S}(\X_{\Pi})$ extending $L$. In other words, the choice of a model $\Pi \in R(\Sigma)$ and a piecewise linear function $\widetilde{\psi}$ on $c(\Pi)$ whose restriction to $\Sigma$ is $\psi$. 

Given such a toric metric $\| \cdot \|$, the analogue of the $L^1$-function $\log\|s\|$ is the Green's function $g \in \widetilde{D}^{0,0}_{\T_K}(X)$ defined for $\pi \colon \X_{\Pi'} \to \X_{\Pi}$ by $g_{\Pi'} = \on{div}\left(\tilde{s}'\right) - \overline{\on{div}(s)}$, where $\left(\tilde{s}'\right)$ is the extension of $s$ to a toric section of $\pi^*\mathcal{L}$ and $\overline{\on{div}(s)}^{\Pi'}$ is the Zariski closure of $\on{div}(s)$ on $\X_{\Pi'}$.  Note that $\on{div}\left(\tilde{s}'\right)$ is a piecewise polynomial function on $c(\Pi')$ but $\overline{\on{div}(s)}^{\Pi'}$ not necessarily. They differ on $\Pi(v)$ for $v \in \Pi'(0)\setminus \Pi(0)$. In this way, $g$ defines a function $N_{\Q} \to \Q$ which is exactly the function associated to $dd^cg \in D^{1,1}_{\closed, \T_S}$ (see Example \ref{exa:green-current}). 
\end{exa}

\section{Equivariant non-arquimedean arithmetic Chow groups of toric varieties}\label{sec:equi-arith-chow}
Let $X = X_{\Sigma}$ be a smooth projective toric variety over the discretely valued field $K$. 
The goal of this section is to define equivariant analogues of the non-arquimedean arithmetic Chow groups $\widehat{\CH}(X)$ and $\widecheck{\CH}(X)$ defined in \cite{BGS} and \cite{GS-direct}, respectively, and to provide combinatorial descriptions of these groups.
\subsection{The equivariant arithmetic Chow group $\widehat{\CH}^*_{\T_K}(X)$}\label{sec:arith-chow1}

Let $W \subseteq X$ be an invariant subvariety of codimension $(k-1)$ and let $f \in K(W)^*$ be a non-zero, invariant rational function on $W$. Then $\on{div}(f)$ on $W$ is an invariant cycle on $X$ of codimension $k$. 

For $\Pi$ in $R(\Sigma)$ we denote by $\on{div}_{\Pi}(f)$ the divisor of $f$ on the Zariski closure $\overline{W}^{\Pi}$ of $W$ on $\X_{\Pi}$. On the other hand, we write $\overline{\on{div}(f)}^{\Pi}$ for the Zariski closure of $\on{div}(f)$ on $\X_{\Pi}$.  Define  
\[
\on{div}_{\nu}(f)_{\Pi} \coloneqq \overline{\on{div}(f)}^{\Pi} - \on{div}_{\Pi}(f).
\]
This is a codimension $(k-1)$ cycle supported on the special fiber $\X_{\Pi,s}$. Given a map of toric models $\pi\colon \X_{\Pi'} \to \X_{\Pi}$ associated to $\Pi' \geq \Pi$ in $R(\Sigma)$, one has 
\[
{\pi_s}_*\left( \on{div}_{\nu}(f)_{\Pi'}\right) = \on{div}_{\nu}(f)_{\Pi}.
\]
Hence, we get a well-defined element 
\[
\on{div}_{\nu}(f) = \left(\on{div}_{\nu}(f)_{\Pi}\right)_{\Pi \in R(\Sigma)} \in \widetilde{D}^{k,k}_{\T_K}(X).
\]

We then have the following equivariant version of the \emph{Poincaré--Lelong formula}. 
\begin{prop}\label{prop:p-l-formula}
Let notations be as above and let $\chi \in M$ be the weight of $f$, i.e. $\chi \in M$ and \[
g \cdot f = \chi(g) f \quad \forall g \in \T_K.
\]
Then
\[
dd^c\left(-\on{div}_{\nu}(f)\right) = \delta_{\chi[W]-\on{div}(f)},
\]
where $\chi[W]$ denotes the action of $M$ on $\CH_*^{\T_K}(X)$ by homogeneous maps of degree $-1$. 
\end{prop}
\begin{proof}
It suffices to show that for a fixed toric model $\Pi \in R(\Sigma)$ we have that 
\[
\iota^*\iota_*\left(\on{div}_{\Pi}(f)-\overline{\on{div}(f)}^{\Pi}\right) = \iota^*\left(\overline{\chi[W]}^{\Pi}-\overline{\on{div}(f)}^{\Pi}\right),
\]
where $\iota \colon \X_{\Pi,s}\hookrightarrow \X_{\Pi}$ denotes the inclusion of the special fiber. But the difference between the two sides of this equation is 
\[
\iota^*\left(\on{div}_{\Pi}(f)-\overline{\chi[W]}^{\Pi}\right) = \iota^*\left(\on{div}_{\Pi}(f)-\widetilde{\chi}\overline{[W]}^{\Pi}\right),
\]
where $\widetilde{\chi} \in \widetilde{M}$ is the weight of $f$ on $\overline{W}^{\Pi}$. 
And this is zero because $\on{div}_{\Pi}(f)-\widetilde{\chi}\overline{[W]}^{\Pi}$ is zero in $\CH_{\T_S}^k(\X_{\Pi})$ (see Proposition \ref{prop:rel-inv}). 
\end{proof}
We are now ready to define the equivariant arithmetic Chow groups $\widehat{\CH}^k_{\T_K}(X)$.  We write $Z^k_{\T_K}(X)$ for the free $\on{Sym}(M)_{\Q}$-module generated by invariant cycles of codimension $k$ on $X$.

\begin{Def}\label{def:arith-chow1}
We set
\[
\widehat{Z}^k_{\T_K}(X) \coloneqq \left\{\eta,g)\; | \; \eta \in Z^k_{\T_K}(X), g \text{ a Green current for }\eta\right\}
\]
be the group of \emph{equivariant arithmetic cycles} and let $\widehat{R}^k_{\T_K}(X) \subseteq \widehat{Z}^k_{\T_K}(X)$ be the subgroup generated by all arithmetic cycles of the form 
\[
\widehat{\on{div}(f)} \coloneqq \left(\chi[W]-\on{div}(f), -\on{div}_{\nu}(f)\right),
\]
where $f \in K(W)^*$ is any non-zero invariant rational function on an invariant subvariety $W$ of codimension $k-1$ of weight $\chi$. Recall that this means that $f$ is an eigenvector of $\T_K$ of weight $\chi \in M$, i.e. such that 
\[
g \cdot f = \chi(g) f \quad \forall g \in \T_K.
\]

 Then the \emph{equivariant arithmetic Chow group of $X$ of degree $k$} is defined by 
\[
\widehat{\CH}^k_{\T_K}(X) \coloneqq \widehat{Z}^k_{\T_K}(X) / \widehat{R}^k_{\T_K}(X).
\]
We also set 
\[
\widehat{\CH}^*_{\T_K}(X) \coloneqq \bigoplus_{k \in \Z} \widehat{\CH}^k_{\T_K}(X).
\]

\end{Def}
\begin{rem}
One has maps 
\begin{enumerate}
\item[i)] 
\[
\zeta \colon \widehat{\CH}^k_{\T_K}(X) \longrightarrow \CH^k_{\T_K}(X), \; (\eta,g) \longmapsto \eta,
\]
\item[ii)]
\[
a\colon \widetilde{A}^{k-1,k-1}_{\T_K}(X) \longrightarrow \widehat{\CH}^k_{\T_K}(X), \; \alpha \longmapsto (0,\alpha),
\]
\item[iii)]
\[
\omega \colon  \widehat{\CH}^k_{\T_K}(X) \longrightarrow A^{k,k}_{\closed,\T_K}(X), \; (\eta,g) \longmapsto \delta_{\eta} + dd^cg.
\]
\end{enumerate}
The map $\zeta$ is well-defined since $\zeta\left(\widehat{\div(f)}\right) = \chi[W]-\on{div}(f)$, which is zero in $\CH^k_{\T_K}(X)$. The map $\omega$ is well-defined by the Poincaré-Lelong formula in Proposition \ref{prop:p-l-formula}.
\end{rem}
\begin{prop}
The maps above fit into an exact sequence 
\begin{eqnarray*}
\begin{tikzcd}
\widetilde{D}_{\T_K}^{k-1,k-1}(X)  \arrow["a"]{r} &\widehat{\CH}^k_{\T_K}(X)\arrow["\zeta"]{r}& \CH^k_{\T_K}(X) \arrow{r} &0
\end{tikzcd}
\end{eqnarray*}
\end{prop}
\begin{proof}
Surjectivity of $\zeta$ is equivalent to the statement that any equivariant cycle has a Green current, which follows from Proposition \ref{prop:green-fun}. Relations $\chi[W]-\div(f)$ lift to relations $\widehat{\div(f)} = \left(\chi[W]-\div(f), -\div_{\nu}(f)\right)$ in $\widehat{CH}^*_{\T_K}(X)$, so $\on{Im}(a) = \on{Ker}(\zeta)$. 
\end{proof}
\textbf{$\widehat{\CH}^*_{\T_K}(X)$ as a direct limit.}
Let $\eta \in Z_{\T_K}^k(X)$ be an invariant cycle of codimension $k$. Recall from Proposition~\ref{prop:green-fun} that a choice of a toric model $\Pi \in R(\Sigma)$ together with a lifting $\widehat{\eta} \in Z_{\T_S}^k(\X_{\Pi})$ of $\eta$ determines a Green current $g_{\widehat{\eta}}$ for $\eta$. Here, $Z^k_{\T_S}(\X_{\Pi})$ denotes the set of invariant cycles of $\X_{\Pi}$ of absolute $S$-codimension $k$. 

Thus, for any $\Pi \in R(\Sigma)$ we have a commutative diagram

\begin{center}
    \begin{tikzpicture}
      \matrix[dmatrix] (m)
    { Z^k_{\T_S}(\X_{\Pi}) & & \widehat{Z}^k_{\T_K}(X) \\
      & Z_{\T_K}^k(X) &\\
      };
      \draw[->] (m-1-1) to (m-2-2);
    \draw[->] (m-1-1) to node[above]{$\Theta_{\Pi}$}(m-1-3);
     \draw[->] (m-1-3) to (m-2-2);

     \end{tikzpicture}
     \end{center}
where $\Theta_{\Pi}\left(\widehat{\eta}\right) = \left(\widehat{\eta}|_X, g_{\widehat{\eta}}\right)$, the left arrow is given by $\widehat{\eta} \mapsto \widehat{\eta}|_X$ and the right arrow by $\left(\eta,g_{\eta}\right) \mapsto \eta$. 
The following Lemma follows similar to \cite[Lemma~3.3.1]{BGS}. 
\begin{lemma} Let $\Pi \in R(\Sigma)$ and let $\widetilde{W} \subseteq \X_{\Pi}$ be an invariant integral subscheme of codimension $k-1$. Further, let $f \in k(\widetilde{W})^*$ be an invariant rational function on $\widetilde{W}$ of weight $\widetilde{\chi} \in \widetilde{M}$. Then 
\[
\Theta_{\Pi}\left(\widetilde{\chi}\left[\widetilde{W}\right]-\on{div}_{\X_{\Pi}}(f)\right) \in \widehat{R}^k_{\T_K}(X),
\]
where, as before, $\on{div}_{\X_{\Pi}}(f)$ denotes the divisor of $f$ on $\X$. 
\end{lemma}
It follows from the above lemma that $\Theta_{\Pi}$ induces a map 
\[
\Theta_{\Pi} \colon \CH_{\T_S}^k(\X_{\Pi}) \longrightarrow \widehat{\CH}^k_{\T_K}(X).
\]
Moreover, let $\Pi'\geq \Pi$ in $R(\Sigma)$ correspond to a map of toric models $\pi \colon \X_{\Pi'}\to \X_{\Pi}$. Then there is a commutative diagram 
\begin{center}
    \begin{tikzpicture}
      \matrix[dmatrix] (m)
    { \CH^k_{\T_S}(\X_{\Pi}) & & \CH^k_{\T_S}(\X_{\Pi'})  \\
      & \widehat{\CH}^k_{\T_K}(X) &\\
      };
      \draw[->] (m-1-1) to node[left]{$\Theta_{\Pi}$}(m-2-2);
    \draw[->] (m-1-1) to node[above]{$\pi^*$}(m-1-3);
     \draw[->] (m-1-3) to node[right]{$\Theta_{\Pi'}$}(m-2-2);

     \end{tikzpicture}
     \end{center}
   Hence we get a map on the direct limit
   \[
\Theta \colon \varinjlim_{\Pi \in R(\Sigma)}\CH^k_{\T_S}\left(\X_{\Pi}\right) \longrightarrow \widehat{\CH}^k_{\T_K}(X).
\]
\begin{theorem}\label{th:direct}
The map $\Theta$ defined above
is an isomorphism.  
\end{theorem}
\begin{proof}
We have
\[
\varinjlim_{\Pi \in R(\Sigma)}\CH^k_{\T_S}\left(\X_{\Pi}\right) \simeq \varinjlim_{\Pi \in R(\Sigma)}PP^k\left(c(\Pi)\right) = PP^k_{\Sigma}\left(N_{\R} \oplus \R_{\geq 0}\right),
\]
where $PP^k_{\Sigma}\left(N_{\R} \oplus \R_{\geq 0}\right)$ denotes the set of piecewise polynomial functions $f$ on $N_{\R} \oplus \R_{\geq 0}$ of degree $k$ such that there exists $\Pi \in R(\Sigma)$ with $f \in PP^k(c(\Pi))$. 

Then $\Theta$ can be described in combinatorial terms as the map 
\[
\Theta \colon PP^k_{\Sigma}\left(N_{\R} \oplus \R_{\geq 0}\right) \longrightarrow \widehat{\CH}^k_{\T_K}(X)
\]
which sends $f$ to the equivariant arithmetic cycle $\left(f|_{N_{\R} \times \{0\}}, g_f\right)$. Here, $g_f$ is the equivariant Green function from Proposition \ref{prop:green-fun} associated to the invariant cycle induced by $f$ on $X$. 

For the inverse map, let $(\eta, g_{\eta}) \in \widehat{\CH}_{\T_K}^k(X)$, where $\eta$ is an invariant cycle and $g_{\eta}$ is an equivariant Green function for $\eta$. Then $g_{\eta}$ being an equivariant Green function for $\eta$ implies that there exists a toric model $\Pi \in R(\Sigma)$ such that 
\[
\delta_{\eta} + dd^cg_{\eta} \in PP(\Pi),
\]
i.e.~such that the above expression defines an affine piecewise polynomial functions with respect to $\Pi$.  Then the inverse map takes $(\eta,g_{\eta})$ to the piecewise polynomial function $f_{\eta} \in PP^k(c(\Pi))$ corresponding to the Zariski closure $\overline{\eta}^{\Pi}$ of $\eta$ in $\X_{\Pi}$. This is independent on the choice  of $\Pi$.

One checks that the assignments $(\eta, g_{\eta}) \mapsto f_{\eta}$ and $f \mapsto \left(f|_{N_{\R} \times \{0\}}, g_f\right)$ are inverse to each other,. This concludes the theorem.
\end{proof}
\begin{rem}\label{rem:direct-toric}
As was mentioned in the introduction, the above theorem suggests that the non-archimedean analogue of a (smooth) toric metric on a toric vector bundle is the choice of a toric model and a piecewise polynomial function on the (cone over the) polyhedral complex associated to the toric model.  In particular, for $k=1$, by considering the restriction to $N_{\R} \times \{1\}$, we see that the set $PP^1_{\Sigma}\left(N_{\R} \oplus \R_{\geq 0}\right)$ is the same as the set of piecewise affine functions on $N_{\R}$ whose recession function is piecewise linear with respect to $\Sigma$.  Hence, the above isomorphism is saying that the equivariant non-arquimedean analogue of choosing a (smooth) hermitian \emph{toric} metric on a \emph{toric} line bundle associated to a virtual support function $\psi$, is the choice of a piecewise affine function on $N_{\R}$ whose recession function agrees $\psi$.  These are the so called \emph{toric model metrics} from \cite{BPS}.
\end{rem}
\begin{cor}\label{cor:direct}
The graded algebra structure on $\varinjlim_{\X \in R(X)}\CH^*_{\T_S}\left(\X\right) \simeq PP^*_{\Sigma}\left(N_{\R} \oplus \R_{\geq 0}\right)$, which is given by multiplication of piecewise polynomial functions, induces a graded algebra structure on $\widehat{\CH}^*_{\T_K}(X)$. Thus $\Theta$ becomes an isomorphism of graded algebras.

 Moreover, if $X \to Y$ is a toric morphism of smooth projective toric varieties over $K$, we obtain a pullback map
\[
f^* \colon \widehat{\CH}^*_{\T_K}(Y) \longrightarrow \widehat{\CH}^*_{\T_K}(X),
\]
given simply by composition of functions.  

In this way, the association $X \mapsto \widehat{\CH}^*_{\T_K}(X)$ defines a contravariant functor from smooth projective toric varieties to rings of piecewise polynomial functions. 
\end{cor}
\subsection{The extended equivariant arithmetic Chow group $\widecheck{\CH}_{\T_K}^*(X)$}
As before, $X = X_{\Sigma}$ denotes a smooth projective toric variety over $K$. We now define the equivariant analogue of the \emph{extended} arithmetic Chow group $\widecheck{\CH}^k(X)$ defined in \cite{GS-direct}. This allows for non-archimedean analogues of singular metrics. Here one considers, instead of pairs $(\eta, g_{\eta})$ consisting of invariant cycles on $X$ together with Green currents, pairs of the form $(\eta, g)$, where $g$ is an arbitrary current in $\widetilde{D}^{k-1,k-1}(X)$.
\begin{Def}
The \emph{extended equivariant arithmetic Chow group of $X$ of degree $k$} is defined as 
\[
\widecheck{\CH}_{\T_K}^k(X) \coloneqq \left(Z_{\T_K}^k(X) \oplus \widetilde{D}_{\T_K}^{k-1,k-1}(X)\right)/ \widehat{R}_{\T_K}^k(X).
\]
We further  set 
\[
\widecheck{\CH}_{\T_K}^*(X) \coloneqq \bigoplus_{k \in \Z} \widecheck{\CH}_{\T_K}^k(X).
\]
\end{Def}
Clearly, we have 
\[
\widehat{\CH}_{\T_K}^k(X) \subseteq \widecheck{\CH}_{\T_K}^k(X)
\]
for any integer $k$.

The maps $\zeta$, $a$ and $\omega$ given after Definition \ref{def:arith-chow1} extend to maps
\begin{enumerate}
\item[i)] 
\[
\zeta \colon \widecheck{\CH}^k_{\T_K}(X) \longrightarrow \CH^k_{\T_K}(X), \; (\eta,g) \longmapsto \eta,
\]
\item[ii)]
\[
a\colon \widetilde{D}^{k-1,k-1}_{\T_K}(X) \longrightarrow \widecheck{\CH}^k_{\T_K}(X), \; \alpha \longmapsto (0,\alpha),
\]
\item[iii)]
\[
\omega \colon  \widecheck{\CH}^k_{\T_K}(X) \longrightarrow D^{k,k}_{\closed,\T_K}(X), \; (\eta,g) \longmapsto \delta_{\eta} + dd^cg.
\]
\end{enumerate}
Then the regularity result in Proposition \ref{prop:regularity} implies that $\alpha \in \widetilde{A}^{k-1,k-1}_{\T_K}(X)$ if and only if $a(\alpha) \in \widehat{\CH}^k_{\T_K}(X)$. Also, we have that $(\eta,g) \in \widecheck{\CH}^k_{\T_K}(X)$ lies in $\widehat{\CH}^k_{\T_K}(X)$ if and only if $\omega(\eta,g) \in A^{k,k}_{\closed,\T_K}(X)$, i.e.~if and only if $\omega(\eta,g)$ is affine piecewise polynomial with respect to some $\Pi \in R(\Sigma)$.

\textbf{$\widecheck{\CH}^k_{\T_K}(X)$ as an inverse limit.}
We will define a dual map to $\Theta$, expressing $\widecheck{\CH}^k_{\T_K}(X)$ as an \emph{inverse} limit of equivariant Chow groups of toric models of $X$.

Let $\Pi \in R(\Sigma)$. We define a map
\[
\Theta_{\Pi}' \colon Z_{\T_S}^k(\X_{\Pi}) \longrightarrow Z_{\T_K}^k(X) \oplus \CH^{\T_S}_{n-k+1}\left(\X_s\right)
\]
by 
\[
\widetilde{Z}\longmapsto \left(\widetilde{Z}|_X, z_{\Pi}\right),
\]
where $z_{\Pi}$ is given by 
\[
z_{\Pi} = \widetilde{Z} - \overline{\widetilde{Z}|_X}^{\Pi},
\]
where $\overline{\widetilde{Z}|_X}^{\Pi}$ denotes the Zariski closure of $\widetilde{Z}|_X$ in $\X_{\Pi}$. Note that the difference $\widetilde{Z} - \overline{\widetilde{Z}|_X}^{\Pi}$ is supported on the special fiber, so that it indeed defines an element in $\CH_{n-k+1}^{\T_S}(\X_s)$.

As before, we have
 \[
\Theta_{\Pi}'\left(\widetilde{\chi}\left[\widetilde{W}\right]-\on{div}_{\X_{\Pi}}(f)\right) \in \widehat{R}^k_{\T_K}(X),
\]
for any invariant integral subscheme $\widetilde{W} \subseteq \X_{\Pi}$ of codimension $k-1$ and any invariant rational function $f \in k(\widetilde{W})^*$ of weight $\widetilde{\chi} \in \widetilde{M}$. 

Hence, $\Theta_{\Pi}'$ induces a map 
\[
\Theta_{\Pi}' \colon \CH_{\T_S}^k(\X_{\Pi}) \longrightarrow Z_{\T_K}^k(X) \oplus \CH^{\T_S}_{n-k+1}\left(\X_{\Pi,s}\right).
\]
Moreover, if $\Pi' \geq \Pi$, then it follows from the combinatorial description of the pushforward map ${\pi_s}_*$ (see Proposition~\ref{prop:com-forms-currents}),  that ${\pi_s}_*(z_{\Pi'}) = z_{\Pi}$. Hence, the tuple $\left(z_{\Pi'}\right)_{\Pi' \in R(\Sigma)}$ defines an element in $\widetilde{D}_{\T_K}^{k-1,k-1}(X)$. We obtain a commutative diagram 
\begin{center}
    \begin{tikzpicture}
      \matrix[dmatrix] (m)
    { \CH^k_{\T_S}(\X_{\Pi'}) & & \CH^k_{\T_S}(\X_{\Pi})  \\
      & \widecheck{\CH}^k_{\T_K}(X) &\\
      };
      \draw[->] (m-1-1) to node[left]{$\Theta_{\Pi}'$}(m-2-2);
    \draw[->] (m-1-1) to node[above]{$\pi_*$}(m-1-3);
     \draw[->] (m-1-3) to node[right]{$\Theta_{\Pi}'$}(m-2-2);

     \end{tikzpicture}
     \end{center}
and therefore a map on the inverse limit
\[
\Theta' \colon \varprojlim_{\Pi \in R(\Sigma)} \CH^k_{\T_S}(\X_{\Pi}) \longrightarrow \widecheck{\CH}^k_{\T_K}(X).
\]
\begin{theorem}\label{th:inverse}
The map $\Theta'$ defined above is an isomorphism. 
\end{theorem}
\begin{proof}
An element $f \in \varprojlim_{\Pi \in R(\Sigma)} \CH^k_{\T_S}(\X_{\Pi})$ can be seen as a tuple $f = \left(f_{\Pi}\right)_{\Pi \in R(\Sigma)}$ with $f_{\Pi} \in PP^k(c(\Pi))$ and compatible under the pushforward map $\pi_*$. In particular, the restriction $f_{\Pi}|_{N_{\R} \times \{0\}}$ (which coincides with the pushforward map to the canonical model $\X_{\Sigma}$) is the same for any $\Pi \in R(\Sigma)$. We denote it by $f_{\Sigma}$. Then the map $\Theta'$ sends $f$ to 
\[
\left(f_{\Sigma}, z\right),
\]
where $z = \left(z_{\Pi}\right)_{\Pi \in R(\Sigma)} \in \widetilde{D}_{\T_K}^{k-1,k-1}(X)$ is given by 
\[
z_{\Pi} = Z_{f_{\Pi}}-\overline{Z_{f_{\Sigma}}}^{\Pi} \in \CH_{n-k+1}^{\T_S}(\X_{\Pi,s}).
\]
Here, we have denote by $Z_p$ the invariant cycle associated to a piecewise polynomial function $p$ and, as always, $\overline{(\cdot)}^{\Pi}$ denotes the Zariski closure in $\X_{\Pi}$. 

We can now define an inverse map 

\[
 \widecheck{\CH}^k_{\T_K}(X) \longrightarrow \varprojlim_{\Pi \in R(\Sigma)} \CH^k_{\T_S}(\X_{\Pi})
 \]
by 
\[
(\eta,g) \longmapsto \left(h_{\Pi}\right)_{\Pi \in R(\Sigma)},
\]
where $h_{\Pi} \in PP^k(c(\Pi))$ is the piecewise polynomial function associated to the cycle 
\[
\overline{\eta}^{\Pi} +\iota_*(g_{\Pi}).
\]  
Here,  $g = (g_{\Pi})_{\Pi \in R(\Sigma)} \in \widetilde{D}_{\T_K}^{k-1,k-1}(X)$ with $g_{\Pi} \in \CH_{n-k+1}^{\T_S}(\X_{\Pi,s})$ and $\iota \colon \X_{\Pi,s} \hookrightarrow \X_{\Pi}$ is the inclusion of the special fiber. 

One checks that the associations $f \mapsto \left(f|_{\Sigma}, z\right)$ and $(\eta,g) \mapsto \left(h_{\Pi}\right)_{\Pi \in R(\Sigma)}$ are inverse to each other. This concludes the theorem.

\end{proof}
\begin{rem} Note that the restriction of the map $\Theta'$ and its inverse to $\widehat{\CH}^*_{\T_K}(X)$ coincides with the map $\Theta$ and its inverse from the proof of Theorem \ref{th:direct}, respectively. 
\end{rem}
\begin{exa}
\begin{enumerate}
\item Let $g = \left(g_{\Pi}\right)_{\Pi \in R(\Sigma)} \in \widetilde{D}_{\T_K}^{k-1,k-1}(X)$. Then the isomorphism $\Theta'$ sends $\left(\iota_*(\eta_{\X})\right)_{\Pi} \in \varprojlim_{\Pi \in R(\Sigma)}\CH^k_{\T_S}(\X_{\Pi})$ to $a(g) = (0,g)$.
\item If $\eta \in Z^k_{\T_K}(X)$, then the isomorphism $\Theta'$ sends $\left(\overline{\eta}^{\Pi}\right)_{\Pi \in R(\Sigma)}$ to the class of $(\eta, 0)$.
\end{enumerate}
\end{exa}
\begin{rem}\label{rem:proj}
\begin{enumerate}
\item Given the isomorphism $\theta'$, we can view an element in $\widecheck{\CH}_{\T_K}^*(X)$ as a function 
\[
f \colon N_{\Q}\oplus \R \longrightarrow \R,
\]
whose restriction to $N_{\Q}\times \{0\}$ is an element in $PP^k(\Sigma)$.  We would like to characterize all such functions arising in this way. 
\item Consider the case $k=1$. Then, as was mentioned in the introduction, we could interpret the isomorphism $\Theta'$ as saying that the non-archimedean analogue of a (singular) hermitian toric metric on a toric line bundle associated to a virtual support function $\psi$ is the choice of a function $f \colon N_{\Q}\oplus \R \to \R$ whose restriction to $N_{\Q} \times \{0\}$ is $\psi$. This extends the class of toric metrics considered in \cite{BPS} (see Proposition~4.3.10 in \emph{loc.~cit.}). 

\item Note that $\widecheck{\CH}_{\T_K}^*(X)$ is not a ring (the pushforward map $\pi_*$ is only a group homomorphism). It is however an $\widehat{\CH}_{\T_K}^*(X)$-module by considering the product with piecewise polynomial functions. Moreover, it seems possible to define an intersection pairing on a suitable subset of ``positive'' extended arithmetic cycles.  We will pursue these ideas in the future. 
\end{enumerate}
\end{rem}

\printbibliography

@article {EG-char,
    AUTHOR = {Dan, E. and Graham, W.},
     TITLE = {Characteristic classes in the Chow ring},
   JOURNAL = {J. Alg. Geo.},
    VOLUME = {6},
    NUMBER = {3},
      YEAR = {1997},
     PAGES = {431--443},
}

@misc{CLD,
 author = {Chambert-Loir, Antoine and Ducros, Antoine},
 title = {Formes diff{\'e}rentielles r{\'e}elles et courants sur les espaces de {Berkovich}},
 year = {2012},
 howpublished = {Preprint, {arXiv}:1204.6277 [math.{AG}] (2012)},
 url = {https://arxiv.org/abs/1204.6277},
 arXiv = {arXiv:1204.6277}
}

@article{GK,
 author = {Gubler, Walter and K{\"u}nnemann, Klaus},
 title = {A tropical approach to nonarchimedean {Arakelov} geometry},
 fjournal = {Algebra \& Number Theory},
 journal = {Algebra Number Theory},
 issn = {1937-0652},
 volume = {11},
 number = {1},
 pages = {77--180},
 year = {2017},
 language = {English},
 doi = {10.2140/ant.2017.11.77},
 zbMATH = {6679113},
 Zbl = {1386.14096}
}

@Article{BS,
 Author = {Burgos Gil, J.~I. and Sombra, M.},
 Title = {When do the recession cones of a polyhedral complex form a fan?},
 FJournal = {Discrete \& Computational Geometry},
 Journal = {Discrete Comput. Geom.},
 ISSN = {0179-5376},
 Volume = {46},
 Number = {4},
 Pages = {789--798},
 Year = {2011},
 Language = {English},
 DOI = {10.1007/s00454-010-9318-4},
 zbMATH = {5979929},
 Zbl = {1233.14031}
}

@article {BFJ:valuations,
  AUTHOR =	 {Boucksom, S. and Favre, C. and Jonsson, M.},
  TITLE =	 {Valuations and plurisubharmonic singularities},
  JOURNAL =	 {Publ. Res. Inst. Math. Sci.},
  FJOURNAL =	 {Kyoto University. Research Institute for
                  Mathematical Sciences. Publications},
  VOLUME =	 44,
  YEAR =	 2008,
  NUMBER =	 2,
  PAGES =	 {449--494},
}

@article{K,
       author={Kramer, J.},
       title={A geometrical approach to the theory of {J}acobi forms},
       journal={Compos. Math.},
       volume={79},
       number={1},
       year = {1991},
       pages = {1--19},
     }

@article{Brion-equi,
       author={Brion, M.},
       title={Equivariant Chow groups for torus actions},
       journal={Trans. Groups},
       volume={2},
       number={3},
       year = {1997},
       pages = {225--267},
     }

@article{EG,
       author={Edidin, D. and Graham, W.},
       title={Equivariant intersection theory (With an Appendix by Angelo Vistoli: The Chow ring of $\mathcal{M}_2$)},
       journal={Invent. Math.},
       volume={131},
       number={},
       year = {1998},
       pages = {595--634},
     }

@article{Gonz,
       author={Gonzales, R.~P.},
       title={Equivariant operational Chow rings of T-linear schemes},
       journal={Doc. Math.},
       volume={20},
       number={},
       year = {2015},
       pages = {401--432},
     }

@article{Gonz-pol,
       author={Gonzales, R.~P.},
       title={Localization in equivariant operational $K$-theory and the Chang-Skjelbred poperty},
       journal={Manuscripta Math.},
       volume={153},
       number={},
       year = {2017},
       pages = {623--644},
     }

@book{BPS, 
author = {{Burgos Gil}, J.~I. and Philippon, P. and
              Sombra, M.},
    title = {Arithmetic geometry of toric varieties. {M}etrics, measures
              and heights},
   PUBLISHER = {Ast\'erisque},
     NUMBER = {360},
      YEAR = {2014},
     PAGES = {vi+222},
      }

@article{BGS,
       author={Bloch, S. and Gillet, H. and Soul\'e, C.},
       title={Non-archimedean Arakelov theory},
       journal={J. Algebr. Geom.},
       volume={4},
       year = {1995},
       pages = {427--485},
     }

@article{Fa,
       author={Faltings, G.},
       title={Calculus on arithmetic surfaces},
       journal={Ann. of Math.},
       volume={119},
       year = {1987},
       pages = {387--4232},
     }

@article{GS-direct,
       author={Gillet, H. and Soul\'e, C.},
       title={Direct images in non-archimedean Arakelov theory},
       journal={Annales de l'institut Fourier},
       volume={50},
       year = {2000},
       pages = {363--399},
     }

@article{GS1,
       author={Gillet, H. and Soul\'e, C.},
       title={Arithmetic intersection theory},
       journal={Publ. Math. I.~H.~E.~S},
       volume={72},
       year = {1990},
       pages = {94--174},
     }

@article{GS2,
       author={Gillet, H. and Soul\'e, C.},
       title={Characteristic classes for algebraic vector bundles with Hermitian metrics},
       journal={Ann. of Math.},
       number={2},
       volume={131},
       year = {1990},
       pages = {163--238},
     }

@article{W,
       AUTHOR = {W{\l}odarczyk, J.},
     TITLE = {Decomposition of birational toric maps in blow-ups \&
              blow-downs},
   JOURNAL = {Trans. Amer. Math. Soc.},
  FJOURNAL = {Transactions of the American Mathematical Society},
    VOLUME = {349},
      YEAR = {1997},
    NUMBER = {1},
     PAGES = {373--411},
     }

@BOOK{CLS, 

 AUTHOR={Cox, D. and Little, J.~B. and Schenck, H.}, 

 TITLE={{T}oric {V}arieties},
 
 SERIES={Graduate texts in Mathematics}, 

 PUBLISHER={Amer. Math. Soc}, 
 
 VOLUME={124},

 YEAR={2010}, 

 }

@BOOK{fulint, 

 AUTHOR={Fulton, C.}, 

 TITLE={Intersection theory}, 

 PUBLISHER={Springer}, 
 
 YEAR={1998},
 }

@article{Mih,
 author = {Mihatsch, Andreas},
 title = {{{\( \delta \)}}-forms on {Lubin}-{Tate} spaces},
 fjournal = {Duke Mathematical Journal},
 journal = {Duke Math. J.},
 issn = {0012-7094},
 volume = {173},
 number = {14},
 pages = {2809--2928},
 year = {2024},
 language = {English},
 doi = {10.1215/00127094-2023-0069},
 url = {projecteuclid.org/journals/duke-mathematical-journal/volume-173/issue-14/%ce%b4-Forms-on-LubinTate-spaces/10.1215/00127094-2023-0069.full},
 zbMATH = {7940535}
}

@ARTICLE{KT, 

 AUTHOR={Kleiman, S. and (with the cooperation of A. Thorup on Section 3)}, 

 TITLE={Intersection theory and enumerative geometry, a decade in review}, 

 JOURNAL={Proc. Sympos. Pure Math.}, 
 
 VOLUME={46},
 
 NUMBER={},

 YEAR={1987}, 
 
 NOTE = {},

 PAGES={321--370},
 }

@INBOOK{H, 
AUTHOR = {Hain, R.},
     TITLE = {Normal functions and the geometry of moduli spaces of curves},
 BOOKTITLE = {Handbook of moduli. {V}ol. {I}},
    SERIES = {Adv. Lect. Math. (ALM)},
    VOLUME = {24},
    YEAR = {2011},
     PAGES = {527--578},
      }

@ARTICLE{G, 

 AUTHOR={van der Geer, G.}, 
     TITLE = {Corrigendum: ``{T}he {C}how ring of the moduli space of
              abelian threefolds'' [MR1642753]},
   JOURNAL = {J. Algebraic Geom.},
  FJOURNAL = {Journal of Algebraic Geometry},
    VOLUME = {18},
      YEAR = {2009},
    NUMBER = {4},
     PAGES = {795--796},
 }

@ARTICLE{E, 

 AUTHOR = {Elizondo, E.~J.},
     TITLE = {{T}he ring of global sections of multiples of a line bundle on a
              toric variety},
   JOURNAL = {Proc. Amer. Math. Soc.},
     VOLUME = {125},
      YEAR = {1997},
    NUMBER = {9},
     PAGES = {2527--2529},
 }

@ARTICLE{N, 

 AUTHOR={Nystr\"om, D. W.}, 

 TITLE={Transforming metrics on a line bundle to the {O}kounkov body}, 

 JOURNAL={Ann. Sci. Ec. Norm. Sup\'er. (4)}, 
 
 VOLUME={47},

 YEAR={2014}, 

  
 PAGES={1111--1161},
 }

@BOOK{KKMD, 

 AUTHOR={Kempf, G. and Knudsen, K. and Mumford, D. and Saint-Donat, B.}, 

 TITLE={Toroidal {E}mbeddings {I}}, 

 SERIES={Lecture Notes in Math.}, 
 
 PUBLISHER={Springer},
 
  YEAR={1973}, 
 }

@BOOK{L, 

 AUTHOR = {Lazarsfeld, R.},
     TITLE = {{P}ositivity in algebraic geometry. {II}},
    SERIES = {Ergebnisse der Mathematik und ihrer Grenzgebiete. 3. Folge. A
              Series of Modern Surveys in Mathematics [Results in
              Mathematics and Related Areas. 3rd Series. A Series of Modern
              Surveys in Mathematics]},
    VOLUME = {49},
      NOTE = {Positivity for vector bundles, and multiplier ideals},
 PUBLISHER = {Springer-Verlag, Berlin},
      YEAR = {2004},
     PAGES = {xviii+385},
      }

@ARTICLE{BFJ, 
AUTHOR = {Boucksom, S. and Favre, C. and Jonsson,
              M.},
     TITLE = {{S}ingular semipositive metrics in non-{A}rchimedean geometry},
   JOURNAL = {J. Algebraic Geom.},
  FJOURNAL = {Journal of Algebraic Geometry},
    VOLUME = {25},
      YEAR = {2016},
    NUMBER = {1},
     PAGES = {77--139},
      }

@ARTICLE{R, 

 AUTHOR={Russo, F.}, 

 TITLE={{O}n the complement of a nef and big divisor on an algebraic variety}, 

 JOURNAL={Math. Proc. Camb. Phil. Soc.}, 
 
 VOLUME={120},

 YEAR={1996}, 

 
 PAGES={411--422},
 }

@unpublished{BKM,
AUTHOR = {Botero, A. and Kaveh, K. and Manon, C.},
     TITLE = {Equivariant Chern Classes of Toric Vector Bundles over a DVR},
  Note = {Prprint, abailable at \url{https://arxiv.org/abs/2402.18712}},
      Year = {2023},
}

@UNPUBLISHED{lore-thesis,
 AUTHOR={Lore, K.}, 

 TITLE={The specialization index of a variety over a discretely valued field}, 

 NOTE={PhD thesis}, 
 
 }

@unpublished{KMT,
AUTHOR = {Kaveh, K. and Manon, C. and Tsvelikhovski},
     TITLE = {Toric vector bundles over a discrete valuation ring and Bruhat--Tits buildings},
  Note = {Preprint available at \url{https://arxiv.org/abs/2212.03569}},
      Year = {2022},
}

@UNPUBLISHED{bgs-deg, 

 AUTHOR={Bloch, S. and Gillet, H. and Soulé, C.}, 

 TITLE={Algebraic cycles on degenerate fibers}, 

YEAR = {1995},
 NOTE={Preprint, available at \url{https://cds.cern.ch/record/284324?ln=es}, 
 
 }

@article {chern-weil,
    AUTHOR = {Botero, A. and Burgos Gil, J. I. and
              Holmes, D. and de Jong, R.},
     TITLE = {Chern-{W}eil and {H}ilbert-{S}amuel formulae for singular
              {H}ermitian line bundles},
   JOURNAL = {Doc. Math.},
  FJOURNAL = {Documenta Mathematica},
    VOLUME = {27},
      YEAR = {2022},
     PAGES = {2563--2624},
      ISSN = {1431-0635,1431-0643},
   MRCLASS = {14C17 (32U05 32U25)},
  MRNUMBER = {4574244},
       DOI = {10.4171/dm/x36},
       URL = {https://doi.org/10.4171/dm/x36},
}

@article {siegel-jacobi,
    AUTHOR = {Botero, A. and Burgos Gil, J. I. and
              Holmes, D. and de Jong, R.},
     TITLE = {Rings of Siegel--Jacobi forms of bounded relative index are not finitely generated},
   JOURNAL = {Duke Math. J.},
  FJOURNAL = {Duke Mathematical Journal},
    VOLUME = {173},
    NUMBER = {12},
      YEAR = {2024},
     PAGES = {2315--2396},
}

@article {FS,
    AUTHOR = {Fulton, W. and Sturmfels, B.},
     TITLE = {Intersection theory on toric varieties},
   JOURNAL = {	Topology},
      VOLUME = {36},
      YEAR = {1997},
    NUMBER = {2},
     PAGES = {335--353},
}

@article {arakelov1,
    AUTHOR = {Arakelov, S.~J.},
     TITLE = {Intersection theory of divisors on an arithmetic surface},
   JOURNAL = {Math. USSR},
      VOLUME = {8},
      YEAR = {1974},
    NUMBER = {6},
     PAGES = {1167--1180},
}

@article {arakelov2,
    AUTHOR = {Arakelov, S.~J.},
     TITLE = {Theory of intersections on an arithmetic surface},
   JOURNAL = {Amer. Math. Soc.},
      VOLUME = {1},
      YEAR = {1975},
    NUMBER = {},
     PAGES = {405--408},
}

@article{chow-toric,
AUTHOR = {Botero, A.~M.},
     TITLE = {On Chow groups of toric schemes over a DVR},
     JOURNAL= {Manuscr. Math.},
     VOLUME = {175},
     NUMBER = {3--4},
     YEAR = {2024},
}

@article{Payne-equi, 
AUTHOR = {Payne, S.},
     TITLE = {{E}quivariant Chow cohomology of toric varieties},
   JOURNAL = {Math. Res. Lett.},
   VOLUME = {19},
      YEAR = {2006},
    NUMBER = {1},
     PAGES = {29--41},
}

@article{botero,
    AUTHOR = {Botero, A.~M.},
     TITLE = {Intersection theory of {$b$}-divisors in toric varieties},
      JOURNAL = {J. Algebraic Geom.},
    VOLUME = {28},
      YEAR = {2018},
      PAGES = {291--338},
}

@article {botero-conv,
    AUTHOR = {Botero, A.~M.},
     TITLE = {The Convex-Set Algebra and intersection theory on the Toric Riemann--Zariski Space},
   JOURNAL = {Acta Math. Sinica, English Series},
      YEAR = {2021},
   VOLUME = {38},
     PAGES = {465--486}
}

@article{BR,
    AUTHOR = {Brion, M.},
     TITLE = {Piecewise polynomial functions, convex polytopes and enumerative geometry},
      JOURNAL = {Parameter Spaces},
    VOLUME = {36},
      YEAR = {1996},
      PAGES = {},
}

@article{Brion-lin, 
AUTHOR = {Brion, M.},
     TITLE = {Linearization of algebraic group actions},
   JOURNAL = {Handbook of group actions Vol.~IV of Adv. Lect. Math. (ALM)},
   VOLUME = {41},
      YEAR = {2018},
    NUMBER = {},
     PAGES = {291--340},
}

@article{TV, 
AUTHOR = {Tevelev, J. and Vogiannou, T.},
     TITLE = {Spherical tropicalization},
   JOURNAL = {Transf. Groups},
   VOLUME = {26},
      YEAR = {2021},
    NUMBER = {2},
     PAGES = {691--718},
}

@book {EGA,
    AUTHOR = {Grothendieck, A. and Dieudonné, J.},
     TITLE = {Élements de Géometrie algébrique (EGA)},
    SERIES = {Publ. Math.IHES,
Volume = {4,8,11,17,20,24,28,32},
      YEAR = {1960--1967},
}

\end{document}